\documentclass[reqno,11pt]{amsart}

\usepackage[top=2.0cm,bottom=2.0cm,left=3cm,right=3cm]{geometry}
\usepackage{amsthm,amsmath,amssymb,dsfont}
\usepackage{mathrsfs,amsfonts,functan,extarrows,mathtools}
\usepackage[colorlinks]{hyperref}

\usepackage{marginnote}
\usepackage{xcolor}
\usepackage{stmaryrd}
\usepackage{esint}
\usepackage{graphicx}

\newtheorem{theorem}{Theorem}[section]

\newtheorem{remark}{Remark}[section]
\newtheorem{lemma}{Lemma}[section]

\numberwithin{equation}{section}
\allowdisplaybreaks

\def\d{\mathrm{d}}

\def\no{\nonumber}
\def\R{\mathbb{R}}
\def\eps{\epsilon}
\def\div{\mathrm{div}}
\def\u{\mathrm{v}}
\def\w{\mathrm{w}}
\def\dr{\mathrm{d}}
\def\v{\mathrm{u}}

\def\l{\langle}
\def\r{\rangle}

\def\A{\mathrm{A}}
\def\N{\mathrm{N}}
\def\g{\mathrm{g}}
\def\B{\mathrm{B}}
\def\J{\mathcal{J}}

\newcounter{wronumber}

\begin{document}
\title[Parabolic-hyperbolic liquid crystal model]
			{On well-posedness of Ericksen-Leslie's parabolic-hyperbolic liquid crystal model}

\author[N. Jiang]{Ning Jiang}
\address[Ning Jiang]{\newline School of Mathematics and Statistics, Wuhan University, Wuhan, 430072, P. R. China}
\email{njiang@whu.edu.cn}

\author[Y. L. Luo]{Yilong Luo}
\address[Yi-Long Luo]
		{\newline Institute of Mathematics, Academy of Mathematics ans System Science, Chinese Academy of Sciences, Beijing, 100190, P. R. China}
\email{yl-luo@amss.ac.cn}
\thanks{ \today}

\maketitle

\begin{abstract}
 We establish the following well-posedness results on Ericksen-Leslie's parabolic-hyperbolic liquid crystal model: 1, if the dissipation coefficients $\beta = \mu_4 - 4 \mu_6 > 0$, and the size of the initial energy $E^{in}$ is small enough, then the life span of the solution is at least $-O(\ln E^{in})$; 2, for the special case that the coefficients $\mu_1 = \mu_2 = \mu_3 = \mu_5 = \mu_6 = 0$, for which the model is the Navier-Stokes equations coupled with the wave map from $\mathbb{R}^n$ to $\mathbb{S}^2$, the same existence result holds but without the smallness restriction on the size of the initial data; 3, with further constraints on the coefficients, namely $\alpha = \mu_4 - 4 \mu_6 - \tfrac{ (|\lambda_1| - 7 \lambda_2)^2 }{\eta} - \tfrac{ 2 ( 7 |\lambda_1| - 2\lambda_2 )^2 }{ |\lambda_1| } > 0 $ and $\mu_2 < \mu_3$, the global classical solution with small initial data can be established. A relation between the Lagrangian multiplier and the geometric constraint $|\dr|=1$ plays a key role in the proof.
\end{abstract}





\section{Introduction}

The hydrodynamic theory of liquid crystals was established by Ericksen \cite{Ericksen-1961-TSR, Ericksen-1987-RM, Ericksen-1990-ARMA} and Leslie \cite{Leslie-1968-ARMA, Leslie-1979} in the 1960's (see also Section 5.1 of \cite{Lin-Liu-2001} ). The so-called Ericksen-Leslie system consists of the following equations of $(\rho(x,t), \v(x,t), \dr(x,t)$, where $(x,t)\in \mathbb{R}^n\times \mathbb{R}^+$ with $n \geq 2$ :
\begin{align}\label{EL-general}
  \left\{ \begin{array}{c}
    \partial_t \rho + \div(\rho\v) = 0\,,\\
    \rho \dot\v = \rho \mathrm{F} + \div \hat{\sigma}\,,\\
    \rho_1 \dot\omega = \rho_1 \mathrm{G} + \hat{\mathrm{g}} + \div \pi\,.
  \end{array}\right.
\end{align}
The system \eqref{EL-general} represents the conservations laws of mass, linear momentum and angular momentum respectively. Here, $\rho$ is the fluid density, $\rho_1 \geq 0$ is an inertial constant,  $\v = (\v_1,\cdots, \v_n)^\top$ is the flow velocity, $\dr = (\dr_1,\cdots, \dr_n)^\top$ is the direction field of the liquid molecules with the constraint $|\dr|=1$. Furthermore, $\hat{g}$ is the intrinsic force associated with $\dr$, $\pi$ is the director stress, $\mathrm{F}$ and $\mathrm{G}$ are external body force and external director body force, respectively. The superposed dot denotes the material derivative $\partial_t + \v\cdot\nabla$. The following notations
\begin{equation}\nonumber
\begin{aligned}
  &\A = \tfrac{1}{2}(\nabla \v + \nabla^\top\v)\,,\quad \B= \tfrac{1}{2}(\nabla \v - \nabla^\top\v)\,,\\
  &\omega = \dot \dr = \partial_t \dr + (\v\cdot \nabla)\dr\,,\quad \N = \omega - \B \dr\,,
\end{aligned}
\end{equation}
represent the rate of strain tensor, skew-symmetric part of the strain rate, the material derivative of $\dr$ and the rigid rotation part of director changing rate by fluid vorticity, respectively.

The constitutive relations for $\hat{\sigma}$, $\pi$ and $\hat{\g}$ are given by:
\begin{equation}\label{Constitutive}
  \begin{aligned}
    \hat{\sigma}_{ij}& = -p\delta_{ij} + \sigma_{ij} - \rho \tfrac{\partial W}{\partial\dr_{k,i}}\dr_{k,j}\,,\\
    \pi_{ij}& = \beta_i \dr_j + \rho \tfrac{\partial W}{\partial\dr_{j,i}}\,,\\
    \hat{\g}_{ij} & = \gamma \dr_i - \beta_j \dr_{i,j} - \rho \tfrac{\partial W}{\partial \dr_i} + \g_i\,.
  \end{aligned}
\end{equation}
Here $p$ is the pressure, the vector $\beta=(\beta_1\,,\cdots\,,\beta_n)^\top$ and the scalar function $\gamma$ are Lagrangian multipliers for the constraint $|\dr|=1$, and $W$ is the Oseen-Frank energy functional for the equilibrium configuration of a unit director field:
\begin{equation}\label{Oseen-Frank-energy}
  \begin{aligned}
    2W =& k_1 (\div \dr)^2 + k_2 |\dr \cdot (\nabla \times \dr)|^2 + k_3 |\dr \times (\nabla \times \dr)|^2 \\
    & + (k_2 + k_4) \left[ \mathrm{tr} (\nabla \dr)^2 - (\div \dr)^2 \right]\, ,
  \end{aligned}
\end{equation}
where the coefficients $k_1,\ k_2,\ k_3$, and $k_4$ are the measure of viscosity, depending on the material and the temperature.

The kinematic transport $\g$ is given by:
\begin{equation}\label{hat-g}
  \g_i = \lambda_1 \N_i + \lambda_2 \dr_j \A_{ij}
\end{equation}
which represents the effect of the macroscopic flow field on the microscopic structure. The material coefficients $\lambda_1$ and $\lambda_2$ reflects the molecular shape and the slippery part between the fluid and the particles. The first term of \eqref{hat-g} represents the rigid rotation of the molecule, while the second term stands for the stretching of the molecule by the flow.

The stress tensor $\sigma$ has the following form:
\begin{equation}\label{Extra-Sress-sigma}
  \begin{aligned}
    \sigma_{ij}=  \mu_1 \dr_k \A_{kp}\dr_p  \dr_i \dr_j + \mu_2 \N_i \dr_j  + \mu_3 \dr_i \N_j  + \mu_4 \A_{ij} + \mu_5 \A_{ik}\dr_k \dr_j   + \mu_6 \dr_i \A_{jk}\dr_k \,.
  \end{aligned}
\end{equation}
These coefficients $\mu_i (1 \leq i \leq 6)$ which may depend on material and temperature, are usually called Leslie coefficients, and are related to certain local correlations in the fluid. Usually, the coefficients $\mu_4 >0$, $\mu_i \geq 0$ for  $(1 \leq i \leq 6$, $i\neq 4)$. Moreover, the following relations are frequently introduced in the literature.
\begin{equation}\label{Coefficients-Relations}
  \lambda_1=\mu_2-\mu_3\,, \quad\lambda_2 = \mu_5-\mu_6\,,\quad \mu_2+\mu_3 = \mu_6-\mu_5\,.
\end{equation}
The first two relations are necessary conditions in order to satisfy the equation of motion identically, while the third relation is called {\em Parodi's relation}, which is derived from Onsager reciprocal relations expressing the equality of certain relations between flows and forces in thermodynamic systems out of equilibrium. Under Parodi's relation, we see that the dynamics of an incompressible nematic liquid crystal flow involve five independent Leslie coefficients in \eqref{Extra-Sress-sigma}.

For simplicity, in this paper, we assume the external forces vanish, that is, $\mathrm{F}=0$, $\mathrm{G}=0$, and the density is constant, i.e. $\rho=1$, which immediately yields the incompressibility $\div \v=0\,.$ Moreover, we take $k_1 = k_2 = k_3 = 1$, $k_4 = 0$ in \eqref{Oseen-Frank-energy}, then
$$ 2W = |\dr \cdot (\nabla \times \dr)|^2 + |\dr \times (\nabla \times \dr)|^2 + \mathrm{tr} (\nabla \dr)^2\, . $$
Since $|\dr| = 1$, this can be further simplified as $ 2W = |\nabla \dr|^2 $, which implies
\begin{equation}\label{Special-Oseen-Frank-Energy-Derivative}
  \tfrac{\partial W}{\partial \dr_i} = 0,\  \tfrac{\partial W}{\partial(\partial_j \dr_i)} = \partial_j \dr_i\, .
\end{equation}
Taking $\beta_i=0$, then $\pi_{ij}$ reduces to
\begin{equation}\label{d-Dissipative-term}
  \partial_j \pi_{ji} = \partial_j \left( \tfrac{\partial W}{\partial(\partial_j \dr_i)} \right) = \partial_j (\partial_j \dr_i) = \Delta \dr_i\,.
\end{equation}
Thus the third equation of \eqref{EL-general} is
\begin{equation}\label{d-equation-Hyper-Parab}
  \rho_1\ddot{\dr} = \Delta \dr + \gamma \dr + \lambda_1 (\dot{\dr} - \mathrm{B} \dr) + \lambda_2 \A \dr \, .
\end{equation}
Since $|\dr|=1$, it is derived from multiplying $\dr$ in \eqref{d-equation-Hyper-Parab} that
\begin{equation}\label{Lagrange-Multiplier}
  \gamma \equiv \gamma (\v, \dr, \dot{\dr}) = - \rho_1 |\dot{\dr}|^2 + |\nabla \dr|^2 - \lambda_2 \dr^\top \A \dr\, .
\end{equation}
The detailed derivation of \eqref{Lagrange-Multiplier} will be given later with deeper comments on its meanings and implications.

Combining the first equality of \eqref{Constitutive} and \eqref{Special-Oseen-Frank-Energy-Derivative}, one can obtain that
\begin{equation}\no
    \div \hat{\sigma} = - \nabla p - \div (\nabla \dr \odot \nabla \dr)  + \div \sigma\,,
\end{equation}
where $(\nabla \dr \odot \nabla \dr)_{ij} = \sum_k \partial_i \dr_k \partial_j \dr_k\,.$ On the other hand, it can yield that by \eqref{Extra-Sress-sigma}
\begin{equation}
  \no \div \sigma= \tfrac{1}{2} \mu_4 \Delta \v + \div \tilde{\sigma}\,,
\end{equation}
where
\begin{align}
  \no \tilde{\sigma}_{ij} \equiv \big{(}\tilde{\sigma}(\v, \dr, \dot{\dr})\big{)}_{ij}  = & \mu_1 \dr_k \dr_p \A_{kp} \dr_i \dr_j + \mu_2 \dr_j (\dot{\dr}_i + \B_{ki} \dr_k)  \\
  \no &+ \mu_3 \dr_i (\dot{\dr}_j + \B_{kj} \dr_k)    +  \mu_5 \dr_j \dr_k \A_{ki} + \mu_6 \dr_i \dr_k \A_{kj} \,.
\end{align}

Hence,  Ericksen-Leslie's parabolic-hyperbolic liquid crystal model reduces to the following form:
\begin{equation}\label{Parabolic-Hyperbolic-Liquid-Crystal-Model}
  \begin{aligned}
    \left\{ \begin{array}{c}
      \partial_t \v + \v \cdot \nabla \v - \frac{1}{2} \mu_4 \Delta \v + \nabla p = - \div (\nabla \dr \odot \nabla \dr) + \div \tilde{\sigma}\, , \\
      \div \v = 0\, ,\\
     \rho_1 \ddot{\dr} = \Delta \dr + \gamma \dr + \lambda_1 (\dot{\dr}- \B \dr) + \lambda_2 \A \dr\, ,
    \end{array}\right.
  \end{aligned}
\end{equation}
on $ \R^n \times \R^+$ with the constraint $|\dr|=1$, where the Lagrangian multiplier $\gamma$ is given by \eqref{Lagrange-Multiplier}.

In this paper, our main concern is the Cauchy problem of \eqref{Parabolic-Hyperbolic-Liquid-Crystal-Model} with the initial data:
\begin{equation}\label{Inital-Data}
  \v|_{t=0} = \v^{in}(x),\ \dot{\dr}|_{t=0} = {\tilde\dr}^{in}(x),\ \dr|_{t=0} = \dr^{in}(x)\,,
\end{equation}
where $\dr^{in}$ and ${\tilde\dr}^{in}$ satisfy the constraint and compatibility condition:
\begin{equation}\label{Inital-Data-Compatablity}
  |\dr^{in}|=1\,,\quad {\tilde\dr}^{in}\cdot \dr^{in}=0\,.
\end{equation}

 Let us remark that a particularly important special case of the parabolic-hyperbolic system of Ericksen-Leslie's model is that the term $\div \tilde{\sigma}$ vanishes. Namely, the coefficients $\mu_i's$, $(1\leq i \leq 6, i \neq 4)$ of $\div \tilde{\sigma}$ are chosen as 0, which immediately implies $\lambda_1 = \lambda_2 = 0$. Consequently, the system \eqref{Parabolic-Hyperbolic-Liquid-Crystal-Model} reduces to a model which is Navier-Stokes equations coupled with a wave map from $\mathbb{R}^n$ to $\mathbb{S}^2$:
\begin{align}\label{Hyperbolic-Liquid-Crystal-Model}
  \left\{ \begin{array}{c}
    \partial_t \v + \v \cdot \nabla \v + \nabla p = \frac{1}{2}\mu_4 \Delta \v - \div (\nabla \dr \odot \nabla \dr )\, , \\
    \div \v =0\, ,\\
    \rho_1\ddot{\dr} = \Delta \dr + (-\rho_1 |\dot{\dr}|^2+|\nabla \dr|^2) \dr\, .
  \end{array}\right.
\end{align}

\subsection{$\rho_1=0, \lambda_1=-1$, parabolic model}

When the coefficients $\rho_1=0$ and $\lambda_1=-1$ in the third equation of \eqref{Parabolic-Hyperbolic-Liquid-Crystal-Model}, the system reduces to the parabolic type equations, which are also called Ericksen-Leslie's system in the literatures.  The static analogue of the parabolic Ericksen-Leslie's system is the so-called Oseen-Frank model, whose mathematical study was initialed from Hardt-Kinderlehrer-Lin \cite{Hardt-Kinderlehrer-Lin-CMP1986}. Since then there have been many works in this direction. In particular, the existence and regularity or partial regularity of the approximation (usually Ginzburg-Landau approximation as in \cite{Lin-Liu-CPAM1995}) dynamical Ericksen-Leslie's system was started by the work of Lin and Liu in \cite{Lin-Liu-CPAM1995}, \cite{Lin-Liu-DCDS1996} and \cite{Lin-Liu-ARMA2000}.

For the simplest system preserving the basic energy law
\begin{align}\label{simplified-model}
  \left\{ \begin{array}{c}
    \partial_t \v + \v \cdot \nabla \v + \nabla p =  \Delta \v - \div (\nabla \dr \odot \nabla \dr )\, , \\
    \div \v =0\, ,\\
    \partial_t \dr + \v\cdot \nabla \dr = \Delta \dr + |\nabla \dr|^2 \dr\,,\quad |\dr|=1\,,
  \end{array}\right.
\end{align}
which can be obtained by neglecting the Leslie stress and specifying some elastic constants. In 2-D case, global weak solutions with at most a finite number of singular times was proved by Lin-Lin-Wang \cite{Lin-Lin-Wang-ARMA2010}. The uniqueness of weak solutions was later on justified by Lin-Wang \cite{Lin-Wang-CAMS2010} and Xu-Zhang \cite{Xu-Zhang-JDE2012}. Recently, Lin and Wang proved global existence of weak solution for 3-D case in \cite{Lin-Wang-CPAM2016}.

For the more general parabolic Ericksen-Leslie's system,  local well-posedness is proved by Wang-Zhang-Zhang in \cite{Wang-Zhang-Zhang-ARMA2013}, and in \cite{Huang-Lin-Wang-CMP2014} regularity and existence of global solutions in $\mathbb{R}^2$ was established by Huang-Lin-Wang. The existence and uniqueness of weak solutions, also in $\mathbb{R}^2$ was proved by Hong-Xin and Li-Titi-Xin in \cite{Hong-Xin-2012} \cite{Li-Titi-Xin} respectively. Similar result was also obtained by Wang-Wang in \cite{Wang-Wang-2014}. For more complete review of the works for the parabolic Ericksen-Leslie's system, please see the reference listed above.

\subsection{$\rho_1 >0$, parabolic-hyperbolic model}
If $\rho_1>0$,  \eqref{Parabolic-Hyperbolic-Liquid-Crystal-Model} is a parabolic-hyperbolic system for which there is very few works comparing the corresponding parabolic model. The only notable exception might be for the most simplified model, say, in \eqref{Hyperbolic-Liquid-Crystal-Model}, taking $\v=\mathbf{0}$, the spacial dimension is $1$. For this case, the system \eqref{Hyperbolic-Liquid-Crystal-Model} can be reduced to a so-called nonlinear variational wave equation. Zhang and Zheng (later on with Bressan and others) studied systematically the dissipative and energy conservative solutions in series work starting from late 90's \cite{Zhang-Zheng-AA1998, Zhang-Zheng-ActaC1999, Zhang-Zheng-ARMA2000, Zhang-Zheng-CAMS2001, Zhang-Zheng-CPDE2001, Zhang-Zheng-PRSE2002, Zhang-Zheng-ARMA2003, Zhang-Zheng-AIPA2005, Bressan-Zhang-Zheng-ARMA2007, Zhang-Zheng-ARMA2010, Zhang-Zheng-CPAM2012, Chen-Zhang-Zheng-ARMA2013}.

For the multidimensional case, to our best acknowledgement, there was no mathematical work on the original parabolic-hyperbolic Ericksen-Leslie's system \eqref{Parabolic-Hyperbolic-Liquid-Crystal-Model}. Very recently, De Anna and Zarnescu \cite{DeAnna-Zarnescu-2016} considered the inertial Qian-Sheng model of liquid crystals which couples a hyperbolic type equation involving a second order derivative with a forced incompressible Navier-Stokes equations. It is a system describing the hydrodynamics of nematic liquid crystals in the Q-tensor framework. They proved global well-posedness and twist-wave solutions. Furthermore, for the inviscid version of the Qian-Sheng model, in \cite{FRSZ-2016}, Feireisl-Rocca-Schimperna-Zarnescu proved a global existence of the {\em dissipative solution} which is inspired from that of incompressible Euler equation defined by P-L. Lions \cite{Lions-1996}.

It is well-known that the geometric constraint $|\dr|=1$ brings difficulties (particularly in higher order nonlinearities) on the Ericksen-Leslie's system \eqref{Parabolic-Hyperbolic-Liquid-Crystal-Model}, even in the parabolic case $\rho_1=0$. In \cite{Wang-Zhang-Zhang-ARMA2013}, Wang-Zhang-Zhang treated the parabolic case to remove this constraint by introducing a new formulation of parabolic version of \eqref{Parabolic-Hyperbolic-Liquid-Crystal-Model}. The key feature of their formulation is projecting the equation of $\dr$ (the third equation of \eqref{Parabolic-Hyperbolic-Liquid-Crystal-Model} with $\rho_1=0$) the space orthogonal to the direction $\mathrm{d}$. However, for the genuine parabolic-hyperbolic system \eqref{Parabolic-Hyperbolic-Liquid-Crystal-Model} with $\rho_1 > 0$, this technique seems not work.  Indeed, projecting the third equation of the system \eqref{Parabolic-Hyperbolic-Liquid-Crystal-Model} into the {\em orthogonal} direction of $\dr$, we have:
\begin{equation}\label{Othogonal-Direction-Model}
  \begin{aligned}
    \left\{ \begin{array}{c}
      \partial_t \v + \v \cdot \nabla \v - \frac{1}{2} \mu_4 \Delta \v + \nabla p = - \div (\nabla \dr \odot \nabla \dr) + \div \tilde{\sigma}\, , \\
      \div \v = 0\, ,\\
     \rho_1 (\mathrm{I} - \dr \dr) \cdot \ddot{\dr} =  \lambda_1 (\dot{ \dr}- \B \dr) + (\mathrm{I} - \dr \dr) \cdot ( \Delta \dr  + \lambda_2 \A \dr)\, ,
    \end{array}\right.
  \end{aligned}
\end{equation}
where $\mathrm{I}$ denotes the $n \times n$ identical matrix. If taking $\rho_1=0$ and $\lambda_1=-1$, \eqref{Othogonal-Direction-Model} is exactly what was employed in \cite{Wang-Zhang-Zhang-ARMA2013}. But for the case $\rho_1 > 0$, the third equation in \eqref{Othogonal-Direction-Model} include only $(\mathrm{I} - \dr \dr) \cdot \ddot{\dr}$, while the {\em parallel} part of the second derivative term, i.e. $\dr\dr\cdot\ddot{\dr}$ is not included so it will have trouble on the energy estimate. In \cite{Wang-Zhang-Zhang-ARMA2013}, this is not a problem since $\rho_1=0$, the leading order term $\dot{ \dr}- \B \dr$ is automatically orthogonal to $\mathrm{d}$.

Our approach is projecting the third equation of \eqref{Parabolic-Hyperbolic-Liquid-Crystal-Model} to the direction parallel to $\mathrm{d}$:
$$ \rho_1 \dr \dr \cdot \ddot{\dr} = \gamma \dr + \dr \dr \cdot ( \Delta \dr  + \lambda_2 \A \dr)\,, $$
from which we can determine the Lagrangian multiplier $\gamma = - \rho_1 |\dot{\dr}|^2 + |\nabla \dr|^2 - \lambda_2 \dr^\top \A \dr$. Now, the key point is: with $\gamma$ given in this form,  if the initial data $\dr^{in},\ {\tilde\dr}^{in}$ satisfy $|\dr^{in}|=1$ and the compatibility condition $ {\tilde\dr}^{in} \cdot \dr^{in} = 0$, then for the {\em solution} to \eqref{Parabolic-Hyperbolic-Liquid-Crystal-Model}, the constraint $|\dr|=1$ will be {\em forced} to hold. Hence, these constraints need only be given on the initial data, while in the system \eqref{Parabolic-Hyperbolic-Liquid-Crystal-Model}, we do not need the constraint $|\dr| =1$ explicitly any more. These facts will be stated and proved in Section 2.

We remark that spiritually this is similar to \cite{Liu-Liu-Pego-CPAM2007} in which they used an unconstrained formulation of the Navier-Stokes equations, namely, the incompressibility is not required in the equations, but only imposed on the initial data.

\subsection{Main results.}

In this paper, we establish the following three well-posedness results: first, if the viscosity $\mu_4$ is large, namely, $\beta \equiv \mu_4 - 4 \mu_6  > 0$, and the initial energy $E^{in}$ is small enough, then the life span of the Ericksen-Leslie's parabolic-hyperbolic liquid crystal system \eqref{Parabolic-Hyperbolic-Liquid-Crystal-Model}-\eqref{Inital-Data} is at least $-O(\ln E^{in})$. Second, for the special case $\mu_1 = \mu_2 = \mu_3 = \mu_5 = \mu_6 = 0$, same existence and life span of the system \eqref{Hyperbolic-Liquid-Crystal-Model}-\eqref{Inital-Data} hold, but without the smallness restriction on the initial energy. Third, if the viscosity $\mu_4$ is even larger, namely, $\alpha \equiv \mu_4 - 4 \mu_6 - \tfrac{ (|\lambda_1| - 7 \lambda_2)^2 }{\eta} - \tfrac{ 2 ( 7 |\lambda_1| - 2\lambda_2 )^2 }{ |\lambda_1| } > 0 $ and furthermore, a damping condition is satisfied, i.e. $\mu_2 < \mu_3$ (i.e. $\lambda_1 < 0$) on the model \eqref{Parabolic-Hyperbolic-Liquid-Crystal-Model}, where $\eta = \frac{1}{2} \min \big{\{} 1 , \frac{1}{\rho_1} , \frac{ |\lambda_1|}{\rho_1} \big{\}} \in (0,\frac{1}{2} ]$, we can prove the existence of a unique global smooth solution to the parabolic-hyperbolic liquid crystal system \eqref{Parabolic-Hyperbolic-Liquid-Crystal-Model}-\eqref{Inital-Data} with small initial data. More precisely, the main results are stated in the following theorem:

\begin{theorem}\label{Main-Thm}
 Let integer $s > \frac{n}{2}+2 $,  and the initial data $\v^{in},\ {\tilde\dr}^{in} \in H^s(\R^n),\ \nabla \dr^{in} \in H^{s}(\R^n)$, $|\dr^{in}| = 1$, $ {\tilde\dr}^{in} \cdot \dr^{in} = 0$. The initial energy is defined as $E^{in} \equiv |\v^{in}|^2_{H^s} + \rho_1 |{\tilde\dr}^{in}|^2_{H^s} + |\nabla \dr^{in}|^2_{H^s}$. Then the following statements hold:

 (I). If  $\beta \equiv \mu_4 - 4 \mu_6  > 0$, and the initial energy $E^{in}$ is small enough, namely,
 $$ E^{in} < \eps_0 \equiv \min \Big{\{} 1, \tfrac{\rho_1 \beta^2}{ [ 96 C ( \sqrt{\rho_1} ( \mu_1 + \mu_6 ) +  |\lambda_1| - \lambda_2  ) ]^2 }  , \tfrac{\rho^2_1 \beta^4}{ [ 96 C ( \sqrt{\rho_1} ( \mu_1 + \mu_6 ) +  |\lambda_1| - \lambda_2  ) ]^4 }  \Big{\}} \, ,$$
  where $C = C (n,s) > 0$ is a constant which will be determined in Lemma \ref{Unif-Bnd-Lemma}, then, there exists a unique solution $\v \in L^\infty(0,T;H^s(\R^n)) \cap L^2(0,T;H^{s+1}(\R^n))$, $\nabla \dr \in L^\infty(0,T;H^{s}(\R^n))$ and $\dot{\dr} \in L^\infty(0,T;H^s(\R^n))$ to the system \eqref{Parabolic-Hyperbolic-Liquid-Crystal-Model}-\eqref{Inital-Data}, where $ 0 <  T \leq \frac{1}{48 C_1} \ln \frac{1}{ E^{in}  } $.

  Moreover, the solution $(\v, \dr)$ satisfies the energy bound
  $$ \left( |\dr - \dr^{in}|^2_{L^2} + |\v|^2_{H^s} + \rho_1 |\dot{\dr}|^2_{H^2} + |\nabla \dr|^2_{H^s} \right)(t) + \frac{1}{4} \beta \int_0^t |\nabla \v|^2_{H^s} (\tau) \d \tau \leq E^{in} + 12 C_1 T \sqrt{E^{in}} $$
  for all $0 \leq t \leq T$, where $C_1 =   C \Big{[}   1 + \tfrac{ 1 + 2 ( |\lambda_1| - \lambda_2 ) }{\sqrt{\rho_1}} + \frac{1}{(\sqrt{\rho_1})^3} +  \tfrac{(\sqrt{\rho_1}-\lambda_2)^2 + ( |\lambda_1| - \lambda_2 )^2 }{\rho_1 \beta}  \Big{]} > 0$.

   (II). If $\mu_1 = \mu_2 = \mu_3 = \mu_5 = \mu_6 = 0$ and the initial energy $E^{in} < \infty$, then there is a unique solution $\v \in L^\infty(0,T;H^s(\R^n)) \cap L^2(0,T;H^{s+1}(\R^n))$, $\nabla \dr \in L^\infty(0,T;H^{s}(\R^n))$ and $\dot{\dr} \in L^\infty(0,T;H^s(\R^n))$ to the hyperbolic-type system \eqref{Hyperbolic-Liquid-Crystal-Model}-\eqref{Inital-Data}, where $ 0 < T < \frac{1}{4 C_2} \ln \frac{(E^{in} + 1)^2}{ E^{in} ( E^{in} + 2 ) } $.

   Moreover, the solution $(\v, \dr)$ satisfies
  \begin{equation}\no
  \left( |\dr - \dr^{in}|^2_{L^2} + |\v|^2_{H^s} + \rho_1 |\dot{\dr}|^2_{H^s} + |\nabla \dr|^2_{H^s} \right)(t) + \frac{1}{2}\mu_4 \int_0^t |\nabla  \v|^2_{H^s} (\tau) \d \tau \leq \widetilde{C}_2 (E^{in},T)
\end{equation}
  for all $0 \leq t \leq T$, where $C_2 = C ( 1 + \tfrac{1}{\mu_4} + |\nabla \dr^{in}|_{H^s} ) > 0$, $\widetilde{C}_2 (E^{in},T) = E^{in} + 2 C_2 T \prod\limits_{i=0}^2 \Big{[} \big{(} 1 - Y(E^{in}) e^{4 C_2 T}\big{)}^{-\frac{1}{2}} + i - 1 \Big{]} > 0$, $Y(E^{in}) = \tfrac{E^{in} (E^{in} + 2)}{(E^{in} + 1)^2} \in (0,1)$ and $C = C (n,s) > 0$ is determined in Lemma \ref{Unif-Bnd-Lemma}.

   (III). Assume that $\alpha \equiv \mu_4 - 4 \mu_6 - \tfrac{ (|\lambda_1| - 7 \lambda_2)^2 }{\eta} - \tfrac{ 2 ( 7 |\lambda_1| - 2\lambda_2 )^2 }{ |\lambda_1| } > 0 $ and $\mu_2 < \mu_3$ (i.e. $\lambda_1 < 0$), where  $\eta = \frac{1}{2} \min \big{\{} 1 , \frac{1}{\rho_1} , \frac{ |\lambda_1|}{\rho_1} \big{\}} \in (0,\frac{1}{2} ]$. If the initial data satisfy $E^{in}  \leq \eps_1 \equiv \frac{1}{|\lambda_1| + 2} \min \big{ \{ } \frac{1}{2} \eps_0 , \frac{\theta^2}{ (8 C_3)^2 } \big{ \} }$, where $ C_3 = 4 C' \big{(} 1 + \frac{1}{\sqrt{\rho_1}} \big{)}  \big{(}1+\mu_1 + |\lambda_1| - \lambda_2 + \mu_6 - \rho_1 \lambda_2 + \rho_1 + \frac{1}{\sqrt{\rho_1}} \big{)} > 0 $, $\theta = \min \big{\{} \alpha, \eta , \frac{1}{2} |\lambda_1| \big{\}} > 0$ and the constant $C' = C' (n,s) > 0$, then there exists a unique global solution $\v \in L^\infty(0,\infty;H^s(\R^n)) \cap L^2(0,\infty;H^{s+1}(\R^n))$, $\nabla \dr \in L^\infty(0,\infty;H^{s}(\R^n))$ and $\dot{\dr} \in L^\infty(0,\infty;H^s(\R^n))$ to the parabolic-hyperbolic system \eqref{Parabolic-Hyperbolic-Liquid-Crystal-Model}-\eqref{Inital-Data}. Moreover, the solution $(\v, \dr)$ satisfies
  \begin{align}
  \no \sup_{ t \geq 0} \left( |\v|^2_{H^s} + \rho_1 |\dot{\dr}|^2_{H^s} + |\nabla \dr|^2_{H^s} \right) + \int_0^\infty |\nabla \v|^2_{H^s} \d t
 \leq 2(|\lambda_1| + 2) E^{in} \, .
 \end{align}
\end{theorem}

To prove the theorem, one of the key is that the constraint $|\dr|=1$ is not needed so that we can construct the approximate system. Next, we derive the energy estimate for the approximate system from which we use standard compactness method to prove the local existence of the smooth solution which automatically obey $|\dr|=1$. With further damping property and new estimate, we extend the local solution to global in time.

The organization of this paper is as follows: in the next section, we justify the relation between the Lagrangian multiplier $\gamma$ and the constraint $|\dr|=1$. In Section 3, we construct the approximate equation of \eqref{Parabolic-Hyperbolic-Liquid-Crystal-Model}. As mentioned above, in the approximate construction we do not need the constraint $|\dr|=1$. In section 4, we provide an apriori estimate of the approximate system. Then the solution to the approximate equation and their compactness will be proved in Section 5.

\section{Lagrangian multiplier $\gamma$ and constraint $|\dr|=1$}
In this section, we prove the following Lemma on the relation between the Lagrangian multiplier $\gamma$ and the geometric constraint $|\dr|=1$.

\begin{lemma}\label{Parallel-d-Lemma}
  Assume $(\v, \dr)$ is a classical solution to the Ericksen-Leslie's parabolic-hyperbolic system \eqref{Parabolic-Hyperbolic-Liquid-Crystal-Model}-\eqref{Inital-Data} satisfying $\v \in L^\infty(0,T;H^s(\R^n)) \cap L^2(0,T;H^{s+1}(\R^n))$, $\nabla \dr \in L^\infty(0,T;H^{s}(\R^n))$ and $\dot{\dr} \in L^\infty(0,T;H^s(\R^n))$ for some $T \in (0, \infty)$, where $s > \frac{n}{2} + 2$.

  If the constraint $|\dr|=1$ is required, then the Lagrangian multiplier $\gamma$ is
  \begin{equation}\label{Lagrange-Multiplier1}
    \gamma = - \rho_1 |\dot{\dr}|^2 + |\nabla \dr|^2 - \lambda_2 \dr^\top \A \dr \, .
  \end{equation}

  Conversely, if we give the form of $\gamma$ as \eqref{Lagrange-Multiplier1} and $\dr$ satisfies the initial data conditions ${\tilde\dr}^{in} \cdot \dr^{in} = 0$, $|\dr^{in}|=1$ and $|\dr|_{L^\infty([0,T] \times \R^n)} < \infty$, then $ |\dr| = 1 \, . $
\end{lemma}

\begin{proof}
 If $|\dr| = 1$, we multiply $\dr$ in the third equation of the system \eqref{Parabolic-Hyperbolic-Liquid-Crystal-Model} and then we get
  \begin{align}\label{Lagrange-Multiplier-Calculate}
    \no \gamma |\dr|^2 =& \rho_1 \ddot{\dr} \cdot \dr - \Delta \dr \cdot \dr - \lambda_1 \dot{\dr} \cdot \dr + \lambda_1 \dr^\top \B \dr - \lambda_2 \dr^\top \A \dr \\
    \no =& \rho_1 (\partial_t + \v \cdot \nabla) \dot{\dr} \cdot \dr - \div \nabla \dr \cdot \dr - \lambda_1 (\partial_t + \v \cdot \nabla) \dr \cdot \dr - \lambda_2 \dr^\top \A \dr \\
    \no =& \rho_1 (\partial_t + \v \cdot \nabla) (\dot{\dr} \cdot \dr) - \rho_1 \dot{\dr} \cdot (\partial_t + \v \cdot \nabla) \dr - \div (\nabla \dr \cdot \dr) + |\nabla \dr|^2 \\
    \no & - \lambda_1 (\partial_t + \v \cdot \nabla) (\tfrac{1}{2} |\dr|^2) - \lambda_2 \dr^\top \A \dr \\
    \no =& \rho_1 (\partial_t + \v \cdot \nabla)^2 (\tfrac{1}{2} |\dr|^2) - \lambda_1 (\partial_t + \v \cdot \nabla) (\tfrac{1}{2} |\dr|^2) - \Delta (\tfrac{1}{2} |\dr|^2) \\
    \no &- \rho_1 |\dot{\dr}|^2 + |\nabla \dr|^2 - \lambda_2 \dr^\top \A \dr \, .
  \end{align}
  Since $|\dr| =1$, the above equation reduces to
$$ \gamma = - \rho_1 |\dot{\dr}|^2 + |\nabla \dr|^2 - \lambda_2 \dr^\top \A \dr \, . $$

Conversely, if we give $ \gamma = - \rho_1 |\dot{\dr}|^2 + |\nabla \dr|^2 - \lambda_2 \dr^\top \A \dr $, from the above calculation and the initial conditions we have
\begin{equation}\label{|d|-Quasilinear-Hyperbolic-Equation}
\begin{aligned}
  \left\{ \begin{array}{l}
    \rho_1 (\partial_t + \v \cdot \nabla)^2 (  |\dr|^2 - 1 ) - \lambda_1 (\partial_t + \v \cdot \nabla) ( |\dr|^2 - 1 )  - \Delta ( |\dr|^2 - 1 ) = 2 \gamma (|\dr|^2 - 1) \, , \\
     (\partial_t + \v \cdot \nabla) (|\dr|^2-1) \big{|}_{t=0} = 2 {\tilde\dr}^{in} \cdot \dr^{in} = 0 \, ,\\
    (|\dr|^2 - 1 ) \big{|}_{t=0} = |\dr^{in}|^2 - 1 = 0 \, .
  \end{array} \right.
\end{aligned}
\end{equation}

Let $h= |\dr|^2 - 1$. Then $h$ solves the following Cauchy problem for a given smooth $\v$:
\begin{equation}\label{h-Quasilinear-Hyperbolic-Equation}
\begin{aligned}
  \left\{ \begin{array}{l}
    \rho_1 \ddot{h} - \lambda_1 \dot{h}  - \Delta h = 2 \gamma h \, , \\
     \dot{h}|_{t=0} =  0 \, ,\\
    h|_{t=0} = 0 \, .
  \end{array} \right.
\end{aligned}
\end{equation}

For a given smooth divergence-free $\v$ we consider the following flow
\begin{align}
  \no \left\{ \begin{array}{l}
    \partial_t X (t,x) = \v (t, X(t,x)) \, , \\
    X(0,x) = x \, . \\
  \end{array} \right.
\end{align}
Then by substituting it into the equation \eqref{h-Quasilinear-Hyperbolic-Equation} we obtain
\begin{align}
  \no \left\{ \begin{array}{c}
    \rho_1 \partial_t^2 h (t, X(t,x)) - \lambda_1 \partial_t h(t, X(t,x)) - \Delta h(t, X(t,x)) = 2 \gamma (t, X(t,x)) h(t, X(t,x)) \, , \\
    \partial_t h(t, X(t,x)) |_{t=0} = 0 \, ,\\
    h(t, X(t,x)) |_{t=0} = 0 \, . \\
  \end{array}\right.
\end{align}
By letting $\psi (t,x) = h(t, X(t,x))$, it can be rewritten as
\begin{align}\label{psi-hyperbolic}
   \left\{ \begin{array}{c}
    \rho_1 \partial_t^2 \psi - \lambda_1 \partial_t \psi - \Delta \psi = 2 \gamma (t, X(t,x)) \psi \, , \\
    \partial_t \psi |_{t=0} = 0 \, ,\\
    \psi |_{t=0} = 0 \, . \\
  \end{array}\right.
\end{align}

 Our goal is to verify $\psi(t,x) = 0$ for all times $t$. Noticing that $|\dr|_{L^\infty([0,T] \times \R^n)} < \infty$ and $\v \in L^\infty(0,T;H^s(\R^n)) \cap L^2(0,T;H^{s+1}(\R^n))$ and $\dot{\dr}\, , \nabla \dr \in L^\infty(0,T;H^{s}(\R^n))$ for $s > \frac{n}{2} + 2$, we deduce that by Sobolve embedding $$|\gamma|_{L^1(0,T,L^\infty(\R^n))} < \infty \, .$$

  We denote by $Z(t,x) \equiv \lambda_1 \partial_t \psi (t,x) + 2 \gamma (t, X(t,x)) \psi(t,x)$. Multiplying by $\partial_t \psi$ in the equation \eqref{psi-hyperbolic} and integrating by parts over $\R^n$, we have
  \begin{align}
    \no \frac{1}{2} \partial_t ( \rho_1 |\partial_t \psi|_{L^2}^2 + |\nabla \psi|^2_{L^2} ) =&  \left\l Z , \partial_t \psi \right\r \\
    \no \leq& |Z|_{L^2} |\partial_t \psi|_{L^2} \\
    \no \leq& \frac{1}{\sqrt{\rho_1}} |Z|_{L^2} \sqrt{ \rho_1 |\partial_t \psi|_{L^2}^2 + |\nabla \psi|^2_{L^2} }\, ,
  \end{align}
  which implies that
  $$ \partial_t \sqrt{ \rho_1 |\partial_t \psi|_{L^2}^2 + |\nabla \psi|^2_{L^2} } \leq \frac{1}{\sqrt{\rho_1}} |Z|_{L^2}\, . $$
Then by integrating on $[0,t]$ we have
\begin{equation}\label{G-1}
  \sqrt{ \rho_1 |\partial_t \psi(t, \cdot)|_{L^2}^2 + |\nabla \psi (t, \cdot)|^2_{L^2} } \leq \sqrt{ \rho_1 |\partial_t \psi(0, \cdot)|_{L^2}^2 + |\nabla \psi (0, \cdot)|^2_{L^2} } + \frac{1}{\sqrt{\rho_1}} \int_0^t |Z(\tau, \cdot)|_{L^2} \d \tau \, .
\end{equation}

Let $G(t) \equiv \sqrt{ \rho_1 |\partial_t \psi(t, \cdot)|_{L^2}^2 + |\nabla \psi (t, \cdot)|^2_{L^2} }$. One notices that
\begin{align}\label{G-2}
  \no |Z(t, \cdot)|_{L^2} \leq& \Big{(} \tfrac{|\lambda_1|}{\sqrt{\rho_1}} + 2 |\gamma(t, \cdot)|_{L^\infty} \Big{)} \Big{(} \sqrt{\rho_1} |\partial_t \psi(t, \cdot)|_{L^2} + |\psi (t, \cdot)|_{L^2} \Big{)} \\
  \leq& \Big{(} \tfrac{|\lambda_1|}{\sqrt{\rho_1}} + 2 |\gamma(t, \cdot)|_{L^\infty} \Big{)} G(t) \, ,
\end{align}
and
\begin{align}\label{G-3}
  \no |\psi(t, \cdot)|_{L^2} \leq & |\psi(0, \cdot)|_{L^s} + \int_0^t |\partial_t \psi(\tau, \cdot)|_{L^2} \d \tau \\
  \leq& |\psi(0, \cdot)|_{L^2} + \frac{1}{\sqrt{\rho_1}} \int_0^t G(\tau) \d \tau \, .
\end{align}
According  to the inequalities \eqref{G-1}, \eqref{G-2} and \eqref{G-3}, we observe that for all $0 \leq t \leq T$
$$ G(t) \leq G(0) + \int_0^t R(\tau) G(\tau) \d \tau \, , $$
where $R(t) = \frac{1+|\lambda_1|}{\sqrt{\rho_1}} + 2 |\gamma(t, \cdot)|_{L^\infty} \in L^1 ([0,T])$. Then it is derived from Gronwall inequality and the fact $G(0) = 0$ that
$$ 0 \leq G(t) \leq G(0) \exp \left( \int_0^t R(\tau) \d \tau \right) = 0  $$
holds for all $t \in [0,T]$. Consequently, $\psi(t,x) = 0$ holds for all times $t$ and then the proof of Lemma \ref{Parallel-d-Lemma} is finished.

\end{proof}

We remark that if let $\rho_1=0$ in \eqref{|d|-Quasilinear-Hyperbolic-Equation} (in this case, the Lagrangian multiplier $\gamma$ will make the corresponding change), and the initial data $|\dr^{in}|=1$, then we can also prove at later time $|\dr|=1$. This is the case considered in \cite{Wang-Zhang-Zhang-ARMA2013}, although it is not treated in this way there.

\section{Approximate system}

\subsection{The approximate system of \eqref{Parabolic-Hyperbolic-Liquid-Crystal-Model}.} We first construct the approximate system of \eqref{Parabolic-Hyperbolic-Liquid-Crystal-Model}:
\begin{align}\label{Approximated-System}
  \displaystyle \left\{ \begin{array}{l}
    \frac{\d}{\d t} ( \v^\epsilon, \dot{\dr}^\epsilon, \dr^\epsilon )^\top = \mathcal{F}_\epsilon (\v^\epsilon, \dot{\dr}^\epsilon, \dr^\epsilon)\, , \\
    \displaystyle (  \v^\epsilon, \dot{\dr}^\epsilon, \dr^\epsilon )^\top \big{|}_{t=0} = ( \J_\eps \v^{ in}, \J_\eps {\tilde\dr}^{in}, \J_\eps \dr^{ in} )^\top \, .
    \end{array}\right.
\end{align}
where
\begin{equation}\nonumber
\begin{aligned}
    \mathcal{F}_\epsilon(\v^\epsilon, \dot{\dr}^\epsilon, \dr^\epsilon) =& ( F_\epsilon(\v^\epsilon, \dot{\dr}^\epsilon, \dr^\epsilon), \tfrac{1}{\rho_1} G_\epsilon(\v^\epsilon, \dot{\dr}^\epsilon, \dr^\epsilon),H_\epsilon(\v^\epsilon, \dot{\dr}^\epsilon, \dr^\epsilon) )^\top \\
    F_\epsilon(\v^\epsilon, \dot{\dr}^\epsilon, \dr^\epsilon) =&\tfrac{1}{2}\mu_4 \J_\eps \Delta \v^\epsilon -\mathcal{P} \J_\eps [\J_\eps \v^\epsilon \cdot \nabla \J_\eps \v^\epsilon] \\
      &- \mathcal{P} \J_\eps \div ( \nabla \J_\eps \dr^\epsilon \odot  \nabla \J_\eps \dr^\epsilon) + \mathcal{P} \J_\eps \div\tilde{\sigma}( \J_\eps \v^\eps, \J_\eps \dr^\eps, \J_\eps \dot{\dr}^\eps ) \, ,\\
    G_\epsilon(\v^\epsilon, \dot{\dr}^\epsilon, \dr^\epsilon) =& \J_\eps \Delta \dr^\epsilon - \rho_1 \J_\eps [\J_\eps \v^\epsilon \cdot \nabla \J_\eps \dot{\dr}^\epsilon]  + \J_\eps ( \gamma ( \J_\eps \v^\eps, \J_\eps \dr^\eps, \J_\eps \dot{\dr}^\eps ) \J_\eps \dr^\epsilon ) \\
     &+ \lambda_1 (\J_\eps \dot{\dr}^\epsilon - \J_\eps ( \J_\eps \B_\epsilon \J_\eps \dr^\epsilon) ) + \lambda_2 \J_\eps ( \J_\eps \A_\epsilon \J_\eps \dr^\epsilon) \, ,\\
    H_\epsilon(\v^\epsilon, \dot{\dr}^\epsilon, \dr^\epsilon) = &\dot{\dr}^\epsilon - \J_\eps [\J_\eps \v^\epsilon \cdot \nabla \J_\eps \dr^\epsilon]\, .
\end{aligned}
\end{equation}
Here the mollifier operator $\J_\eps$ is defined as $\J_\eps f = \mathcal{F}^{-1} ( {\bf 1}_{|\xi| \leq \frac{1}{\eps}} \mathcal{F} (f) )$ and $\mathcal{F}$ is the standard Fourier transform, $\mathcal{P}$ is the Leray projection, $\B_\eps = \frac{1}{2} (\nabla  \v^\eps - \nabla  \v^\eps{}^\top) $, $\A_\eps = \frac{1}{2} (\nabla  \v^\eps + \nabla \v^\eps{}^\top) $. Furthermore, we still use $\dot{\dr}^\eps$ denote the approximate material derivative of $\dr^\eps$:
\begin{equation}
  \dot{\dr}^\eps = \partial_t \dr^\eps + \J_\eps [\J_\eps \v^\epsilon \cdot \nabla \J_\eps \dr^\eps]\,,
\end{equation}
and $\tilde{\sigma}( \J_\eps \v^\eps, \J_\eps \dr^\eps, \J_\eps \dot{\dr}^\eps )$ the approximate stress tensor:
\begin{equation}
\begin{aligned}
  (\tilde{\sigma}( \J_\eps \v^\eps, \J_\eps \dr^\eps, \J_\eps \dot{\dr}^\eps ) )_{ji} =&  \mu_1 \J_\eps \dr^\eps_k \J_\eps \dr^\eps_p \J_\eps (\A_{kp})_\eps \J_\eps \dr_i^\eps \J_\eps \dr_j^\eps \\
  \no +& \mu_2 \J_\eps \dr_j^\eps ( \J_\eps \dot{\dr}_i^\eps + \J_\eps (\B_{ki})_\eps \J_\eps \dr_k^\eps )  +  \mu_3 \J_\eps \dr_i^\eps ( \J_\eps \dot{\dr}_j^\eps + \J_\eps (\B_{kj})_\eps \J_\eps \dr_k^\eps ) \\
  \no +& \mu_5 \J_\eps \dr_j^\eps \J_\eps \dr_k^\eps \J_\eps (\A_{ki})_\eps + \mu_6 \J_\eps \dr_i^\eps \J_\eps \dr_k^\eps \J_\eps (\A_{kj})_\eps\,,
\end{aligned}
\end{equation}
and $\gamma ( \J_\eps \v^\eps, \J_\eps \dr^\eps, \J_\eps \dot{\dr}^\eps )$ the approximate Lagrangian multiplier:
\begin{equation}\label{approximate-gamma}
\gamma ( \J_\eps \v^\eps, \J_\eps \dr^\eps, \J_\eps \dot{\dr}^\eps ) = -\rho_1| \J_\eps \dot{\dr}^\eps|^2+|\nabla \J_\eps \dr^\eps|^2  - \lambda_2 (\J_\eps \dr^\eps){}^\top \J_\eps \A_\eps (\J_\eps \dr^\eps) \, .
\end{equation}

\subsection{The local existence of the approximate system \eqref{Approximated-System}.} The following standard technical lemma will be frequently used later in this paper.
\begin{lemma}\label{Basic-Properties}
  (1)The following calculus inequality
  $$ |\u \v|_{H^s} \leq |\u|_{L^\infty} |\v|_{H^s} + |\v|_{L^\infty} |\u|_{H^s} $$
  holds for all $\u,\ \v \in H^s \bigcap L^\infty$.

  (2)For any $\v \in H^s$,  $|\J_\eps \v - \v|_{H^s} \rightarrow 0$ as $\epsilon \rightarrow 0$, and the inequality $|\J_\eps \v - \v|_{H^{s-1}} \leq C \epsilon |\v|_{H^s}$ holds for some positive constant $C$.

  (3)For any $\v \in H^s$ and $k \in \mathbb{Z}^+ \bigcup\{0\}$, the inequalities hold:
  $$ |\J_\eps \v|_{H^{s+k}} \leq \frac{C(s,k)}{\epsilon^k} |\v|_{H^s} $$
  and
  $$ |\J_\eps \nabla^k \v|_{L^\infty} \leq \frac{C(k)}{\epsilon^{\frac{n}{2}+k}} |\v|_{L^2} $$
\end{lemma}
\begin{proof}
  These properties are standard and we omit the proof. The details can be found in \cite{Majda-2002-BOOK} for instance.
\end{proof}

\begin{lemma}\label{Local-Existence-Approximated}
  Let $s > \frac{n}{2} + 2$. Then there exists a unique solution $( \v^\epsilon, \dot{\dr}^\epsilon, \dr^\epsilon )^\top \in C([0,T_\epsilon);H^s \times H^s \times H^s)$ to the system \eqref{Approximated-System} for the maximal $T_\epsilon > 0$.
\end{lemma}

\begin{proof}
We need only to verify the Lipschitz continuity of the functions $\mathcal{F}_\epsilon(\cdot,\cdot,\cdot)$ in the product space $H^s \times H^s \times H^s$. In general, we define the norm on a space $X_1 \times X_2 \times X_3$ as $ \left| (  a ,b ,c  )^\top \right|^2_{X_1 \times X_2 \times X_3} \equiv |a|^2_{X_1} + |b|^2_{X_2} + |c|^2_{X_3}\,$.

By Sobolev embedding and Lemma \ref{Basic-Properties}, for any $\v^1, \v^2, \dot{\dr}^1, \dot{\dr}^2, \dr^1, \dr^2 \in H^s$
  \begin{align}
    \no & \left| \J_\eps \div ( \nabla \J_\eps \dr^1 \odot \nabla \J_\eps \dr^1) - \J_\eps \div ( \nabla \J_\eps \dr^2 \odot \nabla \J_\eps \dr^2)  \right|_{H^s}\\
    \no \leq& \left| \Delta \J_\eps (\dr^1-\dr^2) \odot  \nabla \J_\eps \dr^1 \right|_{H^s} + \left|  \nabla \J_\eps (\dr^1-\dr^2) \odot \Delta \J_\eps \dr^1 \right|_{H^s} \\
    \no &+ \left| \Delta \J_\eps \dr^2 \odot \nabla \J_\eps (\dr^1-\dr^2) \right|_{H^s} + \left| \nabla \J_\eps \dr^2 \odot \Delta \J_\eps (\dr^1 - \dr^2) \right|_{H^s} \\
    \no \leq& C \left\{ \left| \nabla \J_\eps \dr^1 \right|_{L^\infty} \left| \J_\eps \Delta (\dr^1 - \dr^2) \right|_{H^s} + \left| \Delta \J_\eps (\dr^1 -\dr^2) \right|_{L^\infty} \left| \nabla \J_\eps \dr^1  \right|_{H^s} \right.\\
    \no &+ \left| \Delta \J_\eps \dr^1 \right|_{L^\infty} \left| \nabla \J_\eps (\dr^1 - \dr^2) \right|_{H^s} + \left| \nabla \J_\eps (\dr^1 -\dr^2) \right|_{L^\infty} \left| \Delta \J_\eps \dr^1  \right|_{H^s}\\
    \no &+ \left| \Delta \J_\eps \dr^2 \right|_{L^\infty} \left| \nabla \J_\eps (\dr^1 - \dr^2) \right|_{H^s} + \left| \nabla \J_\eps (\dr^1 -\dr^2) \right|_{L^\infty} \left| \Delta \J_\eps \dr^2  \right|_{H^s}\\
    \no &+ \left. \left| \nabla \J_\eps \dr^2 \right|_{L^\infty} \left| \J_\eps \Delta (\dr^1 - \dr^2) \right|_{H^s} + \left| \Delta \J_\eps (\dr^1 -\dr^2) \right|_{L^\infty} \left| \nabla \J_\eps \dr^2  \right|_{H^s} \right\}\\
    \no \leq& \frac{C}{\epsilon^{\frac{n}{2}}} \left| \nabla \dr^1 \right|_{L^2} \cdot \frac{1}{\epsilon} \left| \dr^1 - \dr^2 \right|_{\dot{H}^s} + \frac{C}{\epsilon^{s}} \left| \nabla \dr^1 \right|_{L^2} \cdot \frac{1}{\epsilon^{\frac{n}{2}+1}} \left| \nabla(\dr^1-\dr^2) \right|_{L^2} \\
    \no &+ \frac{C}{\epsilon^{\frac{n}{2}}+1} \left| \nabla \dr^1 \right|_{L^2} \cdot \frac{1}{\epsilon} \left| \dr^1 - \dr^2 \right|_{\dot{H}^s} + \frac{C}{\epsilon^{s+1}} \left| \nabla \dr^1 \right|_{L^2} \cdot \frac{1}{\epsilon^{\frac{n}{2}}} \left| \nabla(\dr^1-\dr^2) \right|_{L^2} \\
    \no &+ \frac{C}{\epsilon^{\frac{n}{2}+2}} \left| \nabla \dr^2 \right|_{L^2} \left| \dr^1 - \dr^2 \right|_{\dot{H}^s} + \frac{C}{\epsilon^{\frac{n}{2}+s+1}} \left| \nabla \dr^2 \right|_{L^2} \left|\nabla( \dr^1 - \dr^2 )\right|_{L^2} \\
    \no \leq& \frac{C}{\epsilon^{\frac{n}{2}+s + 2}} \left( \left|\nabla \dr^1 \right|_{L^2} + \left| \nabla \dr^2 \right|_{L^2} \right) \left| \dr^1 - \dr^2 \right|_{H^s} \, ,
  \end{align}
  and
  \begin{align}
    \no & \left| \mu_1 \J_\eps \partial_j [ \J_\eps \dr_k^1 \J_\eps \dr_p^1 \J_\eps (\A^1_{kp})_\epsilon \J_\eps \dr_i^1 \J_\eps \dr_j^1] - \mu_1 \J_\eps \partial_j [ \J_\eps \dr_k^2 \J_\eps \dr_p^2 \J_\eps (\A^2_{kp})_\epsilon \J_\eps \dr_i^2 \J_\eps \dr_j^2] \right|_{H^s} \\
    \no \leq& \mu_1 \left| \J_\eps \partial_j [ \J_\eps (\dr_k^1 - \dr_k^2) \J_\eps \dr_p^1 \J_\eps (\A^1_{kp})_\epsilon \J_\eps \dr_i^1 \J_\eps \dr_j^1 ] \right|_{H^s} \\
    \no +& \mu_1 \left| \J_\eps \partial_j [ \J_\eps \dr_k^1  \J_\eps ( \dr_p^1 - \dr_p^2 ) \J_\eps (\A^1_{kp})_\epsilon \J_\eps \dr_i^1 \J_\eps \dr_j^1 ] \right|_{H^s} \\
    \no +& \mu_1 \left| \J_\eps \partial_j [ \J_\eps \dr_k^1  \J_\eps \dr_p^1 \J_\eps (\A^1_{kp} - \A^2_{kp})_\epsilon \J_\eps \dr_i^1 \J_\eps \dr_j^1 ] \right|_{H^s} \\
    \no +& \mu_1 \left| \J_\eps \partial_j [ \J_\eps \dr_k^1  \J_\eps \dr_p^1 \J_\eps (\A^1_{kp})_\epsilon \J_\eps (\dr_i^1-\dr_i^2) \J_\eps \dr_j^1 ] \right|_{H^s} \\
    \no +& \mu_1 \left| \J_\eps \partial_j [ \J_\eps \dr_k^1  \J_\eps \dr_p^1 \J_\eps (\A^1_{kp})_\epsilon \J_\eps \dr_i^1 \J_\eps (\dr_j^1 -\dr_j^1) ] \right|_{H^s} \\
    \no \equiv& I_1 + I_2 + I_3 + I_4 + I_5\, ,
  \end{align}
  where
  \begin{align}
    \no I_1 \leq& \frac{C(s)\mu_1}{\epsilon} \left| \J_\eps (\dr_k^1 - \dr_k^2) \J_\eps \dr_p^1 \J_\eps (\A^1_{kp})_\epsilon \J_\eps \dr_i^1 \J_\eps \dr_j^1 \right|_{H^s}\\
    \no  \leq& \frac{C(s) \mu_1}{ \epsilon} \left| \J_\eps (\dr_k^1 - \dr_k^2)\right|_{H^s}  \left|\J_\eps \dr_p^1 \J_\eps (\A_{kp})_\epsilon \J_\eps \dr_i^1 \J_\eps \dr_j^1 \right|_{L^\infty} \\
    \no &+ \frac{C(s)\mu_1}{\epsilon} \left| \J_\eps (\dr_k^1 - \dr_k^2) \right|_{L^\infty} \left|\J_\eps \dr_p^1 \J_\eps (\A_{kp})_\epsilon \J_\eps \dr_i^1 \J_\eps \dr_j^1 \right|_{H^s} \\
    \no \leq& \frac{C(s)\mu_1}{\epsilon} \left| \dr^1 - \dr^2 \right|_{H^s} \cdot C \left|\nabla^2 \J_\eps \dr^1\right|^3_{L^2} \left| \J_\eps \nabla \v^1 \right|_{L^\infty} \\
    \no +& \frac{C(s)\mu_1}{\epsilon^{\frac{n}{2} +1}}  \left| \dr^1 - \dr^2 \right|_{L^2} \left( \left| \J_\eps \dr_p^1 \J_\eps \dr_i^1 \J_\eps \dr_j^1 \right|_{L^\infty} \left| \J_\eps (\A^1_{kp})_\epsilon \right|_{H^s} + \left| \J_\eps (\A^1_{kp})_\epsilon \right|_{L^\infty} \left| \J_\eps \dr_p^1 \J_\eps \dr_i^1 \J_\eps \dr_j^1 \right|_{H^s} \right) \\
    \no \leq& \frac{C(s)\mu_1}{\epsilon} \left| \dr^1 - \dr^2 \right|_{H^s}  \left|\nabla^2 \J_\eps \dr^1\right|^3_{L^2} \left| \J_\eps \nabla \v^1 \right|_{L^\infty} \\
    \no +& \frac{C(s)\mu_1}{\epsilon^{\frac{n}{2} +1}} \left| \dr^1 - \dr^2 \right|_{L^2} \left( \left| \nabla^2 \J_\eps \dr^1 \right|^3_{L^2} \left| \nabla \J_\eps \v^1 \right|_{H^s} + \left| \J_\eps \nabla \v^1 \right|_{L^\infty} \left| \J_\eps \dr^1 \right|^2_{L^\infty} \left| \J_\eps \dr^1 \right|_{H^s} \right) \\
    \no \leq& \frac{C(s) \mu_1}{\epsilon^{\frac{n}{2} + 5}} \left| \dr^1 - \dr^2 \right|_{H^s} \left| \nabla \dr^1 \right|_{L^2} \left| \v^1 \right|_{L^2} + \frac{C(s) \mu_1}{ \epsilon^{\frac{n}{2} + s + 5}} \left| \dr^1 - \dr^2 \right|_{L^2} \left| \nabla \dr^1 \right|^3_{L^2} \left| \v^1 \right|_{L^2} \\
    \no &+ \frac{C(s)\mu_1}{\epsilon^{n + s + 3}} \left| \dr^1 - \dr^2 \right|_{L^2} \left| \v^1 \right|_{L^2} \left| \nabla \dr^1 \right|^2_{L^2} \\
    \no \leq& C (s, \mu_1, \tfrac{1}{\epsilon}) \left( \left|\nabla \dr^1 \right|_{L^2} \left| \v^1 \right|_{L^2} + \left| \nabla \dr^1 \right|^3_{L^2} \left| \v^1 \right|_{L^2} + \left| \nabla \dr^1 \right|_{L^2}^2 \left| \v^1 \right|_{L^2} \right) \left| \dr^1 - \dr^2 \right|_{H^s} \, ,
  \end{align}
  and by similar estimating
  \begin{align}
    \no I_2 \leq& \frac{C(s) \mu_1}{\epsilon^{\frac{n}{2}+s + 5}} \left| \nabla \dr^1 \right|^2_{L^2} \left| \nabla \dr^2 \right|_{L^2} \left| \v^1 \right|_{L^2} \left| \dr^1 - \dr^2 \right|_{H^s}\, ,    \\
    \no I_3 \leq& \frac{C(s) \mu_1}{\epsilon^{\frac{n}{2}+s + 5}} \left| \nabla \dr^1 \right|^2_{L^2} \left| \nabla \dr^2 \right|^2_{L^2} \left| \v^1 - \v^2 \right|_{H^s} \, ,   \\
    \no I_4 \leq& \frac{C(s) \mu_1}{\epsilon^{\frac{n}{2}+s + 5}} \left| \nabla \dr^1 \right|_{L^2} \left| \nabla \dr^2 \right|^2_{L^2} \left| \v^2 \right|_{L^2} \left| \dr^1 - \dr^2 \right|_{H^s} \, ,    \\
    \no I_5 \leq& \frac{C(s) \mu_1}{\epsilon^{\frac{n}{2}+s + 5}} \left| \nabla \dr^2 \right|^3_{L^2} \left| \v^2 \right|_{L^2} \left| \dr^1 - \dr^2 \right|_{H^s} \, .
  \end{align}

  Furthermore, similar arguments and Lemma \ref{Basic-Properties} induce to the estimates for all $\v^1$, $\v^2$, $\dr^1$, $\dr^2$, $\dot{\dr}^1$, $\dot{\dr}^2 \in H^s$
  \begin{align}\label{Lipschitz-Bounds-F}
    \no &\left| F_\epsilon (\v^1, \dot{\dr}^1, \dr^1) - F_\epsilon (\v^2, \dot{\dr}^2, \dr^2) \right|_{H^s}\\
     \leq& C(\tfrac{1}{\epsilon}) f(|\v^1|_{L^2} +|\v^2|_{L^2}, |\nabla \dr^1|_{L^2} + |\nabla \dr^2|_{L^2}, |\dot{\dr}^1|_{L^2} + |\dot{\dr}^2|_{L^2} )\\
      \no &\qquad \qquad \times\left( |\v^1- \v^2|_{H^s} + |\dr^1 - \dr^2 |_{H^s} \right)\, ,
  \end{align}
  and
  \begin{equation}\label{Lipschitz-Bounds-G}
    \begin{aligned}
      &\left| G_\epsilon (\v^1, \dot{\dr}^1, \dr^1) - G_\epsilon (\v^2, \dot{\dr}^2, \dr^2) \right|_{H^s} \\
      \leq& C(\tfrac{1}{\epsilon}) g(|\v^1|_{L^2}+|\v^2|_{L^2}, |\dot{\dr}^1|_{L^2}+|\dot{\dr}^2|_{L^2},|\nabla \dr^1|_{L^2} + |\nabla \dr^2|_{L^2}) \\
      & \qquad \times \left( |\v^1- \v^2|_{H^s} + |\dot{\dr}^1 - \dot{\dr}^2|_{H^s} + |\dr^1 - \dr^2 |_{H^s} \right)
    \end{aligned}
  \end{equation}
  for some positive increasing functions $f(\cdot, \cdot, \cdot) $ and $g(\cdot, \cdot, \cdot) $ on their variables, and
  \begin{align}\label{Lipschitz-Bounds-H}
    \no &\left| H_\epsilon (\v^1, \dot{\dr}^1, \dr^1) - H_\epsilon (\v^2, \dot{\dr}^2, \dr^2) \right|_{H^s} \\
    \leq& \left|\dot{\dr}^1 \!-\! \dot{\dr}^2\right|_{H^s}\! +\! \frac{C}{\epsilon^{\frac{n}{2}+s+1}}\! \left( |\v^1|_{L^2}\! +\!|\v^2|_{L^2}\!+\! |\nabla \dr^1|_{L^2}\! +\! |\nabla \dr^2|_{L^2} \right)\! \left( |\v^1\!-\! \v^2|_{H^s}\! +\! |\dr^1 \!-\! \dr^2 |_{H^s} \right) \, .
  \end{align}

  Then combining the estimates \eqref{Lipschitz-Bounds-F}, \eqref{Lipschitz-Bounds-G} and \eqref{Lipschitz-Bounds-H} implies that for all $\v^i,\dot{\dr}^i, \dr^i \in H^s$, $i=1,2$
  \begin{equation}\label{Lipschitz-Bounds}
    \begin{aligned}
      & \left| \mathcal{F}_\epsilon(\v^1, \dot{\dr}^1, \dr^1) - \mathcal{F}_\epsilon(\v^2, \dot{\dr}^2, \dr^2) \right|_{H^s \times H^s \times H^s} \\
      \leq& C(\tfrac{1}{\epsilon}) \mathcal{I}(|\v^1|_{L^2}+|\v^2|_{L^2}, |\dot{\dr}^1|_{L^2}+|\dot{\dr}^2|_{L^2},|\nabla \dr^1|_{L^2} + |\nabla \dr^2|_{L^2}) \\
      & \qquad \times \left( |\v^1- \v^2|_{H^s} + |\dot{\dr}^1 - \dot{\dr}^2|_{H^s} + |\dr^1 - \dr^2 |_{H^s} \right)\, ,
    \end{aligned}
  \end{equation}
  where $\mathcal{I}(\cdot,\cdot,\cdot)$ is a positive increasing function on its variables. Hence $\mathcal{F}_\epsilon$ is locally Lipschitz on $ H^s \times H^s \times H^s$. Then the ODE theory implies that there exists a unique solution $( \v^\epsilon, \dot{\dr}^\epsilon, \dr^\epsilon )^\top \in C([0,T_\epsilon);H^s \times H^s \times H^s ) $ to the system \eqref{Approximated-System} on the maximal interval $[0,T_\eps)$ and then we finish the proof of Lemma \ref{Local-Existence-Approximated} .
\end{proof}

We emphasize that this proposition holds for general case $\mu_4 > 0,\ \mu_1, \mu_2, \mu_3, \mu_5, \mu_6 \geq 0$.

\begin{remark}
  As $\J_\eps^2 = \J_\eps\,$, we know that $(\J_\eps \v^\eps, \J_\eps \dr^\eps)$ is also a solution to the system \eqref{Approximated-System}. Then $(\v^\eps, \dr^\eps) = ( \J_\eps \v^\eps, \J_\eps \dr^\eps )$ and $\dot{\dr}^\eps = \J_\eps \dot{\dr}^\eps $. As a consequence, $(\v^\eps, \dr^\eps)$ solves the following system:
  \begin{equation}\label{Appr-Syst-Simple}
    \left\{ \begin{array}{c}
      \partial_t \v^\eps = - \J_\eps \mathcal{P} (\v^\eps \cdot \nabla \v^\eps) + \frac{1}{2} \mu_4 \Delta \v^\eps -\J_\eps \mathcal{P} \div (\nabla \dr^\eps \odot \nabla \dr^\eps) + \J_\eps \mathcal{P} \div \tilde{\sigma} (\v^\eps, \dr^\eps, \dot{\dr}^\eps) \, ,\\
      \div \v^\eps = 0\, ,\\
      \rho_1\partial_t \dot{\dr}^\eps = \!- \rho_1 \J_\eps (\v^\eps \! \cdot\! \nabla \dot{\dr}^\eps) \!+\! \Delta \dr^\eps\! +\! \J_\eps ( \gamma (\v^\eps, \dr^\eps, \dot{\dr}^\eps) \dr^\eps )\! +\! \lambda_1 ( \dot{\dr}^\eps \!-\! \J_\eps (\B_\eps \dr^\eps) ) \!+\! \lambda_2 \J_\eps (\A_\eps \dr^\eps) \\
      (  \v^\epsilon, \dot{\dr}^\epsilon, \dr^\epsilon )^\top \big{|}_{t=0} = ( \J_\eps \v^{ in}, \J_\eps {\tilde\dr}^{in}, \J_\eps \dr^{ in} )^\top \, ,
    \end{array}\right.
  \end{equation}
  where $\dot{\dr}^\eps = \partial_t \dr^\eps + \J_\eps (\v^\eps \cdot \nabla \dr^\eps)$, $ \gamma ( \v^\eps, \dr^\eps, \dot{\dr}^\eps ) = -\rho_1| \dot{\dr}^\eps|^2+|\nabla \dr^\eps|^2  - \lambda_2 ( \dr^\eps){}^\top \A_\eps ( \dr^\eps)\, , $
  and
  \begin{equation}\no
    (\tilde{\sigma}( \v^\eps, \dr^\eps, \dot{\dr}^\eps ) )_{ji} =  \mu_1 \dr^\eps_k \dr^\eps_p (\A_{kp})_\eps  \dr_i^\eps  \dr_j^\eps  + \mu_2  \dr_j^\eps (  \dot{\dr}_i^\eps +  (\B_{ki})_\eps  \dr_k^\eps )  + \mu_5  \dr_j^\eps  \dr_k^\eps  (\A_{ki})_\eps + \mu_6  \dr_i^\eps  \dr_k^\eps  (\A_{kj})_\eps\, .
  \end{equation}
\end{remark}

\section{Uniform Energy Estimate.}
In this section, we want to obtain apriori energy estimate of the system \eqref{Appr-Syst-Simple}. For all $s > \frac{n}{2} + 2$ we define energy functionals $ E_\eps(t) \equiv |\dr^\eps - \J_\eps \dr^{in}|^2_{L^2} + |\v^\eps|^2_{H^s} + \rho_1 |\dot{\dr}^\eps|^2_{H^s} + |\nabla \dr^\eps|^2_{H^s}\, , $ and $F_\eps(t) \equiv |\nabla \v^\eps|^2_{H^s} \, . $

For any integer $s \geq 1$ we define the homogeneous Sobolev spaces $\dot{H}^s(\R^n)$:
$$ f \in \dot{H}^s(\R^n) \Leftrightarrow |f|^2_{\dot{H}^s(\R^n)} \equiv \sum\limits_{1 \leq k \leq s} |\nabla^k f|^2_{L^2(\R^n)} < + \infty\, .$$
If $f \in \dot{H}^s(\R^n) \cap H^s(\R^n)$, then $|f|^2_{H^s(\R^n)} = |f|^2_{\dot{H}^s(\R^n)} + |f|^2_{L^2(\R^n)}\,,$ where $H^s(\R^n)$ is the inhomogeneous Sobolev spaces. Then we have the following lemma:
\begin{lemma}\label{Unif-Bnd-Lemma}
  Assume that $\beta = \mu_4 - 4 \mu_6 > 0$ and $(\v^\eps, \dr^\eps)$ is a sufficiently smooth solution to the approximate system \eqref{Appr-Syst-Simple}. Then there exists a constant  $C = C(n,s) > 0$ such that for all $t \in [0, T_\eps)$
  \begin{align}\label{Bounds-2}
   \frac{1}{2} \frac{\d}{\d t} E_\eps (t) + \frac{1}{4} \beta F_\eps(t) \leq {\bf P} (E_\eps (t)) + {\bf Q} (E_\eps(t)) F_\eps(t)\, ,
\end{align}
where
\begin{align}
   \no {\bf Q} (E_\eps(t)) = C ( \mu_1 + \mu_6 + \tfrac{ |\lambda_1| - \lambda_2 }{\sqrt{\rho_1}} ) ( 1 + |\nabla \dr^{in}|^2_{H^s} ) \Big{(} |\nabla \dr^{in}|_{H^s} + \sum_{i=1}^5 E_\eps^\frac{i}{2} (t) \Big{)}\, ,
\end{align}
\begin{align}
  \no {\bf P} (E_\eps(t)) = C_0 (\lambda_1, \lambda_2, \beta, \rho_1, |\nabla \dr^{in}|_{H^s}) E_\eps(t) [ E_\eps(t) + 1 ] [ E_\eps(t) + 2 ] \, ,
\end{align}
and the constant $C_0 (\lambda_1, \lambda_2, \beta, \rho_1, |\nabla \dr^{in}|_{H^s}) $ is
\begin{align}
  \no C_0 (\lambda_1, \lambda_2, \beta, \rho_1, |\nabla \dr^{in}|_{H^s}) =& C \Big{[} \tfrac{1 + |\lambda_1| - \lambda_2 + ( |\lambda_1| - \lambda_2 ) |\nabla \dr^{in}|^2_{H^s} } {\sqrt{\rho_1}} + |\nabla \dr^{in}|_{H^s} \\
   \no & + \tfrac{1}{(\sqrt{\rho_1})^{3}} + \tfrac{(\sqrt{\rho_1}-\lambda_2)^2 + ( |\lambda_1| - \lambda_2 )^2}{\rho_1 \beta}  \Big{]}\, .
\end{align}
\end{lemma}

\begin{remark}\label{Remark-mu=0}
  If $\mu_1 = \mu_2 = \mu_3 = \mu_5 = \mu_6 = 0$, we have ${\bf Q}(E_\eps(t)) = 0$ and $\beta = \mu_4$. Thus
  \begin{align}\label{Approximated-Apriori-Estimates-Special}
   \frac{1}{2} \frac{\d}{\d t} E_\eps (t) + \frac{1}{4} \mu_4 F_\eps(t) \leq {\bf \widetilde{P}} (E_\eps (t)) \, ,
\end{align}
where
\begin{align}
  \no {\bf \widetilde{P}} (E_\eps(t)) = C \big{(} \tfrac{1}{\sqrt{\rho_1}} + \tfrac{1}{\mu_4} + |\nabla \dr^{in}|_{H^s} \big{)} E_\eps(t) [ E_\eps(t) + 1 ] [ E_\eps(t) + 2 ]\, .
\end{align}
\end{remark}

\begin{proof}
   \underline{$L^2$-Estimates.} By the relations $\partial_t \dr^\eps = \dot{\dr}^\eps - \J_\eps (\v^\eps \cdot \nabla \dr^\eps)$ and $\div \v^\eps = 0$ we know that
   \begin{align}\label{L2-Estimate-1}
     \no \frac{\d}{\d t} |\dr^\eps - \J_\eps \dr^{in}|^2_{L^2} =& 2 \left\l \partial_t \dr^\eps , \dr^\eps - \J_\eps \dr^{in} \right\r \\
     \no =& 2 \left\l \dot{\dr}^\eps - \J_\eps ( \v^\eps \cdot \nabla \dr^\eps ) , \dr^\eps - \J_\eps \dr^{in} \right\r \\
     =&  2 \left\l \dot{\dr}^\eps , \dr^\eps - \J_\eps \dr^{in} \right\r  - 2 \left\l \v^\eps \cdot \J_\eps \dr^{in} , \J_\eps (\dr^\eps - \J_\eps \dr^{in}) \right\r \\
     \no \leq& 2 |\dot{\dr}^\eps|_{L^2} |\dr^\eps - \J_\eps \dr^{in}|_{L^2} + 2 |\nabla \dr^{in}|_{L^\infty} |\v^\eps|_{L^2} |\dr^\eps - \J_\eps \dr^{in}|_{L^2} \\
     \no =& 2 ( |\dot{\dr}^\eps|_{L^2} + |\nabla \dr^{in}|_{L^\infty} |\v^\eps|_{L^2} ) |\dr^\eps - \J_\eps \dr^{in}|_{L^2} \\
     \no \leq& 2 C ( \tfrac{1}{\sqrt{\rho_1}} + |\nabla \dr^{in}|_{H^s} ) E_\eps(t) \, .
   \end{align}

     Multiplying $\dot{\dr}^\eps$ in the third equation of \eqref{Appr-Syst-Simple} and integrating on $\R^n$ by parts, one may obtain by the facts $( \J_\eps \v^\eps, \J_\eps \dr^\eps, \J_\eps \dot{\dr}^\eps ) = ( \v^\eps, \dr^\eps, \dot{\dr}^\eps)$ and $\div \v^\eps = 0$
     \begin{align}\label{L2-Estimate-2}
       \no &\frac{1}{2} \frac{\d}{\d t} \left( \rho_1 |\dot{\dr}^\eps|^2_{L^2} + |\nabla \dr^\eps|^2_{L^2} \right) - \lambda_1 |\dot{\dr}^\eps|^2_{L^2} \\
       \no =& \left\l \Delta \dr^\eps, \v^\eps \cdot \nabla \dr^\eps \right\r - \left\l \rho_1 |\dot{\dr}^\eps|^2 \dr^\eps , \dot{\dr}^\eps \right\r + \left\l |\nabla \dr^\eps|^2 \dr^\eps , \dot{\dr}^\eps \right\r \\
       \no -& \lambda_2 \left\l \dr^\eps{}^\top \A_\eps \dr^\eps, \dr^\eps \cdot \dot{\dr}^\eps \right\r - \lambda_1 \left\l \B_\eps \dr^\eps , \dot{\dr}^\eps \right\r + \lambda_2 \left\l \A_\eps \dr^\eps, \dot{\dr}^\eps \right\r \\
        \leq& \left\l \Delta \dr^\eps, \v^\eps \cdot \nabla \dr^\eps \right\r + |\dr^\eps|_{L^\infty} |\dot{\dr}^\eps|_{L^\infty} ( \rho_1 |\dot{\dr}^\eps|^2_{L^2} + |\nabla \dr^\eps|^2_{L^2} ) \\
       \no  +& |\lambda_2| |\dr^\eps|^3_{L^\infty} |\nabla \v^\eps|_{L^2} |\dot{\dr}^\eps|_{L^2} + ( |\lambda_1| + |\lambda_2| ) |\dr^\eps|_{L^\infty} |\nabla \v^\eps|_{L^2} |\dot{\dr}^\eps|_{L^2} \\
       \no \leq& \left\l \Delta \dr^\eps, \v^\eps \cdot \nabla \dr^\eps \right\r + |\dr^\eps|_{L^\infty} |\dot{\dr}^\eps|_{L^\infty} ( \rho_1 |\dot{\dr}^\eps|^2_{L^2} + |\nabla \dr^\eps|^2_{L^2} ) \\
       \no +& ( |\lambda_1| + |\lambda_2| ) ( |\dr^\eps|_{L^\infty} + |\dr^\eps|^3_{L^\infty} ) |\nabla \v^\eps|_{L^2} |\dot{\dr}^\eps|_{L^2} \, .
     \end{align}

     Multiplying by $\v^\eps$ in the first equation of \eqref{Appr-Syst-Simple} and integrating on $\R^n$ by parts, one can derive by the facts $( \J_\eps \v^\eps, \J_\eps \dr^\eps, \J_\eps \dot{\dr}^\eps ) = ( \v^\eps, \dr^\eps, \dot{\dr}^\eps)$ and $\div \v^\eps = 0$
     \begin{align}\label{L2-Estimate-3}
       \no & \frac{1}{2} \frac{\d}{\d t} |\v^\eps|^2_{L^2} + \frac{1}{2} \mu_4 |\nabla \v^\eps|^2_{L^2} \\
       \no =& - \left\l \partial_j ( \partial_i \dr^\eps_k \partial_j \dr^\eps_k ), \v^\eps_i \right\r - \left\l \tilde{\sigma} (\v^\eps, \dr^\eps, \dot{\dr}^\eps), \nabla \v^\eps \right\r \\
       =& - \left\l \partial_j  \partial_i \dr^\eps_k \partial_j \dr^\eps_k , \v^\eps_i \right\r  - \left\l  \partial_i \dr^\eps_k \partial_j \partial_j \dr^\eps_k , \v^\eps_i \right\r - \left\l \tilde{\sigma} (\v^\eps, \dr^\eps, \dot{\dr}^\eps), \nabla \v^\eps \right\r \\
       \no =& - \left\l \div ( \tfrac{1}{2} |\nabla \dr^\eps|^2 ), \v^\eps \right\r - \left\l \Delta \dr^\eps, \v^\eps \cdot \nabla \dr^\eps \right\r - \left\l \tilde{\sigma} (\v^\eps, \dr^\eps, \dot{\dr}^\eps), \nabla \v^\eps \right\r \\
       \no =& - \left\l \Delta \dr^\eps, \v^\eps \cdot \nabla \dr^\eps \right\r - \left\l \tilde{\sigma} (\v^\eps, \dr^\eps, \dot{\dr}^\eps), \nabla \v^\eps \right\r \, .
     \end{align}
     Now we estimate the term $- \left\l \tilde{\sigma} (\v^\eps, \dr^\eps, \dot{\dr}^\eps), \nabla \v^\eps \right\r $. It is calculated by integrating by parts, H\"older inequality and the fact $\dr^\eps{}^\top \B_\eps \dr^\eps =0 $
     \begin{equation}
       \no - \mu_1 \left\l (\dr^\eps{}^\top \A_\eps \dr^\eps) \dr^\eps \dr^\eps, \nabla \v^\eps \right\r = - \mu_1 |\dr^\eps{}^\top \nabla \v^\eps \dr^\eps|^2_{L^2} \, ,
     \end{equation}
     and
     \begin{align}
       \no  -& \mu_2 \left\l \dr^\eps_j ( \dot{\dr}^\eps_i - (\B_\eps)_{ki} \dr^\eps_k ), \partial_j \v^\eps_i \right\r - \mu_3 \left\l \dr^\eps_i ( \dot{\dr}^\eps_j - (\B_\eps)_{kj} \dr^\eps_k ), \partial_j \v^\eps_i \right\r \\
       \no \leq& (\mu_2 + \mu_3) |\dr^\eps|_{L^\infty} |\dot{\dr}^\eps|_{L^2} |\nabla \v^\eps|_{L^2} + (\mu_2 + \mu_3) |\dr^\eps|^2_{L^\infty} |\nabla \v^\eps|^2_{L^2} \, ,
     \end{align}
     and
     \begin{equation}
       \no - \mu_5 \left\l \dr^\eps_j \dr^\eps_k (\A_\eps)_{ki}, \partial_j \v^\eps_i \right\r - \mu_6 \left\l \dr^\eps_i \dr^\eps_k (\A_\eps)_{kj}, \partial_j \v^\eps_i \right\r \leq (\mu_5 + \mu_6) |\dr^\eps|^2_{L^\infty} |\nabla \v^\eps|^2_{L^2} \, .
     \end{equation}
     Hence the estimate of the term $- \left\l \tilde{\sigma} (\v^\eps, \dr^\eps, \dot{\dr}^\eps), \nabla \v^\eps \right\r $ is
     \begin{align}\label{L2-Estimate-4}
         -& \left\l \tilde{\sigma} (\v^\eps, \dr^\eps, \dot{\dr}^\eps), \nabla \v^\eps \right\r \\
       \no \leq & - \mu_1 |\dr^\eps{}^\top \A_\eps \dr^\eps|^2_{L^2} + (\mu_2 + \mu_3) |\dr^\eps|_{L^\infty} |\dot{\dr}^\eps|_{L^2} |\nabla \v^\eps|_{L^2} + (\mu_2 + \mu_3 + \mu_5 + \mu_6) |\dr^\eps|^2_{L^\infty} |\nabla \v^\eps|^2_{L^2} \, .
     \end{align}

     Then the inequalities \eqref{L2-Estimate-2}, \eqref{L2-Estimate-3} and \eqref{L2-Estimate-4} give the following basic $L^2$-estimate:
     \begin{align}\label{L2-Estimate-5}
       \no &\frac{1}{2} \frac{\d}{\d t} \left( |\v^\eps|^2_{L^2} + \rho_1 |\dot{\dr}^\eps|^2_{L^2} + |\nabla \dr^\eps|^2_{L^2} \right) + \frac{1}{2} \mu_4 |\nabla \v^\eps|^2_{L^2} - \lambda_1 |\dot{\dr}^\eps|^2_{L^2} + \mu_1 |\dr^\eps{}^\top (\nabla \v^\eps) \dr^\eps|^2_{L^2} \\
       \no \leq& |\dr^\eps|_{L^\infty} |\dot{\dr}^\eps|_{L^\infty} ( \rho_1 |\dot{\dr}^\eps|^2_{L^2} + |\nabla \dr^\eps|^2_{L^2} ) + ( |\lambda_1| + |\lambda_2| + \mu_2 + \mu_3 ) ( |\dr^\eps|_{L^\infty} + |\dr^\eps|^3_{L^\infty} ) |\nabla \v^\eps|_{L^2} |\dot{\dr}^\eps|_{L^2} \\
       & + (\mu_2 + \mu_3 + \mu_5 + \mu_6) |\dr^\eps|^2_{L^\infty} |\nabla \v^\eps|^2_{L^2} \, .
     \end{align}

     \underline{$H^s$-Estimates.} For any integer $1 \leq k \leq s$, we take $\nabla^k$ in the first equation of \eqref{Appr-Syst-Simple}, multiply by $\nabla^k \v^\eps$ and integrate by parts, then we have
     \begin{equation}\label{v-H^k-Estimate-1}
      \begin{aligned}
        & \frac{1}{2} \frac{\d}{\d t} |\nabla^k \v^\eps|^2_{L^2} + \frac{1}{2}\mu_4 |\nabla^{k+1} \v^\eps|^2_{L^2} \\
        =&  -\left\l \nabla^k (\v^\eps \cdot \nabla \v^\eps), \nabla^k \v^\eps \right\r - \left\l \nabla^k \div(\nabla \dr^\eps \odot \nabla \dr^\eps), \nabla^k \v^\eps \right\r \\
        &+ \left\l \nabla^k \div \tilde{\sigma}(\v^\eps, \dr^\eps, \dot{\dr}^\eps), \nabla^k \v^\eps \right\r \, .
       \end{aligned}
     \end{equation}
It is derived that by H\"older inequality, Sobolev embedding and the fact $\div \v =0$
\begin{align}\label{v-H^k-Estimate-2}
    \no - &\left\l \nabla^k (\v^\eps \cdot \nabla \v^\eps), \nabla^k \v^\eps \right\r \\
    \no= & - \sum_{\substack{a+b = k \\ a \geq 1}} \left\l \nabla^a \v^\eps \nabla^{b+1} \v^\eps, \nabla^k \v^\eps \right\r \\
    \no= & -  \left\l \nabla \v^\eps \nabla^k \v^\eps, \nabla^k \v^\eps \right\r-\left\l  \nabla^k \v^\eps \nabla \v^\eps, \nabla^k \v^\eps \right\r - \sum_{\substack{a + b = k \\ a \geq 2, b \geq 1}} \left\l \nabla^a \v^\eps \nabla^{b+1} \v^\eps, \nabla^k \v^\eps \right\r \\
    \leq & 2 |\nabla \v^\eps|_{L^\infty} |\nabla^k \v^\eps |^2_{L^2} +  \sum_{\substack{ a + b = k \\ a \geq 2, b \geq 1}} |\nabla^a \v^\eps|_{L^4} |\nabla^{b+1} \v^\eps|_{L^4} |\nabla^k \v^\eps|_{L^2} \\
    \no \leq & C |\nabla^3 \v^\eps|_{L^2} |\nabla^k \v^\eps|^2_{L^2} + C \sum_{\substack{ a + b = k \\ a \geq 2, b \geq 1}} |\nabla^{a+1} \v^\eps|_{L^2} |\nabla^{b+2} \v^\eps|_{L^2} |\nabla^k \v^\eps|_{L^2}\\
    \no\leq & C |\nabla^3 \v^\eps|_{L^2} |\nabla^k \v^\eps|^2_{L^2} + C |\nabla^k \v^\eps|^3_{L^2} \\
    \no\leq & C |\nabla \v^\eps|_{H^s} |\v^\eps|^2_{\dot{H}^s}\, ,
\end{align}
 and
\begin{align}\label{v-H^k-Estimate-3}
    \no& - \left\l \nabla^k \div (\nabla \dr^\eps \odot \nabla \dr^\eps), \nabla^k \v^\eps \right\r \\
    \no= & \left\l \nabla^k (\nabla \dr^\eps \odot \nabla \dr^\eps), \nabla^{k+1} \v^\eps \right\r \\
    \no= & \left\l \nabla \dr^\eps \odot \nabla^{k+1} \dr^\eps, \nabla^{k+1} \v^\eps \right\r + \left\l \nabla^{k+1} \dr^\eps \odot \nabla \dr^\eps, \nabla^{k+1} \v^\eps \right\r + \sum_{\substack{a+b=k \\ a,b \geq 1}} \left\l \nabla^{a+1} \dr^\eps \odot \nabla^{b+1} \dr^\eps , \nabla^{k+1} \v^\eps \right\r \\
    \leq & 2 |\nabla \dr^\eps|_{L^\infty} |\nabla^{k+1} \dr^\eps|_{L^2} |\nabla^{k+1} \v^\eps|_{L^2} + \sum_{\substack{a+b=k \\ a,b \geq 1}} \left\l \nabla^{a+1} \dr^\eps \odot \nabla^{b+1} \dr^\eps , \nabla^{k+1} \v^\eps \right\r \\
    \no\leq & C |\nabla^3 \dr^\eps|_{L^2} |\nabla^{k+1} \dr^\eps|_{L^2} |\nabla^{k+1} \v^\eps|_{L^2} + \sum_{\substack{a+b=k \\ a,b \geq 1}} |\nabla^{a+1} \dr^\eps|_{L^4} |\nabla^{b+1} \dr^\eps|_{L^4} |\nabla^{k+1} \v^\eps|_{L^2} \\
    \no\leq & C |\nabla^3 \dr^\eps|_{L^2} |\nabla^{k+1} \dr^\eps|_{L^2} |\nabla^{k+1} \v^\eps|_{L^2} + \sum_{\substack{a+b=k \\ a,b \geq 1}} |\nabla^{a+2} \dr^\eps|_{L^2} |\nabla^{b+2} \dr^\eps|_{L^2} |\nabla^{k+1} \v^\eps|_{L^2} \\
    \no\leq & C |\nabla^3 \dr^\eps|_{L^2} |\nabla^{k+1} \dr^\eps|_{L^2} |\nabla^{k+1} \v^\eps|_{L^2} + C |\nabla^{k+1} \dr^\eps|^2_{L^2} |\nabla^{k+1} \v^\eps|_{L^2} \\
    \no\leq & C |\nabla \dr^\eps|_{H^s}  |\nabla \dr^\eps|_{\dot{H}^s} |\nabla \v^\eps|_{H^s}\, .
\end{align}

The term $\left\l \nabla^k \div \tilde{\sigma}(\v^\eps, \dr^\eps, \dot{\dr}^\eps), \nabla^k \v^\eps \right\r$, which can be divided into five parts ($\mu_1, \mu_2, \mu_3, \mu_5, \mu_6$), remains to be estimated. First, we estimate the $\mu_1$-part: $\left\l \mu_1 \nabla^k \partial_j (\dr^\eps{}^\top \A_\eps \dr^\eps) \dr^\eps_i \dr^\eps_j, \nabla^k \v^\eps_i \right\rangle$.
\begin{align}
  \no &\left\l \mu_1 \nabla^k \partial_j (\dr^\eps{}^\top \A_\eps \dr^\eps) \dr^\eps_i \dr^\eps_j, \nabla^k \v^\eps_i \right\r \\
  \no =&-\mu_1 \left\l \nabla^k  (\dr^\eps{}^\top \A_\eps \dr^\eps \dr^\eps_i \dr^\eps_j), \nabla^k \partial_j \v^\eps_i \right\r\\
  \no =& - \mu_1 \left\l \dr^\eps{}^\top \nabla^{k+1} \v^\eps \dr^\eps, \dr^\eps{}^{\top} \nabla^{k+1} \v^\eps \dr^\eps \right\r - \mu_1 \sum_{\substack{a+b = k \\ a \geq 1}} \left\l \nabla^a (\dr^\eps_p \dr^\eps_q) \nabla^b \partial_p \v^\eps_q \dr^\eps_i \dr^\eps_j, \nabla^k \partial_j \v^\eps_i \right\r \\
  \no -& \mu_1\!\! \sum_{\substack{b+c=k \\ c \geq 1}}\!\! \left\l \dr^\eps_p \dr^\eps_q \nabla^b \partial_p \v^\eps_q \nabla^c (\dr^\eps_i \dr^\eps_j), \nabla^k \partial_j \v^\eps_i \right\r - \mu_1\!\! \!\!\!\! \sum_{\substack{a + b +c =k \\ a, c \geq 1}}\!\!\!\! \left\l \nabla^{a}(\dr^\eps_p \dr^\eps_q) \nabla^{b} \partial_p \v^\eps_q \nabla^c (\dr^\eps_i \dr^\eps_j) , \nabla^k \partial_j \v^\eps_i \right\r \\
  \no \equiv& - \mu_1 | \dr^\eps{}^\top \nabla^{k+1} \v^\eps \dr^\eps |^2_{L^2} + I_1 + I_2 + I_3\, .
\end{align}
According to H\"older inequality and Sobolev embedding, we calculate that
\begin{align}
  \no I_1 \leq& \mu_1 |\dr^\eps|^2_{L^\infty} \sum_{\substack{a_1 + a_2 + b = k \\ b \leq k-1}} \left\l |\nabla^{a_1} \dr^\eps| |\nabla^{a_2} \dr^\eps| |\nabla^{b+1} \v^\eps|, |\nabla^{k+1} \v^\eps| \right\r \\
  \no =& 2 \mu_1 |\dr^\eps|^3_{L^\infty} \sum_{\substack{a_1 + b = k \\ b \leq k-1}} \left\l |\nabla^{a_1} \dr^\eps| |\nabla^{b+1} \v^\eps|, |\nabla^{k+1} \v^\eps| \right\r \\
  \no&+ \mu_1 |\dr^\eps|^2_{L^\infty} \sum_{\substack{a_1 + a_2 + b = k \\ a_1, a_2 \geq 1}} \left\l |\nabla^{a_1} \dr^\eps| |\nabla^{a_2} \dr^\eps| |\nabla^{b+1} \v^\eps|, |\nabla^{k+1} \v^\eps| \right\r \\
  \no \leq& 2 \mu_1 |\dr^\eps|^3_{L^\infty} \sum_{\substack{a_1 + b = k \\ b \leq k-1}} |\nabla^{a_1} \dr^\eps|_{L^4} |\nabla^{b+1} \v^\eps|_{L^4} |\nabla^{k+1} \v^\eps|_{L^2} \\
   \no &+ \mu_1 |\dr^\eps|^2_{L^\infty} \sum_{\substack{a_1 + a_2 + b = k \\ a_1, a_2 \geq 1}}  |\nabla^{a_1} \dr^\eps|_{L^4} |\nabla^{a_2} \dr^\eps|_{L^4} |\nabla^{b+1} \v^\eps|_{L^\infty} |\nabla^{k+1} \v^\eps|_{L^2} \\
  \no \leq& C \mu_1 |\dr^\eps|^3_{L^\infty} \sum_{\substack{a_1 + b = k \\ b \leq k-1}} |\nabla^{a_1 + 1} \dr^\eps|_{L^2} |\nabla^{b+2} \v^\eps|_{L^2} |\nabla^{k+1} \v^\eps|_{L^2} \\
  \no &+ \mu_1 |\dr^\eps|^2_{L^\infty} \sum_{\substack{a_1 + a_2 + b = k \\ a_1, a_2 \geq 1}}  |\nabla^{a_1 +1} \dr^\eps|_{L^2} |\nabla^{a_2 + 1} \dr^\eps|_{L^2} |\nabla^{b+3} \v^\eps|_{L^2} |\nabla^{k+1} \v^\eps|_{L^2} \\
  \no \leq& C \mu_1 |\dr^\eps|^3_{L^\infty} |\nabla^{k+1} \dr^\eps|_{L^2} |\nabla^{k+1} \v^\eps|^2_{L^2} + C \mu_1 |\dr^\eps|^2_{L^\infty} |\nabla^{k+1} \dr^\eps|^2_{L^2} |\nabla^{k+1} \v^\eps|^2_{L^2} \\
  \no \leq& C \mu_1 \left(|\nabla \dr^\eps|_{H^s} |\dr^\eps|_{L^\infty} + |\nabla \dr^\eps|^2_{H^s}\right) |\dr^\eps|^2_{L^\infty} |\nabla \v^\eps|^2_{H^s}\, ,
\end{align}
and by the same estimation of the term $I_1$
\begin{align}
  \no I_2 \leq C \mu_1 \left(|\nabla \dr^\eps|_{H^s} |\dr^\eps|_{L^\infty} + |\nabla \dr^\eps|^2_{H^s}\right) |\dr^\eps|^2_{L^\infty} |\nabla \v^\eps|^2_{H^s} \, ,
\end{align}
and
\begin{align}
  \no I_3 \leq& \mu_1 \sum_{\substack{a_1+a_2+b+c_1 + c_2=k \\ a_1 + a_2, c_1 + c_2 \geq 1}} \left\l |\nabla^{a_1} \dr^\eps| |\nabla^{a_2} \dr^\eps| |\nabla^{b+1} \v^\eps| |\nabla^{c_1} \dr^\eps| |\nabla^{c_2} \dr^\eps|, |\nabla^{k+1} \v^\eps| \right\r \\
  \no =& 4 \mu_1 |\dr^\eps|_{L^\infty} \sum_{\substack{a_1 + b + c_1 + c_2 = k \\ a_1, c_1 + c_2 \geq 1}} \left\l |\nabla^{a_1} \dr^\eps|  |\nabla^{b+1} \v^\eps| |\nabla^{c_1} \dr^\eps| |\nabla^{c_2} \dr^\eps|, |\nabla^{k+1} \v^\eps| \right\r \\
  \no &+ \mu_1 \sum_{\substack{a_1+a_2+b+c_1 + c_2=k \\ a_1 , a_2, c_1 , c_2 \geq 1}} \left\l |\nabla^{a_1} \dr^\eps| |\nabla^{a_2} \dr^\eps| |\nabla^{b+1} \v^\eps| |\nabla^{c_1} \dr^\eps| |\nabla^{c_2} \dr^\eps|, |\nabla^{k+1} \v^\eps| \right\r \\
  \no =& 8 \mu_1 |\dr^\eps|^2_{L^\infty} \sum_{\substack{a_1 + b + c_1  = k \\ a_1, c_1  \geq 1}} \left\l |\nabla^{a_1} \dr^\eps|  |\nabla^{b+1} \v^\eps| |\nabla^{c_1} \dr^\eps|  , |\nabla^{k+1} \v^\eps| \right\r \\
  \no &+ 4 \mu_1 |\dr^\eps|_{L^\infty} \sum_{\substack{a_1 + b + c_1 + c_2 = k \\ a_1, c_1 ,c_2 \geq 1}} \left\l |\nabla^{a_1} \dr^\eps|  |\nabla^{b+1} \v^\eps| |\nabla^{c_1} \dr^\eps| |\nabla^{c_2} \dr^\eps|, |\nabla^{k+1} \v^\eps| \right\r \\
  \no &+ \mu_1 \sum_{\substack{a_1+a_2+b+c_1 + c_2=k \\ a_1 , a_2, c_1 , c_2 \geq 1}} \left\l |\nabla^{a_1} \dr^\eps| |\nabla^{a_2} \dr^\eps| |\nabla^{b+1} \v^\eps| |\nabla^{c_1} \dr^\eps| |\nabla^{c_2} \dr^\eps|, |\nabla^{k+1} \v^\eps| \right\r \\
  \no \leq& C \mu_1 |\dr^\eps|^2_{L^\infty} |\nabla^{k+1} \v^\eps|^2_{L^2} |\nabla^k \dr^\eps|^2_{L^2} + C \mu_1 |\dr^\eps|_{L^\infty} |\nabla^k \v^\eps|_{L^2} |\nabla^{k+1} \v^\eps|_{L^2} |\nabla^{k-1} \dr^\eps|^3_{L^2}\\
   \no &+ C \mu_1 |\nabla^{k-1} \v^\eps|_{L^2} |\nabla^{k+1} \v^\eps|_{L^2} |\nabla^{k-1} \dr^\eps|^3_{L^2} |\nabla^{k-3} \dr^\eps|_{L^2} \\
  \no \leq& C \mu_1 |\dr^\eps|^2_{L^\infty} |\nabla \v^\eps|^2_{H^s} |\nabla \dr^\eps|^2_{H^s} + C \mu_1 |\dr^\eps|_{L^\infty} |\nabla \v^\eps|^2_{H^s} |\nabla \dr^\eps|^3_{H^s} + C \mu_1 |\nabla \v^\eps|^2_{H^s} |\nabla \dr^\eps|^4_{H^s} \\
  \no =& C \mu_1 |\nabla \v^\eps|^2_{H^s}\left( |\dr^\eps|^2_{L^\infty} |\nabla \dr^\eps|^2_{H^s} + |\dr^\eps|_{L^\infty} |\nabla \dr^\eps|^3_{H^s} + |\nabla \dr^\eps|^4_{H^s} \right)\, .
\end{align}
So we get the estimate of the $\mu_1$-part
\begin{align}\label{mu1-part}
  \no &\left\l \mu_1 \nabla^k \partial_j (\dr^\eps{}^\top \A_\eps \dr^\eps \dr^\eps_i \dr^\eps_j), \nabla^k \v^\eps_i \right\r \leq - \mu_1 |\dr^\eps{}^\top \nabla^{k+1} \v^\eps \dr^\eps|^2_{L^2} \\
   &\qquad+ C \mu_1 |\nabla \v^\eps|^2_{H^s}\left( |\dr^\eps|^3_{L^\infty} |\nabla \dr^\eps|_{H^s} + |\dr^\eps|^2_{L^\infty}  |\nabla \dr^\eps|^2_{H^s} + |\dr^\eps|_{L^\infty} |\nabla \dr^\eps|^3_{H^s} + |\nabla \dr^\eps|^4_{H^s} \right)\, .
\end{align}
Second, we estimate the $\mu_2$-part: $\left\l \mu_2 \nabla^k \partial_j [\dr^\eps_j (\dot{\dr}^\eps_i + (\B_\eps)_{ki} \dr^\eps_k)], \nabla^k \v^\eps_i \right\r$.
\begin{align}\label{mu2-part}
   \no& \left\l \mu_2 \nabla^k \partial_j [\dr^\eps_j (\dot{\dr}^\eps_i + (\B_\eps)_{ki} \dr^\eps_k)], \nabla^k \v^\eps_i \right\r \\
  \no =& -\mu_2 \left\l \nabla^k [\dr^\eps_j(\dot{\dr}^\eps_i + (\B_\eps)_{ki} \dr^\eps_k)], \nabla^k \partial_j \v^\eps_i \right\r  \\
  \no =& -\mu_2 \sum_{a+b=k} \left\l \nabla^a \dot{\dr}^\eps_i \nabla^b \dr^\eps_j, \nabla^k \partial_j \v^\eps_i \right\r - \mu_2 \sum_{a+b+c=k} \left\l \nabla^a \dr^\eps_j \nabla^b (\B_\eps)_{pi} \nabla^c \dr^\eps_p, \nabla^k \partial_j \v^\eps_i \right\r \\
  \no \leq& \mu_2 |\dr^\eps|_{L^\infty} |\nabla^k \dot{\dr}^\eps|_{L^2} |\nabla^{k+1} \v^\eps|_{L^2} + \mu_2 \sum_{\substack{a+b=k \\ b \geq 1}} \left\l |\nabla^a \dot{\dr}^\eps| |\nabla^b \dr^\eps|, |\nabla^{k+1} \v^\eps| \right\r \\
  \no +& \mu_2 |\dr^\eps|^2_{L^\infty} |\nabla^{k+1} \v^\eps|^2_{L^2} + \mu_2 \sum_{\substack{a+b+c=k \\ b \leq k-1}} \left\l |\nabla^a \dr^\eps| |\nabla^{b+1} \v^\eps| |\nabla^c \dr^\eps|, |\nabla^{k+1}  \v^\eps| \right\r \\
  \no \leq& \mu_2 |\dr^\eps|_{L^\infty}  |\nabla^k \dot{\dr}^\eps|_{L^2} |\nabla^{k+1} \v^\eps|_{L^2} + \mu_2 \sum_{\substack{a+b=k \\ b \geq 1}}  |\nabla^a \dot{\dr}^\eps|_{L^4} |\nabla^b \dr^\eps|_{L^4} |\nabla^{k+1} \v^\eps|_{L^2} \\
   +& 2\mu_2 \sum_{\substack{b \leq k-1 \\ c \geq 1}} |\dr^\eps|_{L^\infty} |\nabla^{b+1} \v^\eps|_{L^4} |\nabla^c \dr^\eps|_{L^4} |\nabla^{k+1} \v^\eps|_{L^2} \\
  \no +& \mu_2 |\dr^\eps|^2_{L^\infty} |\nabla^{k+1} \v^\eps|^2_{L^2} + \mu_2 \sum_{\substack{a+b+c=k \\ a,c \geq 1}}  |\nabla^a \dr^\eps|_{L^4} |\nabla^{b+1} \v^\eps|_{L^\infty} |\nabla^c \dr^\eps|_{L^4} |\nabla^{k+1}  \v^\eps|_{L^2} \\
  \no \leq& \mu_2 |\dr^\eps|_{L^\infty} |\nabla^k \dot{\dr}^\eps|_{L^2} |\nabla^{k+1} \v^\eps|_{L^2} + C \mu_2 |\nabla^k \dot{\dr}^\eps|_{L^2} |\nabla^{k+1} \dr^\eps|_{L^2} |\nabla^{k+1} \v^\eps|_{L^2} \\
  \no +& C \mu_2 |\dr^\eps|_{L^\infty} |\nabla^{k+1} \v^\eps|^2_{L^2} |\nabla^{k+1} \dr^\eps|_{L^2} + \mu_2 |\dr^\eps|^2_{L^\infty} |\nabla^{k+1} \v^\eps|^2_{L^2} + C \mu_2 |\nabla^{k+1} \v^\eps|^2_{L^2} |\nabla^{k+1} \dr^\eps|^2_{L^2}\\
  \no \leq& \mu_2 ( |\dr^\eps|_{L^\infty} |\nabla^k \dot{\dr}^\eps|_{L^2} |\nabla^{k+1} \v^\eps|_{L^2} + |\dr^\eps|^2_{L^\infty}  |\nabla^{k+1} \v^\eps|^2_{L^2} ) +  C \mu_2 |\nabla \dr^\eps|_{H^s} |\dot{\dr}^\eps|_{H^s} |\nabla \v^\eps|_{H^s} \\
    \no +& C \mu_2 \left( |\dr^\eps|_{L^\infty} |\nabla \dr^\eps|_{H^s} + |\nabla \dr^\eps|^2_{H^s} \right) |\nabla \v^\eps|^2_{H^s}\, .
\end{align}
Third, by similar arguments on estimating the $\mu_2$-parts, we get the estimate of the $\mu_3$-part:
\begin{align}\label{mu3-part}
  \no & \left\l \mu_3 \nabla^k \partial_j [\dr^\eps_i (\dot{\dr}^\eps_j + (\B_\eps)_{kj} \dr^\eps_k)], \nabla^k \v^\eps_i \right\r \\
   \leq& \mu_3 ( |\dr^\eps|_{L^\infty} |\nabla^k \dot{\dr}^\eps|_{L^2} |\nabla^{k+1} \v^\eps|_{L^2} +  |\dr^\eps|^2_{L^\infty} |\nabla^{k+1} \v^\eps|^2_{L^2} )   \\
  \no +& C \mu_3 |\nabla \dr^\eps|_{H^s} |\dot{\dr}^\eps|_{H^s} |\nabla \v^\eps|_{H^s}  +  C \mu_3 \left( |\dr^\eps|_{L^\infty} |\nabla \dr^\eps|_{H^s} + |\nabla \dr^\eps|^2_{H^s} \right) |\nabla \v^\eps|^2_{H^s}\, .
\end{align}
Fourth, we estimate the $\mu_5$-part: $\left\l \mu_5 \nabla^k \partial_j (\dr^\eps_j \dr^\eps_p (\A_\eps)_{pi}), \nabla^k \v^\eps_i \right\r$.
\begin{align}\label{mu5-part}
   \no & \left\l \mu_5 \nabla^k \partial_j (\dr^\eps_j \dr^\eps_p (\A_\eps)_{pi}), \nabla^k \v^\eps_i \right\r \\
  \no =& - \mu_5 \left\l  \nabla^k  (\dr^\eps_j \dr^\eps_p (\A_\eps)_{pi}), \nabla^k \partial_j \v^\eps_i \right\r \\
  \no =& - \mu_5 \sum_{\substack{a + b + c = k}} \left\l \nabla^a \dr^\eps_j \nabla^b \dr^\eps_p \nabla^c (\A_\eps)_{pi}, \nabla^k \partial_j \v^\eps_i \right\r \\
  \no =& - \mu_5 \left\l \dr^\eps_j \dr^\eps_p \nabla^k (\A_\eps)_{pi} , \nabla^k \partial_j \v^\eps_i \right\r - \mu_5 \sum_{\substack{a + b + c = k \\ c \leq k - 1}} \left\l \nabla^a \dr^\eps_j \nabla^b \dr^\eps_p \nabla^c (\A_\eps)_{pi}, \nabla^k \partial_j \v^\eps_i \right\r \\
  \no \leq& \mu_5 |\dr^\eps|^2_{L^\infty} |\nabla^{k+1} \v^\eps|^2_{L^2} + 2 \mu_5 |\dr^\eps|_{L^\infty} \sum_{\substack{a + c = k \\ c \leq k - 1}} \left\l |\nabla^a \dr^\eps| |\nabla^{c+1} \v^\eps| , |\nabla^{k+1} \v^\eps| \right\r \\
  \no &+  \mu_5 \sum_{\substack{a + b + c = k \\ a, b \geq 1}} \left\l |\nabla^a \dr^\eps| |\nabla^b \dr^\eps| |\nabla^{c+1} \v^\eps |, |\nabla^{k+1} \v^\eps| \right\r \\
   \leq& \mu_5 |\dr^\eps|^2_{L^\infty} |\nabla^{k+1} \v^\eps|^2_{L^2} + 2 \mu_5 |\dr^\eps|_{L^\infty} \sum_{\substack{a + c = k \\ c \leq k - 1}}  |\nabla^a \dr^\eps|_{L^4} |\nabla^{c+1} \v^\eps|_{L^4} |\nabla^{k+1} \v^\eps|_{L^2} \\
  \no &+ \mu_5 \sum_{\substack{a + b + c = k \\ a, b \geq 1}} |\nabla^a \dr^\eps|_{L^4} |\nabla^b \dr^\eps|_{L^4} |\nabla^{c+1} \v^\eps |_{L^\infty} |\nabla^{k+1} \v^\eps|_{L^2} \\
  \no \leq&  \mu_5 |\dr^\eps|^2_{L^\infty} |\nabla^{k+1} \v^\eps|^2_{L^2} + C \mu_5 |\dr^\eps|_{L^\infty} \sum_{\substack{a + c = k \\ c \leq k - 1}}  |\nabla^{a+1} \dr^\eps|_{L^2} |\nabla^{c+2} \v^\eps|_{L^2} |\nabla^{k+1} \v^\eps|_{L^2} \\
  \no &+ C \mu_5 \sum_{\substack{a + b + c = k \\ a, b \geq 1}} |\nabla^{a+1} \dr^\eps|_{L^2} |\nabla^{b+1} \dr^\eps|_{L^2} |\nabla^{c+3} \v^\eps |_{L^2} |\nabla^{k+1} \v^\eps|_{L^2} \\
  \no \leq& \mu_5 |\dr^\eps|^2_{L^\infty} |\nabla^{k+1} \v^\eps|^2_{L^2} + C \mu_5 |\dr^\eps|_{L^\infty} |\nabla \dr^\eps|_{H^s} |\nabla \v^\eps|^2_{H^s} + C \mu_5 |\nabla \dr^\eps|^2_{H^s} |\nabla \v^\eps|^2_{H^s} \\
  \no =& \mu_5 |\dr^\eps|^2_{L^\infty} |\nabla^{k+1} \v^\eps|^2_{L^2} + C \mu_5 \left( |\dr^\eps|_{L^\infty} |\nabla \dr^\eps|_{H^s} +  |\nabla \dr^\eps|^2_{H^s} \right) |\nabla \v^\eps|^2_{H^s} \, .
\end{align}
Finally, the same arguments on the estimate of the $\mu_5$-part implies the bound of the $\mu_6$-part:
\begin{align}\label{mu6-part}
  \no &\left\l \mu_6  \nabla^k \partial_j (\dr^\eps_i \dr^\eps_p (\A_\eps)_{pj}), \nabla^k \v^\eps_i \right\r  \\ \leq & \mu_6 |\dr^\eps|^2_{L^\infty} |\nabla^{k+1} \v^\eps|^2_{L^2} + C \mu_6 \left( |\dr^\eps|_{L^\infty} |\nabla \dr^\eps|_{H^s} +  |\nabla \dr^\eps|^2_{H^s} \right) |\nabla \v^\eps|^2_{H^s} \, .
\end{align}

Substituting the inequalities \eqref{v-H^k-Estimate-2}, \eqref{v-H^k-Estimate-3}, \eqref{mu1-part}, \eqref{mu2-part}, \eqref{mu3-part}, \eqref{mu5-part} and \eqref{mu6-part} into \eqref{v-H^k-Estimate-1}, one has
\begin{align}\label{v-H^k-Estimate-4}
  \no & \frac{1}{2} \frac{\d}{\d t} |\nabla^k \v^\eps|^2_{L^2} + \frac{1}{2} |\nabla^{k+1} \v^\eps|^2_{L^2} + \mu_1 |\dr^\eps{}^\top (\nabla^{k+1} \v^\eps) \dr^\eps|^2_{L^2} \\
  \no \leq& C ( |\v^\eps|^2_{\dot{H}^s} + |\nabla \dr^\eps|_{H^s} |\nabla \dr^\eps|_{\dot{H}^s} ) |\nabla \v^\eps|_{H^s} + C (\mu_2 + \mu_3) |\nabla \dr^\eps|_{H^s} |\dot{\dr}^\eps|_{H^s} |\nabla \v^\eps|_{H^s} \\
  +& (\mu_2 + \mu_3) |\dr^\eps|_{L^\infty} |\nabla^k \dot{\dr}^\eps|_{L^2} |\nabla^{k+1} \v^\eps|_{L^2} + (\mu_2 + \mu_3 + \mu_5 + \mu_6) |\dr^\eps|^2_{L^\infty} |\nabla^{k+1} \v^\eps|^2_{L^2} \\
  \no +& C (\mu_2 + \mu_3 + \mu_5 + \mu_6) ( |\dr^\eps|_{L^\infty} |\nabla \dr^\eps|_{H^s} + |\nabla \dr^\eps|^2_{H^s} ) |\nabla \v^\eps|^2_{H^s} \\
  \no +& C \mu_1 ( |\dr^\eps|^3_{L^\infty} |\nabla \dr^\eps|_{H^s} + |\dr^\eps|^2_{L^\infty} |\nabla \dr^\eps|^2_{H^s} + |\dr^\eps|_{L^\infty} |\nabla \dr^\eps|^3_{H^s} + |\nabla \dr^\eps|^4_{H^s} ) |\nabla \v^\eps|^2_{H^s} \, .
\end{align}

Taking $\nabla^k$ in the third equation of \eqref{Appr-Syst-Simple}, multiplying by $\nabla^k \dot{\dr}^\eps$, integrating by parts, one easily obtains
\begin{align}\label{d-H^k-Estimate-1}
    \no& \frac{1}{2} \frac{\d}{\d t} \left(\rho_1 |\nabla^k \dot{\dr}^\eps|^2_{L^2} + |\nabla^{k+1} \dr^\eps|^2_{L^2} \right) -\lambda_1 |\nabla^k \dot{\dr}^\eps|^2_{L^2} \\
    \no= & - \rho_1 \sum_{\substack{a+b=k \\ a \geq 1}} \left\l \nabla^a \v^\eps \nabla^{b+1} \dot{\dr}^\eps, \nabla^k \dot{\dr}^\eps \right\r - \sum_{\substack{a+b=k+1 \\ a \geq 1}} \left\l \nabla^a \v^\eps \nabla^{b+1} \dr^\eps, \nabla^{k+1} \dr^\eps \right\r \\
    & - \rho_1 \sum_{a+b+c=k} \left\l \nabla^a \dot{\dr}^\eps \nabla^b \dot{\dr}^\eps \nabla^c \dr^\eps, \nabla^k \dot{\dr}^\eps \right\r + \sum_{a+b+c=k} \left\l \nabla^{a+1} \dr^\eps \nabla^{b+1} \dr^\eps \nabla^c \dr^\eps , \nabla^k \dot{\dr}^\eps \right\r\\
    \no & - \lambda_2 \sum_{a+b+c+e=k } \left\langle \nabla^a \dr^\eps{}^\top \nabla^b \A_\eps \nabla^c \dr^\eps \nabla^e \dr^\eps, \nabla^k \dot{\dr}^\eps \right\rangle \\
    \no &- \lambda_1 \sum_{a+b=k} \left\langle \nabla^a \B_\eps \nabla^b \dr^\eps, \nabla^k \dot{\dr}^\eps \right\rangle + \lambda_2 \sum_{a+b=k} \left\langle \nabla^a \A_\eps \nabla^b \dr^\eps, \nabla^k \dot{\dr}^\eps \right\rangle \\
    \no \equiv& R_1+ R_2 + R_3 + R_4 + R_5 + R_6 + R_7 \, .
\end{align}
By using H\"older inequality and Sobolev embedding, one can directly calculate the following estimates:
\begin{align}\label{R1}
   \no R_1= & - \rho_1 \sum_{\substack{a+b=k \\ a \geq 1}} \left\l \nabla^a \v \nabla^{b+1} \dot{\dr}, \nabla^k \dot{\dr} \right\r \\
    \no\leq & \rho_1 |\nabla \v|_{L^\infty} |\nabla^k \dot{\dr}|^2_{L^2} + \rho_1 \sum_{\substack{ a+b=k \\ a \geq 2}} |\nabla^a \v |_{L^4} |\nabla^{b+1} \dot{\dr}|_{L^4} |\nabla^k \dot{\dr}|_{L^2} \\
    \leq& C \rho_1 |\nabla^3 \v|_{L^2} |\nabla^{k} \dot{\dr}|^2_{L^2} + C \rho_1 \sum_{\substack{ a+b=k \\ a \geq 2}} |\nabla^{a+1} \v |_{L^2} |\nabla^{b+2} \dot{\dr}|_{L^2} |\nabla^k \dot{\dr}|_{L^2} \\
    \no \leq & C \rho_1 |\nabla^3 \v|_{L^2} |\nabla^{k} \dot{\dr}|^2_{L^2} + C \rho_1 |\nabla^{k+1} \v|_{L^2} |\nabla^k \dot{\dr}|^2_{L^2} \\
    \no\leq & C \rho_1 |\nabla \v|_{H^s} |\dot{\dr}|^2_{H^s}\, ,
\end{align}
and
\begin{align}\label{R2}
   \no R_2=& - \sum_{\substack{a+b=k+1 \\ a \geq 1}} \left\l \nabla^a \v^\eps \nabla^{b+1} \dr^\eps, \nabla^{k+1} \dr^\eps \right\r + |\nabla \dr^\eps|_{L^\infty} |\nabla^{k+1}  \dr^\eps|_{L^2}  |\nabla^{k+1} \v^\eps|_{L^2}  \\
   &+  \sum_{\substack{a+b=k+1 \\ a \geq 2, b\geq 2}}  |\nabla^a \v^\eps|_{L^4} |\nabla^{b+1} \dr^\eps|_{L^4} |\nabla^{k+1} \dr^\eps|_{L^2} \\
    \no \leq & C |\nabla^3 \v^\eps|_{L^2} |\nabla^{k+1} \dr^\eps|^2_{L^2} + C |\nabla^3 \dr^\eps|_{L^2} |\nabla^{k+1} \v^\eps|_{L^2} |\nabla^{k+1} \dr^\eps|_{L^2} + C |\nabla^{k+1} \v^\eps|_{L^2} |\nabla^{k+1} \dr^\eps|^2_{L^2} \\
    \no\leq & C |\nabla \v^\eps|_{H^s} |\nabla \dr^\eps|^2_{\dot{H}^s} |\nabla \dr^\eps|^2_{H^s}\, ,
\end{align}
and
\begin{align}\label{R3}
    \no R_3=& - \rho_1 \sum_{a+b+c=k} \left\l \nabla^a \dot{\dr}^\eps \nabla^b \dot{\dr}^\eps \nabla^c \dr^\eps, \nabla^k \dot{\dr}^\eps \right\r \\
    \no\leq & \rho_1 |\dr^\eps|_{L^\infty} \sum_{a+b=k} \left\l |\nabla^a \dot{\dr}^\eps| |\nabla^b \dot{\dr}^\eps|, |\nabla^k \dot{\dr}^\eps| \right\r + \rho_1 \sum_{\substack{ a + b + c = k \\ c \geq 1}} \left\l |\nabla^a \dot{\dr}^\eps| |\nabla^b \dot{\dr}^\eps| |\nabla^c \dr^\eps|, |\nabla^k \dot{\dr}^\eps| \right\r \\
    \no \leq& 2 \rho_1 |\dr^\eps|_{L^\infty} |\dot{\dr}^\eps|_{L^\infty} |\nabla^k \dot{\dr}^\eps|^2_{L^2} + \rho_1 |\dr^\eps|_{L^\infty} \sum_{\substack{a+b=k \\ a,b \geq 1}} |\nabla^a \dot{\dr}^\eps|_{L^4} |\nabla^b \dot{\dr}^\eps|_{L^4} |\nabla^k \dot{\dr}^\eps|_{L^2} \\
    +& \rho_1 \sum_{\substack{ a+b+c=k \\ c \geq 1}} |\nabla^a \dot{\dr}^\eps|_{L^6} |\nabla^b \dot{\dr}^\eps|_{L^6} |\nabla^c \dr^\eps|_{L^6} |\nabla^k \dot{\dr}^\eps|_{L^2} \\
    \no \leq& C \rho_1 |\dr^\eps|_{L^\infty} |\nabla \dot{\dr}^\eps|_{L^2} |\nabla^k \dot{\dr}^\eps|^2_{L^2} + C \rho_1 |\dr^\eps|_{L^\infty} \sum_{\substack{a+b=k \\ a,b \geq 1}} |\nabla^{a+1} \dot{\dr}^\eps|_{L^2} |\nabla^{b+1} \dot{\dr}^\eps|_{L^2} |\nabla^k \dot{\dr}^\eps|_{L^2} \\
    \no +& \rho_1 \sum_{\substack{ a+b+c=k \\ c \geq 1}} |\nabla^{a+1} \dot{\dr}^\eps|_{L^2} |\nabla^{b+1} \dot{\dr}^\eps|_{L^2} |\nabla^{c+1} \dr^\eps|_{L^2} |\nabla^k \dot{\dr}^\eps|_{L^2} \\
    \no \leq& C \rho_1 |\dr^\eps|_{L^\infty} |\dot{\dr}^\eps|^3_{H^s} + C \rho_1 |\nabla \dr^\eps|_{H^s} |\dot{\dr}^\eps|^3_{H^s} \\
    \no =& C \rho_1 \left( |\dr^\eps|_{L^\infty} + |\nabla \dr^\eps|_{H^s} \right) |\dot{\dr}^\eps|^3_{H^s}\ ,
\end{align}
and
  \begin{align}\label{R4}
    \no R_4=& \sum_{a+b+c=k} \left\l \nabla^{a+1} \dr^\eps \nabla^{b+1} \dr^\eps \nabla^c \dr^\eps , \nabla^k \dot{\dr}^\eps \right\r \\
    \no = & \sum_{a+b=k} \left\l \nabla^{a+1} \dr^\eps \nabla^{b+1} \dr^\eps \ \dr^\eps, \nabla^k \dot{\dr}^\eps \right\r + \sum_{\substack{a+b+c=k \\ c \geq 1}} \left\l \nabla^{a+1} \dr^\eps \nabla^{b+1} \dr^\eps \nabla^c \dr^\eps, \nabla^k \dot{\dr}^\eps \right\r \\
    \no \leq& 2 |\dr^\eps|_{L^\infty} |\nabla \dr^\eps|_{L^\infty} |\nabla^{k+1} \dr^\eps|_{L^2} |\nabla^k \dot{\dr}^\eps|_{L^2} + |\dr^\eps|_{L^\infty} \sum_{\substack{a+b=k \\ a,b \geq 1}} |\nabla^{a+1} \dr^\eps|_{L^4} |\nabla^{b+1} \dr^\eps|_{L^4} |\nabla^k \dot{\dr}^\eps|_{L^2} \\
    \no &+ \sum_{\substack{a+b+c=k \\ c \geq 1}} |\nabla^{a+1} \dr^\eps|_{L^6} |\nabla^{b+1} \dr^\eps|_{L^6} |\nabla^c \dr^\eps|_{L^6} |\nabla^k \dot{\dr}^\eps|_{L^2} \\
     \no\leq& C |\dr^\eps|_{L^\infty} |\nabla^3 \dr^\eps|_{L^2} |\nabla^{k+1} \dr^\eps|_{L^2} |\nabla^k \dot{\dr}^\eps|_{L^2} + C |\dr^\eps|_{L^\infty} \sum_{\substack{a+b=k \\ a,b \geq 1}} |\nabla^{a+2} \dr^\eps|_{L^2} |\nabla^{b+2} \dr^\eps|_{L^2} |\nabla^k \dot{\dr}^\eps|_{L^2} \\
     &+ C \sum_{\substack{a+b+c=k \\ c \geq 1}} |\nabla^{a+2} \dr^\eps|_{L^2} |\nabla^{b+2} \dr^\eps|_{L^2} |\nabla^{c+1} \dr^\eps|_{L^2} |\nabla^k \dot{\dr}^\eps|_{L^2} \\
    \no \leq& C |\dr^\eps|_{L^\infty} |\nabla \dr^\eps|^2_{\dot{H}^s} |\dot{\dr}^\eps|_{H^s} + C |\nabla \dr^\eps|^3_{\dot{H}^s} |\dot{\dr}^\eps|_{H^s} \\
    \no =&  C \left( |\dr^\eps|_{L^\infty} |\nabla \dr^\eps|^2_{\dot{H}^s} + |\nabla \dr^\eps|^3_{\dot{H}^s} \right) |\dot{\dr}^\eps|_{H^s} \, ,
  \end{align}
and
\begin{align}\label{R5}
  \no R_5 = & \lambda_2 \sum_{a+b+c+e=k} \left\l \nabla^a {\dr}^\eps{}^\top \nabla^b \A_\eps \nabla^c \dr^\eps \nabla^e \dr^\eps, \nabla^k \dot{\dr}^\eps \right\r \\
  \no \leq& 3 |\lambda_2| |\dr^\eps|_{L^\infty} \sum_{a+b+c=k} \left\l |\nabla^a \dr^\eps| |\nabla^{b+1} \v^\eps| |\nabla^c \dr^\eps|, |\nabla^k \dot{\dr}^\eps| \right\r \\
  \no +& |\lambda_2| \sum_{\substack{a+b+c+e=k \\ a,c,e \geq 1}} \left\l |\nabla^a \dr^\eps| |\nabla^{b+1} \v^\eps| |\nabla^c \dr^\eps| |\nabla^e \dr^\eps|, |\nabla^k \dot{\dr}^\eps| \right\r \\
  \no =& 6 |\lambda_2| |\dr^\eps|^2_{L^\infty} \sum_{a+b=k} \left\l |\nabla^a \dr^\eps| |\nabla^{b+1} \v^\eps|, |\nabla^k \dot{\dr}^\eps|  \right\r \\
  \no+& 3 |\lambda_2| |\dr^\eps|_{L^\infty} \sum_{\substack{a+b+c=k \\ a,c \geq 1}} \left\l |\nabla^a \dr^\eps| |\nabla^{b+1} \v^\eps| |\nabla^c \dr^\eps|, |\nabla^k \dot{\dr}^\eps|   \right\r \\
  \no +& |\lambda_2| \sum_{\substack{a+b+c+e=k \\ a,c,e \geq 1}} \left\l |\nabla^a \dr^\eps| |\nabla^{b+1} \v^\eps| |\nabla^c \dr^\eps| |\nabla^e \dr^\eps|, |\nabla^k \dot{\dr}^\eps| \right\r \\
  \no =& 6 |\lambda_2| |\dr^\eps|^3_{L^\infty} \left\l |\nabla^{k+1} \v^\eps|, |\nabla^k \dot{\dr}^\eps| \right\r + 6 |\lambda_2| |\dr^\eps|^2_{L^\infty} \sum_{\substack{a+b=k \\ a \geq 1}} \left\l |\nabla^a \dr^\eps| |\nabla^{b+1} \v^\eps|, |\nabla^k \dot{\dr}^\eps|  \right\r \\
  \no+& 3 |\lambda_2| |\dr^\eps|_{L^\infty} \sum_{\substack{a+b+c=k \\ a,c \geq 1}} \left\l |\nabla^a \dr^\eps| |\nabla^{b+1} \v^\eps| |\nabla^c \dr^\eps|, |\nabla^k \dot{\dr}^\eps|   \right\r \\
  \no +& |\lambda_2| \sum_{\substack{a+b+c+e=k \\ a,c,e \geq 1}} \left\l |\nabla^a \dr^\eps| |\nabla^{b+1} \v^\eps| |\nabla^c \dr^\eps| |\nabla^e \dr^\eps|, |\nabla^k \dot{\dr}^\eps| \right\r \\
  \leq& 6 |\lambda_2| |\dr^\eps|^3_{L^\infty} |\nabla^{k+1} \v^\eps|_{L^2} |\nabla^k \dot{\dr}^\eps|_{L^2} + 6 |\lambda_2| |\dr^\eps|^2_{L^\infty} \sum_{\substack{a+b=k \\ a \geq 1}} |\nabla^a \dr^\eps|_{L^4} |\nabla^{b+1} \v^\eps|_{L^4} |\nabla^k \dot{\dr}^\eps|_{L^2} \\
  \no +& 3 |\lambda_2| |\dr^\eps|_{L^\infty} \sum_{\substack{a+b+c=k \\ a,c \geq 1}} |\nabla^a \dr^\eps|_{L^4} |\nabla^{b+1} \v^\eps|_{L^\infty} |\nabla^c \dr^\eps|_{L^4} |\nabla^k \dot{\dr}^\eps|_{L^2} \\
  \no +& |\lambda_2| \sum_{\substack{a+b+c+e=k \\ a,c,e \geq 1}} |\nabla^a \dr^\eps|_{L^6} |\nabla^{b+1} \v^\eps|_{L^\infty} |\nabla^c \dr^\eps|_{L^6} |\nabla^e \dr^\eps|_{L^6} |\nabla^k \dot{\dr}^\eps|_{L^2} \\
  \no \leq& 6 |\lambda_2| |\dr^\eps|^3_{L^\infty} |\nabla^{k+1} \v^\eps|_{L^2} |\nabla^k \dot{\dr}^\eps|_{L^2} + C  |\lambda_2| |\dr^\eps|^2_{L^\infty} \sum_{\substack{a+b=k \\ a \geq 1}} |\nabla^{a+1} \dr^\eps|_{L^2} |\nabla^{b+2} \v^\eps|_{L^2} |\nabla^k \dot{\dr}^\eps|_{L^2} \\
  \no +& C |\lambda_2| |\dr^\eps|_{L^\infty} \sum_{\substack{a+b+c=k \\ a,c \geq 1}} |\nabla^{a+1} \dr^\eps|_{L^2} |\nabla^{b+3} \v^\eps|_{L^2} |\nabla^{c+1} \dr^\eps|_{L^2} |\nabla^k \dot{\dr}^\eps|_{L^2} \\
  \no +& C |\lambda_2| \sum_{\substack{a+b+c+e=k \\ a,c,e \geq 1}} |\nabla^{a+1} \dr^\eps|_{L^2} |\nabla^{b+3} \v^\eps|_{L^3} |\nabla^{c+1} \dr^\eps|_{L^2} |\nabla^{e+1} \dr^\eps|_{L^2} |\nabla^k \dot{\dr}^\eps|_{L^2} \\
  \no \leq& 6 |\lambda_2| |\dr^\eps|^3_{L^\infty} |\nabla^{k+1} \v^\eps|_{L^2} |\nabla^k \dot{\dr}^\eps|_{L^2} + C |\lambda_2| |\dr^\eps|^2_{L^\infty} |\nabla \dr^\eps|_{H^s} |\nabla \v^\eps|_{H^s} |\dot{\dr}^\eps|_{H^s} \\
  \no +& C |\lambda_2| |\dr^\eps|_{L^\infty} |\nabla \dr^\eps|^2_{H^s} |\nabla \v^\eps|_{H^s} |\dot{\dr}^\eps|_{H^s} + C |\lambda_2| |\nabla \dr^\eps|^3_{H^s} |\nabla \v^\eps|_{H^s} |\dot{\dr}^\eps|_{H^s} \\
  \no =& 6 |\lambda_2| |\dr^\eps|^3_{L^\infty} |\nabla^{k+1} \v^\eps|_{L^2} |\nabla^k \dot{\dr}^\eps|_{L^2}  \\
  \no +& C |\lambda_2| \left( |\dr^\eps|^2_{L^\infty} |\nabla \dr^\eps|_{H^s} + |\dr^\eps|_{L^\infty} |\nabla \dr^\eps|^2_{H^s} + |\nabla \dr^\eps|^3_{H^s} \right) |\nabla \v^\eps|_{H^s} |\dot{\dr}^\eps|_{H^s}\, ,
\end{align}
and
\begin{align}\label{R6}
  \no R_6 = & - \lambda_1 \sum_{a+b=k} \left\langle \nabla^a \B_\eps \nabla^b \dr^\eps, \nabla^k \dot{\dr}^\eps \right\rangle \\
  \no =& - \lambda_1 \left\l  \nabla^k \B_\eps \ \dr^\eps, \nabla^k \dot{\dr}^\eps \right\r - \lambda_1 \sum_{\substack{a+b=k \\ b \geq 1}} \left\l \nabla^a \B_\eps \nabla^b \dr^\eps, \nabla^k \dot{\dr}^\eps \right\r \\
  \no \leq& |\lambda_1| |\dr^\eps|_{L^\infty} |\nabla^{k+1} \v^\eps|_{L^2} |\nabla^k \dot{\dr}^\eps|_{L^2} + |\lambda_1| \sum_{\substack{a+b=k \\ b \geq 1}} |\nabla^{a+1} \v^\eps|_{L^4} |\nabla^b \dr^\eps|_{L^4} |\nabla^k \dot{\dr}^\eps|_{L^2} \\
  \leq& |\lambda_1| |\dr^\eps|_{L^\infty} |\nabla^{k+1} \v^\eps|_{L^2} |\nabla^k \dot{\dr}^\eps|_{L^2} + C |\lambda_1| \sum_{\substack{a+b=k \\ b \geq 1}} |\nabla^{a+2} \v^\eps|_{L^2} |\nabla^{b+1} \dr^\eps|_{L^2} |\nabla^k \dot{\dr}^\eps|_{L^2} \\
  \no \leq&  |\lambda_1| |\dr^\eps|_{L^\infty} |\nabla^{k+1} \v^\eps|_{L^2} |\nabla^k \dot{\dr}^\eps|_{L^2} + C |\lambda_1| |\nabla \dr^\eps|_{H^s} |\nabla \v^\eps|_{H^s} |\dot{\dr}^\eps|_{H^s} \\
  \no =& |\lambda_1| |\dr^\eps|_{L^\infty} |\nabla^{k+1} \v^\eps|_{L^2} |\nabla^k \dot{\dr}^\eps|_{L^2} +  C |\lambda_1|  |\nabla \dr^\eps|_{H^s} |\nabla \v^\eps|_{H^s} |\dot{\dr}^\eps|_{H^s}\, ,
\end{align}
and by similarly estimating on the term $R_6$
\begin{align}\label{R7}
  R_7 \leq |\lambda_2| |\dr^\eps|_{L^\infty} |\nabla^{k+1} \v^\eps|_{L^2} |\nabla^k \dot{\dr}^\eps|_{L^2} + C |\lambda_2|  |\nabla \dr^\eps|_{H^s} |\nabla \v^\eps|_{H^s} |\dot{\dr}^\eps|_{H^s}\, .
\end{align}

Combining the inequalities \eqref{R1}, \eqref{R2}, \eqref{R3}, \eqref{R4}, \eqref{R5}, \eqref{R6} and \eqref{R7}, it can be derived from \eqref{d-H^k-Estimate-1} that the following estimate holds:
\begin{align}\label{d-H^k-Estimate-2}
    \no& \frac{1}{2} \frac{\d}{\d t} \left( \rho_1 |\nabla^k \dot{\dr}^\eps|^2_{L^2} + |\nabla^{k+1} \dr^\eps|^2_{L^2} \right) - \lambda_1 |\nabla^k \dot{\dr}^\eps|^2_{L^2} \\
    \no \leq & C \left( \rho_1 |\dot{\dr}^\eps|^2_{H^s} + |\nabla \dr^\eps|^2_{\dot{H}^s} \right) |\nabla \v^\eps|_{H^s} + C \rho_1 \left( |\dr^\eps|_{L^\infty} + |\nabla \dr^\eps|_{H^s} \right) |\dot{\dr}^\eps|^3_{H^s} \\
    +& C \left( |\dr^\eps|_{L^\infty} |\nabla \dr^\eps|^2_{\dot{H}^s} + |\nabla \dr^\eps|^3_{\dot{H}^s} \right) |\dot{\dr}^\eps|_{H^s} \\
     \no +& \left[6 |\lambda_1| |\dr^\eps|^2_{L^\infty} +( |\lambda_1| + |\lambda_2|) |\dr^\eps|_{L^\infty} \right] |\nabla^{k+1} \v^\eps|_{L^2} |\nabla^k \dot{\dr}^\eps|_{L^2}  \\
      \no +& C (|\lambda_1| + |\lambda_2|) \left( |\nabla \dr^\eps|_{H^s} + |\dr^\eps|^2_{L^\infty} |\nabla \dr^\eps|_{H^s} + |\dr^\eps|_{L^\infty} |\nabla \dr^\eps|^2_{H^s}  + |\nabla \dr^\eps|^3_{H^s} \right) |\dot{\dr}^\eps|_{H^s} |\nabla \v^\eps|_{H^s} \, .
\end{align}

Then, by the coefficients relations $\lambda_1 = \mu_2 - \mu_3$, $\lambda_2 = \mu_5 - \mu_6$ and $\mu_2 + \mu_3 = \mu_6 - \mu_5 \geq 0$, the inequalities \eqref{L2-Estimate-5}, \eqref{v-H^k-Estimate-4} and \eqref{d-H^k-Estimate-2} reduce to
\begin{align}\label{H^s-Estimate-1}
  \no  & \frac{1}{2} \frac{\d}{\d t} \left( |\v^\eps|^2_{H^s} + \rho_1 |\dot{\dr}^\eps|^2_{H^s} + |\nabla \dr^\eps|^2_{H^s} \right) + \frac{1}{2} \mu_4 |\nabla \v^\eps|^2_{H^s} - \lambda_1 |\dot{\dr}^\eps|^2_{H^s} + \mu_1 \sum_{k=0}^s |\dr^\eps{}^\top (\nabla^{k+1} \v^\eps) \dr^\eps|^2_{L^2} \\
  \no \leq& ( 7 |\lambda_1| - 2 \lambda_2 ) ( |\dr^\eps|_{L^\infty} + |\dr^\eps|^2_{L^\infty} + |\dr^\eps|^3_{L^\infty} ) |\dot{\dr}^\eps|_{H^s} |\nabla \v^\eps|_{H^s} + 2 \mu_6 |\dr^\eps|^2_{L^\infty} |\nabla \v^\eps|^2_{H^s} \\
  \no +& C  ( |\v^\eps|^2_{\dot{H}^s} + \rho_1 |\dot{\dr}^\eps|^2_{H^s} +  |\nabla \dr^\eps|_{H^s} |\nabla \dr^\eps|_{\dot{H}^s}  ) |\nabla \v^\eps|_{H^s} - C \lambda_2 |\nabla \dr^\eps|_{H^s} |\dot{\dr}^\eps|_{H^s} |\nabla \v^\eps|_{H^s} \\
  +& C ( |\dr^\eps|_{L^\infty} |\dot{\dr}^\eps|_{H^s} + |\dr^\eps|_{L^\infty} |\nabla \dr^\eps|_{H^s} + |\nabla \dr^\eps|^2_{H^s} + |\dot{\dr}^\eps|^2_{H^s} ) ( \rho_1 |\dot{\dr}^\eps|^2_{H^s} + |\nabla \dr^\eps|^2_{H^s} ) \\
  \no +& C (|\lambda_1| - \lambda_2) ( |\nabla \dr^\eps|_{H^s} + |\dr^\eps|^2_{L^\infty} |\nabla \dr^\eps|_{H^s} + |\dr^\eps|_{L^\infty} |\nabla \dr^\eps|^2_{H^s} + |\nabla \dr^\eps|^3_{H^s} ) |\dot{\dr}^\eps|_{H^s} |\nabla \v^\eps|_{H^s} \\
  \no +& C (\mu_1 + 2 \mu_6) \Big{[} (|\dr^\eps|_{L^\infty} + |\dr^\eps|^3_{L^\infty}) |\nabla \dr^\eps|_{H^s} \\
  \no +& (1 + |\dr^\eps|^2_{L^\infty}) |\nabla \dr^\eps|^2_{H^s} + |\dr^\eps|_{L^\infty} |\nabla \dr^\eps|^3_{H^s} + |\nabla \dr^\eps|^4_{H^s} \Big{]} |\nabla \v^\eps|^2_{H^s} \, ,
\end{align}
where $C = C (n,s) > 0$. According to the inequality \eqref{L2-Estimate-1}, we know that
\begin{align}\label{L^infty-Estimate}
  \no |\dr^\eps|_{L^\infty} \leq& |\dr^\eps - \J_\eps \dr^{in}|_{L^\infty} + |\J_\eps \dr^{in} |_{L^\infty} \\
  \no \leq& C |\dr^\eps - \J_\eps \dr^{in}|_{H^{s+1}} + 1 \\
  \leq& C |\dr^\eps - \J_\eps \dr^{in}|_{L^2} + C |\nabla \dr^\eps|_{H^s} +  C |\nabla \dr^{in}|_{H^s} + 1 \\
  \no \leq& C E_\eps^\frac{1}{2} (t) + C |\nabla \dr^{in}|_{H^s} + 1 \, .
\end{align}

As a consequence, it is derived from the inequalities \eqref{L2-Estimate-1}, \eqref{H^s-Estimate-1} and \eqref{L^infty-Estimate} that
\begin{align}
  \no & \frac{1}{2} \frac{\d}{\d t} E_\eps(t) + \frac{1}{2} \mu_4 F_\eps (t) \\
  \no \leq& C \big{(} \tfrac{1 + |\lambda_1|}{\sqrt{\rho_1}} + |\nabla \dr^{in}|_{H^s} \big{)} E_\eps(t) + \tfrac{1}{\sqrt{\rho_1}} ( 7 |\lambda_1| - 2 \lambda_2 ) \sum_{i=1}^3 |\dr^\eps|^i_{L^\infty} E_\eps^\frac{1}{2} (t) F_\eps^\frac{1}{2} (t) \\
  \no +& 2 \mu_6 |\dr^\eps|^2_{L^\infty} F_\eps(t) + C \big{(} 1 - \tfrac{\lambda_2}{\sqrt{\rho_1}} \big{)} E_\eps(t) F_\eps^\frac{1}{2} (t) + C ( 1 + \tfrac{1}{(\sqrt{\rho_1})^3} ) ( |\dr^\eps|_{L^\infty} E_\eps^\frac{1}{2} (t) + E_\eps(t) ) E_\eps(t) \\
  \no +& \tfrac{C}{\sqrt{\rho_1}} ( |\lambda_1| - \lambda_2 ) ( E_\eps^\frac{1}{2} (t) + |\dr^\eps|_{L^\infty} E_\eps^\frac{1}{2} (t) + |\dr^\eps|_{L^\infty} E_\eps(t) + E_\eps^\frac{3}{2}(t) ) ( E_\eps(t) + F_\eps(t)) \\
  \no +& C (\mu_1 + \mu_6) [ (|\dr^\eps|_{L^\infty} + |\dr^\eps|^3_{L^\infty}) E_\eps^\frac{1}{2} (t) + (1+ |\dr^\eps|^2_{L^\infty}) E_\eps(t) + |\dr^\eps|_{L^\infty} E_\eps^\frac{3}{2} (t) + E_\eps^2 (t)] F_\eps (t) \\
  \no \leq& C \Big{[} \tfrac{1 + |\lambda_1| - \lambda_2}{\sqrt{\rho_1}} + \tfrac{1}{(\sqrt{\rho_1})^3} + |\nabla \dr^{in}|_{H^s} + \tfrac{ |\lambda_1| - \lambda_2 }{\sqrt{\rho_1}} |\nabla \dr^{in}|^2_{H^s} \Big{]} ( E_\eps (t) + E_\eps^2(t) + E_\eps^3(t) ) \\
  \no +& C \big{(}1-\tfrac{\lambda_2}{\sqrt{\rho_1}}\big{)} E_\eps(t) F_\eps^\frac{1}{2}(t) + \tfrac{3}{\sqrt{\rho_1}} (|\lambda_1| - 2 \lambda_2) E_\eps^\frac{1}{2}(t) F_\eps^\frac{1}{2}(t) \\
  \no +& 2 \mu_6 F_\eps(t) + C \big{(} \mu_1 + \mu_6 + \tfrac{|\lambda_1| - \lambda_2}{\sqrt{\rho_1}}  \big{)} (1+|\nabla \dr^{in}|^2_{H^s}) \Big{(} |\nabla \dr^{in}|_{H^s} + \sum_{i=1}^5 E_\eps^\frac{i}{2} (t) \Big{)} F_\eps(t) \, ,
\end{align}
which immediately implies the inequality \eqref{Bounds-2} and then the proof of Lemma \ref{Unif-Bnd-Lemma} is finished.

\end{proof}

We emphasize that if $|\dr^\eps|=1$, we know that $\left\l \gamma(\v^\eps, \dr^\eps, \dot{\dr}^\eps) \dr^\eps, \dot{\dr}^\eps \right\r = 0$. Then the $L^2$-estimate \eqref{L2-Estimate-5} would be presented as
\begin{align}\label{L2-Estimate-global}
   \no &\frac{1}{2} \frac{\d}{\d t} \left( |\v^\eps|^2_{L^2} + \rho_1 |\dot{\dr}^\eps|^2_{L^2} + |\nabla \dr^\eps|^2_{L^2} \right) + \frac{1}{2} \mu_4 |\nabla \v^\eps|^2_{L^2} - \lambda_1 |\dot{\dr}^\eps|^2_{L^2} + \mu_1 |\dr^\eps{}^\top (\nabla \v^\eps) \dr^\eps|^2_{L^2} \\
    \leq& ( |\lambda_1| - 2 \lambda_2| )  |\dr^\eps|_{L^\infty}  |\nabla \v^\eps|_{L^2} |\dot{\dr}^\eps|_{L^2}  + 2 \mu_6  |\dr^\eps|^2_{L^\infty} |\nabla \v^\eps|^2_{L^2} \, .
\end{align}
By using the fact $|\dr^\eps|=1$, one can deduce that by the inequalities \eqref{v-H^k-Estimate-4}, \eqref{d-H^k-Estimate-2} and \eqref{L2-Estimate-global}
\begin{align}\label{H^s-Estimate-global}
  \no & \frac{1}{2} \frac{\d}{\d t} \left( |\v^\eps|^2_{H^s} + \rho_1 |\dot{\dr}^\eps|^2_{H^s} + |\nabla \dr^\eps|^2_{H^s} \right) + \frac{1}{2} \mu_4 |\nabla \v^\eps|^2_{H^s} - \lambda_1 |\dot{\dr}^\eps|^2_{H^s} + \mu_1 \sum_{k=0}^s |\dr^\eps{}^\top (\nabla^{k+1} \v^\eps) \dr^\eps|^2_{L^2} \\
  \no \leq& 2 \mu_6 |\nabla \v^\eps|^2_{H^s} + ( 7 |\lambda_2| - 2 \lambda_2 ) |\dot{\dr}^\eps|_{H^s} |\nabla \v^\eps|_{H^s} \\
  \no +& C  (1-\lambda_2) ( |\v^\eps|^2_{\dot{H}^s} + \rho_1 |\dot{\dr}^\eps|^2_{H^s} + |\nabla \dr^\eps|_{H^s} |\nabla \dr^\eps|_{\dot{H}^s} + |\nabla \dr^\eps|_{H^s} |\dot{\dr}^\eps|_{H^s} ) |\nabla \v^\eps|_{H^s} \\
  +& C \big{(} 1 + \tfrac{1}{\sqrt{\rho_1}} \big{)} ( |\dot{\dr}^\eps|_{H^s} + |\nabla \dr^\eps|_{H^s} + |\dot{\dr}^\eps|^2_{H^s} + |\nabla \dr^\eps|^2_{H^s} ) ( \rho_1 |\dot{\dr}^\eps|^2_{H^s} + |\nabla \dr^\eps|^2_{\dot{H}^s} ) \\
  \no +& C ( |\lambda_1| - \lambda_2 ) ( |\nabla \dr^\eps|_{H^s} + |\nabla \dr^\eps|^2_{H^s} + |\nabla \dr^\eps|^3_{H^s} ) |\dot{\dr}^\eps|_{H^s} |\nabla \v^\eps|_{H^s} \\
  \no +& C (\mu_1 + 2 \mu_6) ( |\nabla \dr^\eps|_{H^s} + |\nabla \dr^\eps|^2_{H^s} + |\nabla \dr^\eps|^3_{H^s} + |\nabla \dr^\eps|^4_{H^s} ) |\nabla \v^\eps|^2_{H^s} \, ,
\end{align}
which will be used in extending the global solutions.

\section{The proof of Theorem \ref{Main-Thm}.}

In this section, we are going to show the main theorem by analyzing some results of compactness and then passing to the limit on the approximate system \eqref{Appr-Syst-Simple}.

\subsection{The proof of Part (I) of Theorem \ref{Main-Thm}.} One observes that
$$ E_\epsilon (0) = |\J_\eps \v^{in}|^2_{H^s} + \rho_1 |\J_\eps {\tilde\dr}^{in}|^2_{H^s} + |\nabla \J_\eps \dr^{in}|^2_{H^s} \leq E^{in} \, .$$
If $ E^{in} <  \min \left\{ 1,  \frac{\rho_1 \beta^2}{ [ 96 C ( \sqrt{\rho_1} (\mu_1 + \mu_6) +  |\lambda_1| - \lambda_2  ) ]^2 }  \right\}  $, where the positive constant $C = C(n,s) > 0$ is determined in Lemma \ref{Unif-Bnd-Lemma}, then by simple calculation we have ${\bf Q} (E_\eps (0)) \leq \frac{1}{8} \beta $ and
$$C_0 (\lambda_1, \lambda_2, \beta,\rho_1, |\nabla \dr^{in}|_{H^s}) \leq C_1 \equiv  C_0 (\lambda_1, \lambda_2, \beta,\rho_1, 1)\, .$$
 We define
  $$ T^*_\eps = \sup \left\{ T > 0; E_\eps(t) \leq 2 \textrm{\ and\ } {\bf Q }(E_\eps (t)) \leq \frac{1}{4} \beta \textrm{\ hold \ for\ all\ } t \in [0,T] \right\} \, . $$

  According to the calculation \eqref{L2-Estimate-1}, we know that
  $$ \frac{\d}{\d t} |\dr^\eps - \J_\eps \dr^{in}|_{L^2} \leq C (1 + |\nabla \dr^{in}|_{H^s}) ( |\v^\eps|_{H^s} + |\dot{\dr}^\eps|_{H^s} ) \, . $$
  Then Lemma \ref{Local-Existence-Approximated} implies that $E_\eps(t)$ is continuous. Consequently, we know that $T^*_\eps > 0$. Then the inequality \eqref{Bounds-2} in Lemma \ref{Unif-Bnd-Lemma} implies that for any fixed $\eps > 0$ and for all $t \in [ 0,  T^*_\eps  ]$
    \begin{align}
     \no \frac{1}{2} \frac{\d}{\d t} E_\epsilon(t) + \big{[} \tfrac{1}{4}\beta - {\bf Q}(E_\eps (t) ) \big{]} F_\epsilon(t) \leq 12 C_1 E_\epsilon(t) \, ,
  \end{align}
  which immediately imply that
  \begin{equation}\no
    E_\eps (t) \leq E_\eps(0) e^{24 C_1 t} \leq E^{in} e^{24 C_1 t}
  \end{equation}
  holds for all $\eps > 0$ and $t \in [ 0,  T^*_\eps  ]$. Thus there is a $0 < T \leq \frac{1}{48 C_1} \ln \frac{1}{E^{in}}$ such that for all $t \in [0, \min \{ T, T^*_\eps \}]$
  $$ E_\eps (t) \leq E^{in} e^{24 C_1 T} \leq \sqrt{E^{in}} < 1 \, , $$
  which consequently implies that for all $t \in [0, \min \{ T, T^*_\eps \}]$
  \begin{align}
    \no {\bf Q} (E_\eps(t)) \leq & 2 C \big{(} \mu_1 + \mu_6 + \tfrac{|\lambda_1| - \lambda_2}{\sqrt{\rho_1}} \big{)} ( \sqrt{E^{in}} + 5 \sqrt{E_\eps(t)} ) \\
    \no \leq& 2 C \big{(} \mu_1 + \mu_6 + \tfrac{|\lambda_1| - \lambda_2}{\sqrt{\rho_1}} \big{)} ( \sqrt{E^{in}} + 5 \sqrt{\sqrt{E^{in}}} ) \\
    \no \leq& 12 C \big{(} \mu_1 + \mu_6 + \tfrac{|\lambda_1| - \lambda_2}{\sqrt{\rho_1}} \big{)} (E^{in})^\frac{1}{4} \, .
  \end{align}

  So there is an $\eps_0 = \min \Big{\{} 1 , \frac{\rho_1 \beta^2}{ [ 96 C ( \sqrt{\rho_1} (\mu_1 + \mu_6) +  |\lambda_1| - \lambda_2  ) ]^2 } , \frac{\rho^2_1 \beta^4}{ [ 96 C ( \sqrt{\rho_1} (\mu_1 + \mu_6) +  |\lambda_1| - \lambda_2  ) ]^4 } \Big{\}} > 0 $ such that if $E^{in} < \eps_0$, then for all $t \in [0, \min \{ T, T^*_\eps \}]$
  $$E_\eps(t) \leq \sqrt{E^{in}} < 1$$
  and $$ {\bf Q} (E_\eps(t)) \leq \frac{1}{8} \beta \, . $$
  Moreover, by the continuity of $E_\eps(t)$ we know that $T^*_\eps \geq T$. As a consequence, we obtain that for all $t \in [0,T]$
  \begin{align}\label{Bounds-3}
    |\dr^\eps - \J_\eps \dr^{in}|^2_{L^2}(t) + |\v^\epsilon|^2_{H^s}(t) + \rho_1 |\dot{\dr}^\epsilon|^2_{H^s}(t) + |\nabla  \dr^\epsilon|^2_{H^s}(t) + \frac{1}{4} \beta \int_0^t |\nabla  \v^\epsilon|^2_{H^s}(\tau) \d \tau \leq \widetilde{C}_1 (E^{in}, T) \, ,
  \end{align}
  where $\widetilde{C}_1 (E^{in}, T) =  E^{in} + 12 C_1 T \sqrt{E^{in}} > 0$. By the inequality \eqref{L^infty-Estimate} we also know that for all $\eps > 0$ and $t \in [0,T]$
  \begin{equation}\label{L^infty-Unif}
    |\dr^\eps|_{L^\infty} \leq C E_\eps^\frac{1}{2} (t) + C |\nabla \dr^{in}|_{H^s} + 1 \leq C \sqrt{\sqrt{6 E^{in}}} + C \sqrt{E^{in}} + 1 < \infty\, .
  \end{equation}

  Thus we know that there exist functions $$\v \in L^\infty(0,T;H^s)\cap L^2(0,T;H^{s+1})\, , \dr \in L^\infty(0,T;\dot{H}^{s+1})\cap L^\infty ((0,T ) \times \R^n) \, , \w \in L^\infty(0,T;H^s)$$ such that (extracting subsequence if necessary)
\begin{align}\label{Weak-Compactness-General}
   \begin{array}{l}
    \v^\epsilon \rightharpoonup \v \ \textrm{weakly* \ in\ } L^\infty(0,T;H^s)\, ,\\
    \v^\epsilon \rightharpoonup \v \ \textrm{weakly \ in\ } L^2(0,T;H^{s+1})\, ,\\
    \dr^\epsilon \rightharpoonup \dr \ \textrm{weakly* \ in\ } L^\infty(0,T;\dot{H}^{s+1})\, ,\\
    \dot{\dr}^\epsilon \rightharpoonup \w \ \textrm{weakly* \ in\ } L^\infty(0,T;H^s)\, .\\
  \end{array}
\end{align}

Then the second equation of \eqref{Appr-Syst-Simple} reduces to
\begin{equation}\label{1st-Lim-Equa}
   \div \v = 0 \, .
\end{equation}

 Noticing that $s > \frac{n}{2} + 2$, one can estimate by Lemma \ref{Basic-Properties}, the equations \eqref{Appr-Syst-Simple}, the bounds \eqref{Bounds-3} and Sobolev embedding theory that
\begin{align}\label{t-Derivative-Bounds-d}
   \no \sup_{0 \leq t \leq T} | \partial_t \dr^\epsilon |_{H^s} \leq& \sup_{0 \leq t \leq T} | \dot{ \dr}^\epsilon |_{H^s} + \sup_{0 \leq t \leq T} \left| \v^\epsilon \cdot \nabla  \dr^\epsilon \right|_{H^s} \\
     \leq& \sup_{0 \leq t \leq T} | \dot{ \dr}^\epsilon |_{H^s} + C \sup_{0 \leq t \leq T} \left( |  \v^\epsilon|_{L^\infty} | \nabla  \dr^\epsilon |_{H^s} + |\nabla  \dr^\epsilon|_{L^\infty} | \v^\epsilon|_{H^s} \right)  \\
    \no\leq& \sup_{0 \leq t \leq T} | \dot{ \dr}^\epsilon |_{H^s} + C \sup_{0 \leq t \leq T} | \v^\epsilon |_{H^s} \sup_{0 \leq t \leq T} | \nabla \dr^\epsilon |_{H^s}  \\
    \no \leq&  \widetilde{C}_1(E^{in},T)  \, ,
\end{align}
and
\begin{align}
    \no | \partial_t \v^\epsilon |_{H^{s-2}} \leq& \left| \mathcal{P} \J_\eps [ \v^\epsilon \cdot \nabla  \v^\epsilon] \right|_{H^{s-2}} + \frac{1}{2}\mu_4 \left|  \Delta   \v^\epsilon  \right|_{H^{s-2}} \\
    \no &+ \left| \mathcal{P} \J_\eps \div (  \nabla   \dr^\epsilon \odot  \nabla  \dr^\epsilon) \right|_{H^{s-2}} + \left| \mathcal{P} \J_\eps \div\tilde{\sigma}(\v^\eps, \dr^\eps, \dot{\dr}^\eps)  \right|_{H^{s-2}} \\
    \no \leq& \left|  \v^\epsilon \cdot \nabla  \v^\epsilon \right|_{H^{s-2}} + \frac{1}{4}\mu_4 \left| \v^\epsilon \right|_{H^s} + \left| \div (  \nabla   \dr^\epsilon \odot  \nabla  \dr^\epsilon) \right|_{H^{s-2}}\\
     \no &+ \left| \mathcal{P} \J_\eps \div\tilde{\sigma}(\v^\eps, \dr^\eps, \dot{\dr}^\eps)  \right|_{H^{s-2}} \\
    \no \leq& C \left| \v^\epsilon \right|_{H^s}^2 + \frac{1}{2} \mu_4 \left| \v^\epsilon \right|_{H^s} + C \left| \nabla \dr^\epsilon \right|_{H^s}^2 + \left| \mathcal{P} \J_\eps \div\tilde{\sigma}(\v^\eps, \dr^\eps, \dot{\dr}^\eps)  \right|_{H^{s-2}} \\
     \no \leq& \widetilde{C}_1(E^{in},T) + \left| \mathcal{P} \J_\eps \div\tilde{\sigma}(\v^\eps, \dr^\eps, \dot{\dr}^\eps)  \right|_{H^{s-2}} \, ,
\end{align}
where
\begin{align}
  \no &\left| \mathcal{P} \J_\eps \div\tilde{\sigma}(\v^\eps, \dr^\eps, \dot{\dr}^\eps)  \right|_{H^{s-2}} \\
  \no \leq& \mu_1  | (  \dr^\epsilon \otimes  \dr^\epsilon :  \A_\epsilon )  \dr^\epsilon \otimes  \dr^\epsilon  |_{H^{s-1}} + (\mu_2 + \mu_3)  |  \dr^\epsilon \otimes  \dot{\dr}^\epsilon  |_{H^{s-1}} \\
  \no +& (\mu_2 + \mu_3)  | \dr^\epsilon \otimes  \dot{\dr}^\epsilon :  \B_\epsilon |_{H^{s-1}} + (\mu_5 + \mu_6)  |  \dr^\epsilon \otimes  \dr^\epsilon :  \A_\epsilon |_{H^{s-1}} \\
  \no \equiv& S_1 + S_2 + S_3 + S_4 \, .
\end{align}

Now we estimate the terms $S_1,\ S_2,\ S_3$ and $S_4$. We can estimate that by \eqref{Bounds-3}, \eqref{L^infty-Unif}, H\"older inequality and Sobolev embedding
\begin{align}
  \no S_1 =& \mu_1 |   (   \dr^\eps \otimes  \dr^\eps : \A^\eps ) \dr^\eps \otimes \dr^\eps  |_{L^2} +  \mu_1 \sum_{k=1}^{s-1} \left| \nabla^k [  \dr^\eps \otimes \dr^\eps : \A^\eps ) \dr^\eps \otimes \dr^\eps ] \right|_{L^2} \\
  \no \leq& \mu_1 |\dr^\eps|^4_{L^\infty} \sum_{k=0}^{s-1} |\nabla^{k+1}  \v^\eps|_{L^2} + 24 \mu_1 |\dr^\eps|^3_{L^\infty} \sum_{k=1}^{s-1} \sum_{\substack{ a+ f = k \\ a \geq 1}} \left| |\nabla^a  \dr^\eps| |\nabla^{f+1}  \v^\eps | \right|_{L^2} \\
  \no +& 12 \mu_1 |\dr^\eps|^2_{L^\infty} \sum_{k=1}^{s-1} \sum_{\substack{ a+b+ f = k \\ a, b \geq 1}} \left| |\nabla^a  \dr^\eps| |\nabla^b  \dr^\eps | |\nabla^{f+1} \v^\eps | \right|_{L^2} \\
  \no +& 4 \mu_1 |\dr^\eps|_{L^\infty} \sum_{k=1}^{s-1} \sum_{\substack{ a+b+c+ f = k \\ a, b, c \geq 1}} \left| |\nabla^a  \dr^\eps| |\nabla^b  \dr^\eps | |\nabla^c \dr^\eps | |\nabla^{f+1}  \v^\eps | \right|_{L^2} \\
  \no +& \mu_1 \sum_{k=1}^{s-1} \sum_{\substack{ a+b+c+ e+ f = k \\ a, b, c, e \geq 1}} \left| |\nabla^a \dr^\eps| |\nabla^b \dr^\eps | |\nabla^c \dr^\eps | |\nabla^e \dr^\eps | |\nabla^{f+1}  \v^\eps | \right|_{L^2} \\
  \no \leq& \mu_1 |\dr^\eps|^4_{L^\infty} |\v^\eps|_{H^s} + 24 \mu_1 |\dr^\eps|^3_{L^\infty} \sum_{k=1}^{s-1} \sum_{\substack{ a+ f = k \\ a \geq 1}}  |\nabla^a \dr^\eps|_{L^4} |\nabla^{f+1} \v^\eps |_{L^4} \\
  \no +& 12 \mu_1 |\dr^\eps|^2_{L^\infty} \sum_{k=1}^{s-1} \sum_{\substack{ a+b+ f = k \\ a, b \geq 1}} |\nabla^a \dr^\eps|_{L^6} |\nabla^b \dr^\eps |_{L^6} |\nabla^{f+1} \v^\eps |_{L^6} \\
  \no +& 4 \mu_1 |\dr^\eps|_{L^\infty} \sum_{k=1}^{s-1} \sum_{\substack{ a+b+c+ f = k \\ a, b, c \geq 1}} |\nabla^a \dr^\eps|_{L^6} |\nabla^b \dr^\eps |_{L^6} |\nabla^c \dr^\eps |_{L^6} |\nabla^{f+1} \v^\eps |_{L^\infty} \\
  \no +& \mu_1 \sum_{k=1}^{s-1} \sum_{\substack{ a+b+c+ e+ f = k \\ a, b, c, e \geq 1}} |\nabla^a \dr^\eps|_{L^6} |\nabla^b \dr^\eps |_{L^6} |\nabla^c \dr^\eps |_{L^6} |\nabla^e \dr^\eps |_{L^\infty} |\nabla^{f+1} \v^\eps |_{L^\infty} \\
  \no \leq& \mu_1 |\dr^\eps|^4_{L^\infty} |\v^\eps|_{H^s} + C \mu_1 |\dr^\eps|^3_{L^\infty} |\nabla \dr^\eps|_{H^s} |\v^\eps |_{H^s} + C \mu_1 |\dr^\eps|^2_{L^\infty} |\nabla \dr^\eps|^2_{H^s} |\v^\eps |_{H^s} \\
  \no &+ C \mu_1 |\dr^\eps|_{L^\infty} |\nabla \dr^\eps|^3_{H^s} |\v^\eps |_{H^s} + C \mu_1 |\nabla \dr^\eps|^4_{H^s} |\v^\eps |_{H^s} \\
  \no \leq& \widetilde{C}_1(E^{in},T) \mu_1\, ,
\end{align}
and by the similar calculation on estimating the term $S_1$
$$ S_2 + S_3 + S_4 \leq \widetilde{C}_1 (E^{in}, T) (\mu_2 + \mu_3 + \mu_5 + \mu_6)\, . $$
Therefore, we have obtained the bound
\begin{equation}\label{t-Derivative-Bounds-v}
  \sup_{0 \leq t \leq T} \left| \partial_t \v^\epsilon \right|_{H^{s-2}} \leq \widetilde{C}_1(E^{in}, T) (1 + \mu_1 + \mu_2 + \mu_3 + \mu_5 + \mu_6)\, .
\end{equation}

It also can be derived from  the third equation of the system \eqref{Appr-Syst-Simple} and the bounds \eqref{Bounds-3}-\eqref{L^infty-Unif}
\begin{align}
  \no \rho_1 |\partial_t \dot{\dr}^\eps|_{H^{s-1}} \leq& \rho_1 |  \v^\eps \cdot \nabla  \dot{\dr}^\eps ] |_{H^{s-1}} + |  \Delta  \dr^\eps |_{H^{s-1}} + |\gamma(\v^\eps, \dr^\eps, \dot{\dr}^\eps) \dr^\eps |_{H^{s-1}} \\
  \no &+ |\lambda_1| |\dot{\dr}^\eps|_{H^{s-1}} + |\lambda_1| |  \B_\eps  \dr^\eps |_{H^{s-1}} + |\lambda_2| | \A_\eps \dr^\eps  |_{H^{s-1}} \\
  \no \leq& |\nabla \dr^\eps|_{H^s} + |\lambda_1| |\dot{\dr}^\eps|_{H^s} + \rho_1 | \v^\eps \cdot \nabla \dot{\dr}^\eps |_{H^{s-1}} \\
  \no &+ |\gamma(\v^\eps, \dr^\eps, \dot{\dr}^\eps) \dr^\eps|_{H^{s-1}} +  ( |\lambda_1| + |\lambda_2| ) |\nabla \v^\eps \cdot \dr^\eps|_{H^{s-1}} \\
  \no \leq& \widetilde{C}_1(E^{in} , T) +\rho_1 | \v^\eps \cdot \nabla \dot{\dr}^\eps |_{H^{s-1}} \\
  \no &+ ( |\lambda_1| + |\lambda_2| ) |\nabla \v^\eps \cdot \dr^\eps|_{H^{s-1}} + |\gamma(\v^\eps, \dr^\eps, \dot{\dr}^\eps) \dr^\eps|_{H^{s-1}} \\
  \no \equiv& \widetilde{C}_1(E^{in} , T) + K_1 + K_2 + K_3 \, .
\end{align}
Now we estimate the terms $K_1$, $K_2$ and $K_3$ by H\"older inequality, calculus inequality and Sobolev embedding:
\begin{align}
  \no K_1 =& \rho_1 | \v^\eps \cdot \nabla \dot{\dr}^\eps |_{H^{s-1}} \\
  \no \leq& C \rho_1 \left( | \v^\eps|_{H^{s-1}} |\nabla \dot{\dr}^\eps |_{L^\infty} + | \v^\eps|_{L^\infty} |\nabla \dot{\dr}^\eps|_{H^{s-1}} \right) \\
  \no \leq& C \rho_1 \left( |\v^\eps|_{H^s} |\nabla^3 \dot{\dr}^\eps|_{L^2} + |\nabla^2 \v^\eps|_{L^2} |\dot{\dr}^\eps|_{H^s} \right) \\
  \no \leq& C \rho_1 |\v^\eps|_{H^s} |\dot{\dr}^\eps|_{H^s} \\
  \no \leq& \widetilde{C}_1(E^{in} , T) \sqrt{\rho_1} \, ,
\end{align}
and
\begin{align}
  \no K_2 =& (|\lambda_1| + |\lambda_2|) |\nabla  \v^\eps \cdot  \dr^\eps|_{L^2} + ( |\lambda_1| + |\lambda_2| ) \sum_{k=1}^{s-1} \left| \nabla^k (\nabla \v^\eps \cdot \dr^\eps) \right|_{L^2} \\
  \no \leq& (|\lambda_1| + |\lambda_2|) |\dr^\eps|_{L^\infty} |\nabla \v^\eps|_{L^2} + (|\lambda_1| + |\lambda_2|) \sum_{k=1}^{s-1} \sum_{a+b=k} \left| \nabla^{a+1} \v^\eps \cdot \nabla^b \dr^\eps \right|_{L^2} \\
  \no \leq& (|\lambda_1| + |\lambda_2|) |\dr^\eps|_{L^\infty} |\nabla \v^\eps|_{L^2} + (|\lambda_1| + |\lambda_2|) |\dr^\eps|_{L^\infty} \sum_{k=1}^{s-1} |\nabla^{k+1} \v^\eps|_{L^2} \\
  \no &+ (|\lambda_1| + |\lambda_2|) \sum_{k=1}^{s-1} \sum_{\substack{a+b=k \\ b \geq 1}} \left| \nabla^{a+1} \v^\eps \cdot \nabla^b \dr^\eps \right|_{L^2} \\
  \no \leq& C (|\lambda_1| + |\lambda_2|) |\dr^\eps|_{L^\infty} |\v^\eps|_{H^s} + (|\lambda_1| + |\lambda_2|) \sum_{k=1}^{s-1} \sum_{\substack{a+b=k \\ b \geq 1}} |\nabla^{a+1} \v^\eps|_{L^4} |\nabla^b \dr^\eps|_{L^4} \\
  \no \leq& (|\lambda_1| + |\lambda_2|) |\dr^\eps|_{L^\infty} |\v^\eps|_{H^s} + C (|\lambda_1| + |\lambda_2|) \sum_{k=1}^{s-1} \sum_{\substack{a+b=k \\ b \geq 1}} |\nabla^{a+2} \v^\eps|_{L^2} |\nabla^{b+1} \dr^\eps|_{L^2} \\
  \no \leq& C (|\lambda_1| + |\lambda_2|) |\dr^\eps|_{L^\infty} |\v^\eps|_{H^s} + C (|\lambda_1| + |\lambda_2|)  |\v^\eps|_{H^s} |\nabla \dr^\eps|_{H^s} \\
  \no \leq& \widetilde{C}_1(E^{in}, T) (|\lambda_1| + |\lambda_2|) \, ,
\end{align}
and by similar calculation on estimating the term $K_2$
\begin{align}
  \no K_3 \leq & C ( |\dr^\eps|_{L^\infty} + |\nabla \dr^\eps|_{H^s} ) ( \rho_1 |\dot{\dr}^\eps|^2_{H^s} + |\nabla \dr^\eps|^2_{H^s}  ) \\
  \no &+ C |\lambda_1| |\v^\eps|_{H^s} ( |\dr^\eps|^3_{L^\infty} + |\dr^\eps|^2_{L^\infty} |\nabla \dr^\eps|_{H^s} + |\dr^\eps|_{L^\infty} |\nabla \dr^\eps |^2_{H^s} + |\nabla \dr^\eps|^3_{H^s} ) \\
  \no \leq& \widetilde{C}_1 (E^{in}, T) (1+|\lambda_1|) \, .
\end{align}
Thus the following bound holds uniformly in $\eps\,$:
\begin{equation}\label{t-Derivative-Bounds-dot-d}
  \sup_{0 \leq t \leq T} |\partial_t \dot{\dr}^\eps|_{H^{s-1}} \leq \widetilde{C}_1(E^{in},T) \tfrac{1}{\rho_1} (1+\sqrt{\rho_1}+|\lambda_1|+|\lambda_2|)\, .
\end{equation}

Thanks to the Aubin-Lions lemma, the weak convergence \eqref{Weak-Compactness-General} and the bounds \eqref{t-Derivative-Bounds-d}, \eqref{t-Derivative-Bounds-v} and \eqref{t-Derivative-Bounds-dot-d} reduce to the following convergence
\begin{align}\label{Strong-Compactness-General}
   \begin{array}{l}
    \v^\epsilon \rightarrow \v \ \textrm{strongly \ in\ } L^\infty(0,T;H^{s'})\, ,\\
    \dr^\epsilon \rightarrow \dr \ \textrm{strongly \ in\ } L^\infty(0,T;\dot{H}^{s'+1})\, ,\\
    \dot{\dr}^\epsilon \rightarrow \w \ \textrm{strongly \ in\ } L^\infty(0,T;H^{s'})\\
  \end{array}
\end{align}
for any $ 0 < s' < s$.

Assuming that $s' > \frac{n}{2}$, then $H^{s'} \hookrightarrow L^\infty$. We claim that $\J_\eps ( \v^\eps \cdot \dr^\eps)$ strongly converges to $\v \cdot \nabla \dr$ in $L^\infty(0,T;H^{s'})$ with $\v \cdot \nabla \dr \in L^\infty(0,T;H^s)$.

Indeed, by Lemma \ref{Basic-Properties}, we have
\begin{align}
  \no & |\J_\eps (  \v^\eps \cdot \nabla  \dr^\eps) - \v \cdot \nabla \dr|_{H^{s'}} \\
  \no \leq& | \J_\eps (  \v^\eps \cdot \nabla  \dr^\eps) - \J_\eps ( \v \cdot \nabla \dr^\eps) |_{H^{s'}} \\
  \no &+ | \J_\eps ( \v \cdot \nabla  \dr^\eps) - \J_\eps ( \v \cdot \nabla  \dr ) |_{H^{s'}} + | \J_\eps ( \v \cdot \nabla  \dr ) - \v \cdot \nabla  \dr  |_{H^{s'}} \\
  \no \leq& C \left( |\v^\eps - \v|_{H^{s'}} |\nabla  \dr^\eps|_{L^\infty} + |\v^\eps - \v|_{L^\infty} |\nabla   \dr^\eps|_{H^{s'}}  \right) \\
  \no &+ C \left( |\v|_{H^{s'}} |\nabla   \dr^\eps -\nabla \dr|_{L^\infty} + |\v|_{L^\infty} |\nabla  \dr^\eps -\nabla \dr|_{H^{s'}} \right) + | \J_\eps ( \v \cdot \nabla  \dr ) - \v \cdot \nabla  \dr  |_{H^{s'}} \\
  \no \leq& C |\v^\eps - \v|_{H^{s'}} |\nabla \dr^\eps|_{H^{s'}} + C |\v|_{H^{s'}} |\nabla \dr^\eps - \nabla \dr|_{H^{s'}} + | \J_\eps ( \v \cdot \nabla  \dr ) - \v \cdot \nabla  \dr  |_{H^{s'}} \\
  \no \leq& \widetilde{C}_1(E^{in},T) \left( |\v^\eps - \v|_{H^{s'}} + |\nabla  \dr^\eps - \nabla \dr|_{H^{s'}} \right) + | \J_\eps ( \v \cdot \nabla  \dr ) - \v \cdot \nabla  \dr  |_{H^{s'}} \\
  \no & \rightarrow 0 \ \textrm{as\ } \eps \rightarrow 0\, .
\end{align}
With the similar arguments in \eqref{t-Derivative-Bounds-d} we know that $\mathcal{J}_\eps ( \v^\eps \cdot \dr^\eps)$ is uniformly bounded in $L^\infty(0,T;H^s)$. Then this claim holds.

Furthermore, by \eqref{t-Derivative-Bounds-d} one knows that
\begin{equation}
  \no \partial_t \dr = \lim_{\eps \rightarrow 0 } \partial_t \dr^\eps = \lim_{ \eps \rightarrow 0} \left( \dot{\dr}^\eps - \J_\eps ( \v^\eps \cdot \nabla \dr^\eps) \right) = \w - \v \cdot \nabla \dr \in L^\infty (0,T; H^s) \, ,
\end{equation}
where the limits are considered in the sense of distribution. Consequently, it reduces to $ \partial_t \dr \in L^\infty (0,T; H^s)  $ and $ \w = \dot{\dr} \in L^\infty (0,T; H^s)$. Namely, \begin{equation}\label{2nd-Lim-Equa}
  \partial_t \dr = \dot{\dr} - \v \cdot \dr \, .
\end{equation}

Assuming that $s' \geq \frac{n}{2} + 2$, we have continuous embeddings $H^{s'} \hookrightarrow C^2$. Then the following convergence holds:
\begin{equation}\label{C2-Strong-Converge}
   \begin{array}{l}
    \v^\eps \rightarrow \v \textrm{\ strongly \ in \ } L^\infty (0,T; C^2) \, ,\\
    \dr^\eps \rightarrow \dr \textrm{\ strongly \ in \ } L^\infty(0,T; C^3) \, , \\
    \dot{\dr}^\eps \rightarrow \dot{\dr} \textrm{\ strongly \ in \ } L^\infty (0,T;C^2) \, . \\
  \end{array}
\end{equation}

Thus by the first equation of the approximate system \eqref{Appr-Syst-Simple} we observe that
\begin{align}\label{Strong-converge-v}
  \no \partial_t \v^\eps = & - \mathcal{P} \J_\eps ( \v^\eps \cdot \nabla \v^\eps ) + \frac{1}{2} \mu_4  \Delta  \v^\eps - \mathcal{P} \J_\eps \div (\nabla  \dr^\eps \odot \nabla  \dr^\eps) + \mathcal{P} \J_\eps \div \tilde{\sigma}(\v^\eps, \dr^\eps, \dot{\dr}^\eps) \\
  \rightarrow & - \mathcal{P} (\v \cdot \nabla \v) + \frac{1}{2} \mu_4 \Delta \v - \mathcal{P} \div (\nabla \dr \odot \nabla \dr) + \mathcal{P} \div \tilde{\sigma}
\end{align}
strongly in $L^\infty (0,T; C^0)$ as $\eps \rightarrow 0$, where
\begin{align}
  \no (\tilde{\sigma}(\v^\eps, \dr^\eps, \dot{\dr}^\eps) )_{ji} =&  \mu_1  \dr^\epsilon_k \dr^\epsilon_p  (\A_{kp})_\epsilon  \dr_i^\epsilon  \dr_j^\epsilon + \mu_2  \dr_j^\epsilon (  \dot{\dr}_i^\epsilon +   (\B_{ki})_\epsilon  \dr_k^\epsilon )  \\
 \no +&  \mu_3  \dr_i^\epsilon (  \dot{\dr}_j^\epsilon +  (\B_{kj})_\epsilon  \dr_k^\epsilon ) + \mu_5  \dr_j^\epsilon  \dr_k^\epsilon  (\A_{ki})_\epsilon + \mu_6 \dr_i^\epsilon  \dr_k^\epsilon  (\A_{kj})_\epsilon \\
 \no \rightarrow & \mu_1  \dr_k \dr_p \A_{kp} \dr_i \dr_j + \mu_2  \dr_j (  \dot{\dr}_i + \B_{ki} \dr_k )  \\
 \no +&  \mu_3  \dr_i (  \dot{\dr}_j + \B_{kj} \dr_k ) + \mu_5  \dr_j \dr_k \A_{ki} + \mu_6 \dr_i \dr_k \A_{kj} \\
 \no \equiv \tilde{\sigma}_{ji}
\end{align}
strongly in $L^\infty (0,T; C^1)$ as $\eps \rightarrow 0$. Recalling that $\v^\eps \rightarrow \v$ strongly in $L^\infty(0,T;H^{s'})$ and the bound \eqref{t-Derivative-Bounds-v}, we know that for all $\varphi \in C_0^\infty ((0,T) \times \R^n; \R^n)$
\begin{align}
  \no \lim_{\eps \rightarrow 0} \int_0^T \int_{\R^n} \partial_t \v^\eps \cdot \varphi \d x \d t = & - \lim_{\eps \rightarrow 0} \int_0^T \int_{\R^n} \v^\eps \cdot \partial_t \varphi \d x \d t \\
  \no =& - \int_0^T \int_{\R^n} \v \cdot \partial_t \varphi \d x \d t \\
  \no =& \int_0^T \int_{\R^n} \partial_t \v \cdot \varphi \d x \d t \, ,
\end{align}
hence
\begin{equation}\label{weak-converge-v-tDerivative}
  \lim_{\eps \rightarrow 0} \partial_t \v^\eps = \partial_t \v
\end{equation}
in the sense of distribution. Then the convergence \eqref{Strong-converge-v} and \eqref{weak-converge-v-tDerivative} imply that
\begin{align}\label{3rd-Lim-Equa}
  \partial_t \v = - \mathcal{P} (\v \cdot \nabla \v) + \frac{1}{2} \mu_4 \Delta \v - \mathcal{P} \div (\nabla \dr \odot \nabla \dr) + \mathcal{P} \div \tilde{\sigma} \in  L^\infty(0,T;C^0) \, ,
\end{align}
where $\tilde{\sigma}_{ji} = \mu_1  \dr_k \dr_p \A_{kp} \dr_i \dr_j + \mu_2  \dr_j (  \dot{\dr}_i + \B_{ki} \dr_k )   +  \mu_3  \dr_i (  \dot{\dr}_j + \B_{kj} \dr_k ) + \mu_5  \dr_j \dr_k \A_{ki} + \mu_6 \dr_i \dr_k \A_{kj}  $.

Similarly, the strong convergence \eqref{C2-Strong-Converge} and the third equation of \eqref{Appr-Syst-Simple} reduce to
\begin{align}
  \no \gamma(\v^\eps, \dr^\eps, \dot{\dr}^\eps) = & - \rho_1 |\dot{\dr}^\eps|^2 + |\nabla   \dr^\eps|^2 - \lambda_2   \dr^\eps{}^\top \A_\eps   \dr^\eps  \\
  \no \rightarrow & - \rho_1 |\dot{\dr}|^2 + |\nabla \dr|^2 - \lambda_2 \dr^\top \A \dr \equiv \gamma \in L^\infty(0,T; C^0)
\end{align}
strongly in $L^\infty(0,T; C^0)$ and $\eps \rightarrow 0$, and then
\begin{align}\label{Strong-converge-d}
  \no \rho_1 \partial_t \dot{\dr}^\eps = & - \rho_1 \J_\eps ( \v^\eps \cdot \dot{\dr}^\eps ) +  \Delta \dr^\eps  + \gamma(\v^\eps, \dr^\eps, \dot{\dr}^\eps) \dr^\eps + \lambda_1 ( \dot{\dr}^\eps - \J_\eps (\B_\eps  \dr^\eps) ) + \lambda_2 \J_\eps (\A_\eps \dr^\eps) \\
  \rightarrow & - \rho_1 \v \cdot \nabla \dot{\dr} + \Delta \dr + \gamma \dr + \lambda_1 ( \dot{\dr} - \B \dr ) + \lambda_2 \A \dr
\end{align}
strongly in $L^\infty (0,T; C^0)$ as $\eps \rightarrow 0$. It is derived from the strong convergence $\dot{\dr}^\eps \rightarrow \dot{\dr}$ in $L^\infty(0,T;H^{s'})$ and the bound \eqref{t-Derivative-Bounds-dot-d} that
\begin{equation}\label{weak-converge-dotd-tDerivative}
  \lim_{\eps \rightarrow 0} \partial_t \dot{\dr}^\eps = \partial_t \dot{\dr}
\end{equation}
in the distributional sense. Then the following equation is implied by the convergence \eqref{Strong-converge-d} and \eqref{weak-converge-dotd-tDerivative}:
\begin{equation}\label{4th-Lim-Equa}
  \rho_1 \partial_t \dot{\dr} + \rho_1 \v \cdot \nabla \dot{\dr} = \Delta \dr + \gamma \dr + \lambda_1 ( \dot{\dr} - \B \dr ) + \lambda_2 \A \dr \, .
\end{equation}
and $\partial_t \dot{\dr} \in L^\infty(0,T;C^0)$. Therefore, combining the equations \eqref{1st-Lim-Equa}, \eqref{2nd-Lim-Equa}, \eqref{3rd-Lim-Equa} and \eqref{4th-Lim-Equa}, we know that $(\v, \dr )$ satisfies the system
  \begin{align}
    \no \left\{ \begin{array}{c}
      \partial_t \v + \v \cdot \nabla \v - \frac{1}{2} \mu_4 \Delta \v + \nabla p = - \div (\nabla \dr \odot \nabla \dr) + \div \tilde{\sigma}\, , \\
      \div \v = 0\, ,\\
     \rho_1 \ddot{\dr} = \Delta \dr + \gamma \dr + \lambda_1 (\dot{\dr}- \B \dr) + \lambda_2 \A \dr\, ,
    \end{array}\right.
  \end{align}
where $\gamma = - \rho_1 |\dot{\dr}|^2 + |\nabla \dr|^2 - \lambda_2 \dr^\top \A \dr $ and $(\v , \dr)$ obeys the initial data conditions \eqref{Inital-Data} and \eqref{Inital-Data-Compatablity}. Moreover, the inequality \eqref{L^infty-Unif} implies that $|\dr|_{L^\infty} < \infty$. Then Lemma \ref{Parallel-d-Lemma} yields the geometric constraint $|\dr| = 1$.

Furthermore, by Fatou lemma, the bound \eqref{Bounds-3} implies that
\begin{equation}\no
 |\dr - \dr^{in}|^2_{L^2} (t) +  |\v|^2_{H^s}(t) + \rho_1 |\dot{\dr}|^2_{H^s}(t) + |\nabla  \dr|^2_{H^s}(t) + \frac{1}{4} \beta \int_0^t |\nabla  \v|^2_{H^s}(\tau) \d \tau \leq \widetilde{C}_1 (E^{in},T)
\end{equation}
for all $t \in [0,T]$. The proof of Part (I) of Theorem \ref{Main-Thm} is finished.

\subsection{The proof of Part (II) of Theorem \ref{Main-Thm}.} Since $\mu_1 = \mu_2 = \mu_3 = \mu_5 = \mu_6 = 0$, the inequality \eqref{Approximated-Apriori-Estimates-Special} in Remark \ref{Remark-mu=0} immediately implies that for any $\eps > 0$ and all $t \in [ 0, T_\eps )$
\begin{equation}\label{Gronwall-Inequ-2}
    \frac{\d}{\d t} \left( \ln \frac{ E_\eps^\frac{1}{2} (t) [ E_\eps (t) + 2 ]^\frac{1}{2} }{ E_\eps (t) + 1 } \right) \leq 2 C_2 \, ,
  \end{equation}
  where the constant $C_2 = C \big{(} \tfrac{1}{\sqrt{\rho_1}} + \tfrac{1}{\mu_4} + |\nabla \dr^{in}|_{H^s} \big{)} > 0$.

  We claim that for any fixed $0 < T < \frac{1}{4 C_2} \ln \frac{1}{ Y(E^{in})} $ and for all $\eps > 0$, $t \in [ 0, T ]$
  $$ E_\eps(t) \leq W(E^{in} , T) \equiv \frac{1}{\sqrt{ 1 - Y(E^{in} ) e^{4 C_2 T} }} - 1 < \infty \, , $$
  where the quantity $ Y(E^{in}) = \frac{  E^{in} ( E^{in} + 2 )  }{( E^{in} + 1 )^2} \in (0,1)$.

  Indeed, for any $t\in [0,\min \{ T, T_\eps \}]$, we integrates the inequality \eqref{Gronwall-Inequ-2} on [0, t] and then we have
  \begin{equation}\no
     \frac{ E_\eps^\frac{1}{2} (t) [ E_\eps (t) + 2 ]^\frac{1}{2} }{ E_\eps (t) + 1 } \leq \frac{ E_\eps^\frac{1}{2} (0) [ E_\eps (0) + 2 ]^\frac{1}{2} }{ E_\eps (0) + 1 } e^{2 C_2 t} \leq \sqrt{Y (E^{in} )} e^{2 C_2 T} \, .
  \end{equation}
  Thus for any fixed $\eps > 0$ and for all $t \in [0, T]$
  \begin{equation}
    \no \left[ 1 - Y(E^{in} ) e^{4 C_2 T} \right] E_\eps^2 (t) + 2 \left[ 1 - Y( E^{in} ) e^{4 C_2 T} \right] E_\eps (t) - Y( E^{in} ) e^{4 C_2 T} \leq 0 \, .
  \end{equation}
  Let $a = Y(E^{in} ) e^{4 C_2 T}  \in (0,1) $ and $ y = E_\eps(t) \geq 0 $, then
  $$ (1 -a) y^2 + 2 (1 -a) y - a \leq 0 \, . $$
   Consequently, we have
  \begin{equation}
    \no 0 \leq y = E_\eps(t) \leq \frac{1}{\sqrt{1 -a}} - 1 = \frac{1}{\sqrt{ 1 - Y(E^{in} ) e^{4 C_2 T} }} - 1 \, .
  \end{equation}
  Thus the claim is justified.

We claim that $ T_\eps \geq T$. Indeed, if $ T_\eps < T$, by the continuity of $E_\eps(t)$ we know that $E_\eps(T_\eps) < \infty$. Then we can extend the solutions constructed in Lemma \ref{Local-Existence-Approximated}, which contradicts with the fact that $[0, T_\eps)$ is the maximal interval for the existence of the solutions. So $T_\eps \geq T$.

In fact, by the arbitrariness of $T$ we know that $T_\eps \geq \frac{1}{4 C_2} \ln \frac{1}{ Y(E^{in})}$. As a consequence, if $E^{in} < \infty$, then there exists a $ 0 < T < \frac{1}{4 C_2} \ln \frac{1}{ Y(E^{in})} $ such that for any $\eps > 0$ and for all $t \in [0,T]$
 \begin{align}\label{H^s-Bounds}
   \no &|\dr^\eps(t) - \J_\eps \dr^{in}|^2_{L^2} + |\v^\epsilon(t)|^2_{H^s} + \rho_1 |\dot{\dr}^\epsilon(t)|^2_{H^s} + |\nabla  \dr^\epsilon(t)|^2_{H^s} + \frac{1}{2}\mu_4 \int_0^t |\nabla   \v^\epsilon|^2_{H^s} (\tau) \d \tau \leq \widetilde{C}_2 (E^{in},T)\, , \\
  &|\dr^\eps|_{L^\infty} \leq 1 + C \sqrt{E^{in}} + C \sqrt{W(E^{in},T)}\, ,
 \end{align}
where $\widetilde{C}_2 (E^{in},T) = E^{in} + 2 C_2 T W(E^{in},T) [W(E^{in},T) + 1 ] [ W(E^{in},T) + 2 ] > 0 $. Using the similar arguments of passing to the limit in the proof of Part (I) of Theorem \ref{Main-Thm}, we show that $(\v, \dr)$ is a strong solution to the hyperbolic-type system \eqref{Hyperbolic-Liquid-Crystal-Model}. Moreover, Fatou lemma and the bounds \eqref{H^s-Bounds} yield that the solution $(\v, \dr)$ satisfies $|\dr|_{L^\infty} < \infty$ and the following inequality holds for all $t \in [0,T]$
\begin{equation}\no
  |\dr(t) - \dr^{in}|^2_{L^2} + |\v(t)|^2_{H^s} + \rho_1 |\dot{\dr}(t)|^2_{H^s} + |\nabla \dr(t)|^2_{H^s} + \frac{1}{2}\mu_4 \int_0^t |\nabla  \v|^2_{H^s} (\tau) \d \tau \leq \widetilde{C}_2 (E^{in},T) \, .
\end{equation}
Then we finish the proof of Part (II) of Theorem \ref{Main-Thm}.

\subsection{The proof of Part (III) of Theorem \ref{Main-Thm}.} Since $\mu_2 < \mu_3$, we have $- \lambda_1 = |\lambda_1|= \mu_3 - \mu_2 > 0$. Assume that $(\v,\dr)$ is the solution constructed in Part (I) of Theorem \ref{Main-Thm} with the geometric constraint $|\dr|=1$.

  Next we construct another apriori estimate. For all $1 \leq k \leq s $, we take $\nabla^k$ in the third direction equation of the parabolic-hyperbolic model \eqref{Parabolic-Hyperbolic-Liquid-Crystal-Model}, multiply by $\nabla^k \dr$ and then we get
  \begin{equation}\no
    \rho_1 \left\l \nabla^k \ddot{\dr} , \nabla^k \dr \right\r = \left\l \nabla^k \Delta \dr , \nabla^k \dr \right\r + \left\l \nabla^k (\gamma \dr), \nabla^k \dr \right\r + \lambda_1 \left\l \nabla^k ( \dot{\dr} - \B \dr ), \nabla^k \dr \right\r + \left\l \nabla^k (\A \dr), \nabla^k \dr \right\r \, .
  \end{equation}
  By simple calculation one deduces
  \begin{align}
    \no \left\l \nabla^k \ddot{\dr} , \nabla^k \dr \right\r =& \left\l \partial_t \nabla^k \dot{\dr}, \nabla^k \dr \right\r + \left\l \nabla^k (\v \cdot \nabla \dot{\dr}), \nabla^k \dr \right\r \\
    \no =& \frac{\d}{ \d t} \left\l \nabla^k \dot{\dr} , \nabla^k \dr \right\r - \left\l \nabla^k \dot{\dr} , \partial_t \nabla^k \dr \right\r \\
    \no &+ \left\l \v \cdot \nabla \nabla^k \dot{\dr}, \nabla^k \dr \right\r + \sum_{\substack{a+b=k \\ a \geq 1}} \left\l \nabla^a \v \nabla^{b+1} \dot{\dr} , \nabla^k \dr \right\r \\
    \no =& \frac{\d}{ \d t} \left\l \nabla^k \dot{\dr} , \nabla^k \dr \right\r - \left\l \nabla^k \dot{\dr} , \partial_t \nabla^k \dr \right\r \\
    \no &- \left\l \v \cdot \nabla \nabla^k \dr , \nabla^k \dot{\dr} \right\r + \sum_{\substack{a+b=k \\ a \geq 1}} \left\l \nabla^a \v \nabla^{b+1} \dot{\dr} , \nabla^k \dr \right\r \\
    \no =& \frac{\d}{ \d t} \left\l \nabla^k \dot{\dr} , \nabla^k \dr \right\r - \left\l \nabla^k \dot{\dr} , \partial_t \nabla^k \dr \right\r - \left\l \nabla^k \dot{\dr}, \nabla^k (\v \cdot \nabla \dr) \right\r \\
    \no &+ \sum_{\substack{a+b=k \\ a \geq 1}} \left\l \nabla^a \v \nabla^{b+1} \dr , \nabla^k \dot{\dr} \right\r + \sum_{\substack{a+b=k \\ a \geq 1}} \left\l \nabla^a \v \nabla^{b+1} \dot{\dr} , \nabla^k \dr \right\r \\
    \no =& \frac{1}{2} \frac{\d}{ \d t} \left( |\nabla^k (\dot{\dr} - \dr)|^2_{L^2} - |\nabla^k \dot{\dr}|^2_{L^2} - |\nabla^k \dr|^2_{L^2} \right) - |\nabla^k \dot{\dr}|^2_{L^2} \\
    \no &+ \sum_{\substack{a+b=k \\ a \geq 1}} \left\l \nabla^a \v \nabla^{b+1} \dr , \nabla^k \dot{\dr} \right\r + \sum_{\substack{a+b=k \\ a \geq 1}} \left\l \nabla^a \v \nabla^{b+1} \dot{\dr} , \nabla^k \dr \right\r\, ,
  \end{align}
  and
\begin{equation}
  \no \left\l \nabla^k \Delta \dr , \nabla^k \dr \right\r = - |\nabla^{k+1} \dr|^2_{L^2} \, ,
\end{equation}
and by the fact $\mu_2 < \mu_3$, i.e. $- \lambda_1 = |\lambda_1|= \mu_3 - \mu_2 > 0$
\begin{align}
  \no \lambda_1 \left\l \nabla^k (\dot{\dr} - \B \dr), \nabla^k \dr \right\r =& \lambda_1 \left\l \nabla^k \partial_t \dr, \nabla^k \dr \right\r + \lambda_1 \left\l \nabla^k (\v \cdot \nabla \dr), \nabla^k \dr \right\r \\
  \no = & - \frac{1}{2} |\lambda_1| \frac{\d}{\d t} |\nabla^k \dr|^2_{L^2} + \lambda_1 \left\l \nabla^k (\v \cdot \nabla \dr), \nabla^k \dr \right\r \, .
\end{align}
Then we have
  \begin{align}\label{Multiply-d-H^k-1}
    \no &\frac{1}{2} \frac{\d}{\d t} \left( \rho_1 |\nabla^k (\dot{\dr} + \dr)|^2_{L^2} +( |\lambda_1|-\rho_1) |\nabla^k \dr|^2_{L^2} - \rho_1 |\nabla^k \dot{\dr}|^2_{L^2} \right) - \rho_1 |\nabla^k \dot{\dr}|^2_{L^2} + |\nabla^{k+1} \dr|^2_{L^2} \\
    \no =& - \rho_1 \sum_{\substack{a+b=k \\ a \geq 1}} \left\langle \nabla^k \dot{\dr}, \nabla^a \v \nabla^{b+1} \dr \right\rangle - \rho_1 \sum_{\substack{a+b=k \\ a \geq 1}} \left\langle \nabla^a \v \nabla^{b+1} \dot{\dr}, \nabla^k \dr \right\rangle \\
    +& \left\langle \nabla^k(\gamma \dr), \nabla^k \dr \right\rangle + \lambda_1 \left\langle \nabla^k (\v \cdot \nabla \dr - \B \dr) , \nabla^k \dr \right\rangle + \lambda_2 \left\langle \nabla^k (\A \dr) , \nabla^k \dr \right\rangle \\
    \no \equiv& M_1 + M_2 + M_3 + M_4 + M_5\, ,
  \end{align}
  where $\gamma = - \rho_1 |\dot{\dr}|^2 + |\nabla \dr|^2  - \lambda_2 \dr^\top \A \dr\,$. By using the Sobolev embedding $W^{1,2} \hookrightarrow L^4$ and $W^{2,2} \hookrightarrow L^\infty$ for $n = 2,3$ and H\"older inequality, we have
  \begin{align}\label{Multiply-d-H^k-2}
    \no M_1 + M_2 \leq& \rho_1 \sum_{\substack{a+b=k \\ a \geq 1}} |\nabla^k \dot{\dr}|_{L^2} |\nabla^a \v|_{L^4} |\nabla^{b+1} \dr|_{L^4} \\
    \no &+ \rho_1 \sum_{\substack{a+b=k \\ a \geq 2}} |\nabla^k \dr|_{L^2} |\nabla^a \v|_{L^4} |\nabla^{b+1} \dot{\dr}|_{L^4} + \rho_1 |\nabla \v|_{L^\infty} |\nabla^k \dot{\dr}|_{L^2} |\nabla^k \dr|_{L^2} \\
    \leq& C \rho_1 |\nabla^k \dot{\dr}|_{L^2} |\nabla^{k+1} \v|_{L^2} |\nabla^{k+1} \dr|_{L^2} + C \rho_1 |\nabla^k \dot{\dr}|_{L^2} |\nabla^{k+1} \v|_{L^2} |\nabla^{k} \dr|_{L^2} \\
    \no & \qquad+ C \rho_1 |\nabla^3 \v|_{L^2} |\nabla^k \dot{\dr}|_{L^2} |\nabla^k \dr|_{L^2} \\
    \no \leq& C \rho_1 |\dot{\dr}|_{H^s} |\nabla \dr|_{H^s} |\nabla \v|_{H^s} \, ,
  \end{align}
  and
  \begin{align}\label{Multiply-d-H^k-3}
    \no M_4 + M_5 = & \lambda_1 \sum_{\substack{ a+b=k \\ a \geq 1 }} \left\l \nabla^a \v \nabla^{b+1} \dr , \nabla^k \dr \right\r + \left\l \lambda_1 \nabla^{k-1} (\B \dr) - \lambda_2 \nabla^{k-1} (\A \dr) , \nabla^{k+1} \dr \right\r \\
    \no \leq& |\lambda_1| \sum_{\substack{ a+b=k \\ a \geq 1 }} |\nabla^a \v|_{L^4} |\nabla^{b+1} \dr|_{L^2} |\nabla^k \dr|_{L^4} + ( |\lambda_1| + |\lambda_2| ) \sum_{a+b=k-1} \left\l |\nabla^{a+1} \v| |\nabla^b \dr|, |\nabla^{k+1} \dr| \right\r \\
    \no \leq& C |\lambda_1| \sum_{\substack{ a+b=k \\ a \geq 1 }} |\nabla^{a+1} \v|_{L^2} |\nabla^{b+1} \dr|_{L^2} |\nabla^{k+1} \dr|_{L^2} \\
    \no +& ( |\lambda_1| + |\lambda_2| ) \sum_{\substack{a+b=k-1 \\ b \geq 1}} |\nabla^{a+1} \v|_{L^4} |\nabla^b \dr|_{L^4} |\nabla^{k+1} \dr|_{L^2} + ( |\lambda_1| + |\lambda_2| ) |\nabla^k \v|_{L^2} |\nabla^{k+1} \dr|_{L^2} \\
     \leq& ( |\lambda_1| + |\lambda_2| ) |\nabla^k \v|_{L^2} |\nabla^{k+1} \dr|_{L^2} + C |\lambda_1| |\nabla \v|_{H^s} |\nabla \dr|_{\dot{H}^s} |\nabla \dr|_{H^s} \\
     \no +& C ( |\lambda_1| + |\lambda_2| ) |\nabla \v|_{H^s} |\nabla \dr|_{\dot{H}^s} |\nabla \dr|_{H^s} \\
     \no \leq& ( |\lambda_1| + |\lambda_2| ) |\nabla^k \v|_{L^2} |\nabla^{k+1} \dr|_{L^2} + C ( |\lambda_1| + |\lambda_2| ) |\nabla \v|_{H^s} |\nabla \dr|_{\dot{H}^s} |\nabla \dr|_{H^s} \, ,
  \end{align}
  and
  \begin{align}\label{Multiply-d-H^k-4}
    \no M_3 =& - \rho_1 \sum_{\substack{a+b+c=k}} \left\l \nabla^a \dot{\dr} \nabla^b \dot{\dr} \nabla^c \dr , \nabla^k \dr \right\r\\
      +& \sum_{a+b+c=k} \left\l \nabla^{a+1} \dr \nabla^{b+1} \dr \nabla^c \dr , \nabla^k \dr \right\r +  \lambda_2 \left\l \nabla^{k-1} [ (\dr^\top \A \dr) \dr ], \nabla^{k+1} \dr \right\r \\
     \no \equiv& M_{31} + M_{32} + M_{33}\ ,
  \end{align}
  where
  \begin{align}\label{Multiply-d-H^k-5}
    \no M_{31} =& - \rho_1 \sum_{\substack{ a+b+c=k \\ c \geq 1 }} \left\l \nabla^a \dot{\dr} \nabla^b \dot{\dr} \nabla^c \dr , \nabla^k \dr \right\r - \rho_1 \sum_{a+b=k} \left\l \nabla^a \dot{\dr} \nabla^b \dot{\dr} , \nabla^k \dr \right\r \\
     \no \leq& \rho_1 \sum_{\substack{ a+b+c=k \\ c \geq 1 }} |\nabla^a \dot{\dr}|_{L^4} |\nabla^b \dot{\dr}|_{L^4} |\nabla^c \dr|_{L^4} |\nabla^k \dr|_{L^4} \\
     \no +& 2 \rho_1 |\dot{\dr}|_{L^4} |\nabla^k \dot{\dr}|_{L^2} |\nabla^k \dr|_{L^4} + \rho_1 \sum_{\substack{ a+b=k \\ a,b \geq 1 }} |\nabla^a \dot{\dr}|_{L^2} |\nabla^b \dot{\dr}|_{L^4} |\nabla^k \dr|_{L^4} \\
     \leq& C \rho_1 \sum_{\substack{ a+b+c=k \\ c \geq 1 }} |\nabla^{a+1} \dot{\dr}|_{L^2} |\nabla^{b+1} \dot{\dr}|_{L^2} |\nabla^{c+1} \dr|_{L^2} |\nabla^{k+1} \dr|_{L^2} \\
     \no +& C \rho_1 | \nabla \dot{\dr}|_{L^2} |\nabla^{k} \dot{\dr}|_{L^2} |\nabla^{k+1} \dr|_{L^2} + C \rho_1 \sum_{\substack{ a+b=k \\ a,b \geq 1 }} |\nabla^a \dot{\dr}|_{L^2} |\nabla^{b+1} \dot{\dr}|_{L^2} |\nabla^{k+1} \dr|_{L^2} \\
     \no \leq& C \rho_1 |\nabla^k \dot{\dr}|^2_{L^2} |\nabla^{k+1} \dr|^2_{L^2} + C \rho_1 |\nabla^k \dot{\dr}|^2_{L^2} |\nabla^{k+1} \dr|_{L^2} \\
     \no \leq& C \rho_1 |\dot{\dr}|^2_{H^s} ( | \nabla \dr |^2_{H^s} + |\nabla \dr|_{H^s} ) \, ,
  \end{align}
  and
  \begin{align}\label{Multiply-d-H^k-6}
    \no M_{32} =& \sum_{a+b=k} \left\l \nabla^{a+1} \dr \nabla^{b+1} \dr \dr , \nabla^k \dr \right\r + \sum_{\substack{a+b+c=k \\ c \geq 1}} \left\l \nabla^{a+1} \dr \nabla^{b+1} \dr \nabla^c \dr , \nabla^k \dr \right\r \\
    \no \leq& 2 |\nabla \dr|_{L^4} |\nabla^{k+1} \dr|_{L^2} |\nabla^k \dr|_{L^4} + \ \sum_{\substack{a+b=k \\ a,b \geq 1}} |\nabla^{a+1} \dr|_{L^3} |\nabla^{b+1} \dr|_{L^3} |\nabla^k \dr|_{L^3} \\
    \no +& \sum_{\substack{a+b+c=k \\ c \geq 1}} |\nabla^{a+1} \dr|_{L^4} |\nabla^{b+1} \dr|_{L^4} |\nabla^c \dr|_{L^4} |\nabla^k \dr|_{L^4} \\
    \leq& C |\nabla^2 \dr|_{L^2} |\nabla^{k+1} \dr|^2_{L^2} + C \sum_{\substack{ a+b=k \\ a,b \geq 1 }} |\nabla^{a+2} \dr|_{L^2} |\nabla^{b+2} \dr|_{L^2} |\nabla^{k+1} \dr|_{L^2} \\
    \no +& C \sum_{\substack{a+b+c=k \\ c \geq 1}} |\nabla^{a+2} \dr|_{L^2} |\nabla^{b+2} \dr|_{L^2} |\nabla^{c+1} \dr|_{L^2} |\nabla^{k+1} \dr|_{L^2} \\
    \no \leq& C |\nabla^{k+1} \dr|^3_{L^2} + C |\nabla^{k+1} \dr|^4_{L^2} \\
    \no \leq& C ( |\nabla \dr|^3_{\dot{H}^s} + |\nabla \dr|^4_{\dot{H}^s} ) \, ,
  \end{align}
  and
  \begin{align}\label{Multiply-d-H^k-7}
     \no M_{33} \leq& |\lambda_2| \sum_{\substack{a+b+c+e=k-1 \\ a,b,c \geq 1}} \left\l |\nabla^a \dr| |\nabla^b \dr| |\nabla^c \dr| |\nabla^{e+1} \v|, |\nabla^{k+1} \dr| \right\r \\
    \no +& 3 |\lambda_2| \sum_{\substack{a+b+e=k-1 \\ a,b \geq 1}} \left\l |\nabla^a \dr| |\nabla^b \dr| |\nabla^{e+1} \v|, |\nabla^{k+1} \dr| \right\r \\
    \no +& 6 |\lambda_2| \sum_{\substack{a+e=k-1 \\ a \geq 1}} \left\l |\nabla^a \dr| |\nabla^{e+1} \v|, |\nabla^{k+1} \dr| \right\r + 6 |\lambda_2| \left\l |\nabla^k \v| , |\nabla^{k+1} \dr| \right\r \\
    \no \leq& \lambda_2| \sum_{\substack{a+b+c+e=k-1 \\ a,b,c \geq 1}}  |\nabla^a \dr|_{L^6} |\nabla^b \dr|_{L^6} |\nabla^c \dr|_{L^6} |\nabla^{e+1} \v|_{L^\infty} |\nabla^{k+1} \dr|_{L^2} \\
     \no+& 3 |\lambda_2| \sum_{\substack{a+b+e=k-1 \\ a,b \geq 1}} |\nabla^a \dr|_{L^6} |\nabla^b \dr|_{L^6} |\nabla^{e+1} \v|_{L^6} |\nabla^{k+1} \dr|_{L^2} \\
    \no +& 6 |\lambda_2| \sum_{\substack{a+e=k-1 \\ a \geq 1}} |\nabla^a \dr|_{L^4} |\nabla^{e+1} \v|_{L^4} |\nabla^{k+1} \dr|_{L^2} + 6 |\lambda_2| |\nabla^k \v|_{L^2} |\nabla^{k+1} \dr|_{L^2} \\
     \leq& C \lambda_2| \sum_{\substack{a+b+c+e=k-1 \\ a,b,c \geq 1}}  |\nabla^{a+1} \dr|_{L^2} |\nabla^{b+1} \dr|_{L^2} |\nabla^{c+1} \dr|_{L^2} |\nabla^{e+3} \v|_{L^2} |\nabla^{k+1} \dr|_{L^2} \\
    \no +& C |\lambda_2| \sum_{\substack{a+b+e=k-1 \\ a,b \geq 1}} |\nabla^{a+1} \dr|_{L^2} |\nabla^{b+1} \dr|_{L^2} |\nabla^{e+2} \v|_{L^2} |\nabla^{k+1} \dr|_{L^2} \\
    \no +& C |\lambda_2| \sum_{\substack{a+e=k-1 \\ a \geq 1}} |\nabla^{a+1} \dr|_{L^2} |\nabla^{e+2} \v|_{L^2} |\nabla^{k+1} \dr|_{L^2} + 6 |\lambda_2| |\nabla^k \v|_{L^2} |\nabla^{k+1} \dr|_{L^2} \\
    \no \leq& C |\lambda_2| |\nabla^{k+1} \v|_{L^2}  ( |\nabla^{k+1} \dr|^2_{L^2} + |\nabla^{k+1} \dr|^3_{L^2} + |\nabla^{k+1} \dr|^4_{L^2} ) + 6 |\lambda_2| |\nabla^k \v|_{L^2} |\nabla^{k+1} \dr|_{L^2} \\
    \no \leq& 6 |\lambda_2| |\nabla^k \v|_{L^2} |\nabla^{k+1} \dr|_{L^2} + C |\lambda_2| |\nabla \v|_{H^s} |\nabla \dr|_{\dot{H}^s} ( |\nabla \dr|_{H^s} + |\nabla \dr|^2_{H^s} + |\nabla \dr|^3_{H^s} ) \, .
  \end{align}

  It is immediately implied that by substituting \eqref{Multiply-d-H^k-2}, \eqref{Multiply-d-H^k-3}, \eqref{Multiply-d-H^k-4}, \eqref{Multiply-d-H^k-5}, \eqref{Multiply-d-H^k-6} and \eqref{Multiply-d-H^k-7} into \eqref{Multiply-d-H^k-1}
  \begin{align}
    \no & \frac{1}{2} \frac{\d}{\d t} \left( \rho_1 |\nabla^k (\dot{\dr} + \dr)|^2_{L^2} +( |\lambda_1|-\rho_1) |\nabla^k \dr|^2_{L^2} - \rho_1 |\nabla^k \dot{\dr}|^2_{L^2} \right) - \rho_1 |\nabla^k \dot{\dr}|^2_{L^2} + |\nabla^{k+1} \dr|^2_{L^2} \\
    \no\leq& ( |\lambda_1| - 7 \lambda_2 ) |\nabla^k \v|_{L^2} |\nabla^{k+1} \dr|_{L^2} + C \rho_1 |\dot{\dr}|_{H^s} |\nabla \v|_{H^s} |\nabla \dr|_{H^s} \\
    \no +& C (1+ \rho_1 ) ( |\dot{\dr}|^2_{H^s} + |\nabla \dr|^2_{\dot{H}^s} ) ( |\nabla \dr|_{H^s} + |\nabla \dr|^2_{H^s} ) \\
    \no +& C ( |\lambda_1| -\lambda_2 ) |\nabla \v|_{H^s} |\nabla \dr|_{\dot{H}^s} ( |\nabla \dr|_{H^s} + |\nabla \dr|^2_{H^s} + |\nabla \dr|^3_{H^s} )
  \end{align}
  for all $1 \leq k \leq s$. Then it is derived that by summing up the above inequalities for $1 \leq k \leq s$
  \begin{align}\label{Multiply-d-H^s}
    \no & \frac{1}{2} \frac{\d}{\d t} \left( \rho_1 |\dot{\dr} + \dr|^2_{\dot{H}^s} +( |\lambda_1|-\rho_1 ) | \dr|^2_{\dot{H}^s} - \rho_1 | \dot{\dr}|^2_{\dot{H}^s} \right) - \rho_1 | \dot{\dr}|^2_{\dot{H}^s} + |\nabla \dr|^2_{\dot{H}^s} \\
    \no\leq& ( |\lambda_1| - 7 \lambda_2 ) |\nabla \v|_{H^s} |\nabla \dr|_{\dot{H}^s} + C \rho_1 |\dot{\dr}|_{H^s} |\nabla \v|_{H^s} |\nabla \dr|_{H^s} \\
     +& C (1+ \rho_1 ) ( |\dot{\dr}|^2_{H^s} + |\nabla \dr|^2_{\dot{H}^s} ) ( |\nabla \dr|_{H^s} + |\nabla \dr|^2_{H^s} ) \\
    \no +& C ( |\lambda_1| -\lambda_2 ) |\nabla \v|_{H^s} |\nabla \dr|_{\dot{H}^s} ( |\nabla \dr|_{H^s} + |\nabla \dr|^2_{H^s} + |\nabla \dr|^3_{H^s} ) \, .
  \end{align}

  Taking a positive constant $\eta = \frac{1}{2} \min \big{\{} 1 , \frac{1}{\rho_1} , \frac{ |\lambda_1|}{\rho_1} \big{\}} \in (0,\frac{1}{2} ]$, we multiply by $\eta$ in the inequality \eqref{Multiply-d-H^s} and then add it to the inequality \eqref{H^s-Estimate-global} (dropped the index $\eps$), so that we know that
  \begin{align}\label{Global-Bounds-H^s}
     \no & \frac{1}{2} \frac{\d}{\d t} \left( |\v|^2_{H^s} + \rho_1 (1-\eta) |\dot{\dr}|^2_{H^s} + (1-\eta \rho_1 ) |\nabla \dr|^2_{H^s} \right. \\
     \no &\qquad+\left. \eta \rho_1 |\dot{\dr} + \dr|^2_{\dot{H}^2} + \eta \rho_1 |\nabla^{s+1} \dr|^2_{L^2} + \eta \rho_1 |\lambda_1| |\dr|^2_{\dot{H}^s} \right) \\
     \no +& \frac{1}{2} \mu_4 |\nabla \v|^2_{H^s} + ( |\lambda_1| - \eta \rho_1 ) |\dot{\dr}|^2_{H^s} +  \eta |\nabla \dr|^2_{\dot{H}^s} + \eta \rho_1 |\dot{\dr}|^2_{L^2} + \mu_1 \sum_{k=0}^s |\dr^\top  (\nabla^{k+1} \v) \dr|^2_{L^2} \\
      \leq& 2 \mu_6 |\nabla \v|^2_{H^s} + (|\lambda_1| - 7 \lambda_2) |\nabla \v|_{H^s} |\nabla \dr|_{\dot{H}^s} + ( 7 |\lambda_1| - 2 \lambda_2 ) |\nabla \v|_{H^s} |\dot{\dr}|_{H^s} \\
     \no +& C' ( 1 + \mu_1 + |\lambda_1| - \lambda_2 + \mu_6 - \rho_1 \lambda_2 + \rho_1 + \tfrac{1}{\sqrt{\rho_1}} ) \Big{(}  |\v|_{H^s} + |\dot{\dr}|_{H^s} + \sum_{i=1}^4 |\nabla \dr|^i_{H^s} \Big{)} \\
      \no & \quad \times \big{(} |\nabla \v|_{H^s} + |\dot{\dr}|_{H^s} + |\nabla \dr|_{\dot{H}^s} \big{)} |\nabla \v|_{H^s} \, ,
  \end{align}
  where the constant $C' = C' (n,s) > 0$.

  Notice that $|\lambda_1| - \eta \rho_1 \geq \frac{1}{2} |\lambda_1|$ and
  \begin{align}
    \no ( |\lambda_1| - 7 \lambda_2 ) |\nabla \v|_{H^s} |\nabla \dr|_{\dot{H}^s} \leq& \frac{1}{2} \eta |\nabla \dr|^2_{\dot{H}^s} + \tfrac{ (|\lambda_1| - 7 \lambda_2)^2 }{2 \eta} |\nabla \v|^2_{H^s} \, , \\
    \no ( 7 |\lambda_1| - 2 \lambda_2 ) |\nabla \v|_{H^s} |\dot{\dr}|_{H^s} \leq& \frac{1}{4} |\lambda_1| |\dot{\dr}|^2_{H^s} + \tfrac{ ( 7 |\lambda_1| - 2\lambda_2 )^2 }{ |\lambda_1| } |\nabla \v|^2_{H^s} \, .
  \end{align}
  Then the inequality \eqref{Global-Bounds-H^s} reduces to
  \begin{align}
     \no & \frac{1}{2} \frac{\d}{\d t} \left( |\v|^2_{H^s} + \rho_1 (1-\eta) |\dot{\dr}|^2_{H^s} + (1-\eta \rho_1 ) |\nabla \dr|^2_{H^s} \right. \\
     \no &\qquad+\left. \eta \rho_1 |\dot{\dr} + \dr|^2_{\dot{H}^2} + \eta \rho_1 |\nabla^{s+1} \dr|^2_{L^2} + \eta \rho_1 |\lambda_1| |\dr|^2_{\dot{H}^s} \right) \\
     \no +& \frac{1}{2} \alpha |\nabla \v|^2_{H^s} + ( |\lambda_1| - \eta \rho_1 ) |\dot{\dr}|^2_{H^s} +  \eta |\nabla \dr|^2_{\dot{H}^s}  \\
      \no \leq&  C' ( 1 + \mu_1 + |\lambda_1| - \lambda_2 + \mu_6 - \rho_1 \lambda_2 + \rho_1 + \tfrac{1}{\sqrt{\rho_1}} ) \Big{(}  |\v|_{H^s} + |\dot{\dr}|_{H^s} + \sum_{i=1}^4 |\nabla \dr|^i_{H^s} \Big{)} \\
      \no & \quad \times \big{(} |\nabla \v|_{H^s} + |\dot{\dr}|_{H^s} + |\nabla \dr|_{\dot{H}^s} \big{)} |\nabla \v|_{H^s} \, ,
  \end{align}
  where $\alpha = \mu_4 - 4 \mu_6 - \tfrac{ (|\lambda_1| - 7 \lambda_2)^2 }{\eta} - \tfrac{ 2 ( 7 |\lambda_1| - 2\lambda_2 )^2 }{ |\lambda_1| } > 0 $. We denote by $$\mathcal{E}(t) \equiv |\v|^2_{H^s} + \rho_1 (1-\eta) |\dot{\dr}|^2_{H^s} + (1-\eta \rho_1 ) |\nabla \dr|^2_{H^s} + \eta \rho_1 |\dot{\dr} + \dr|^2_{\dot{H}^s} + \eta \rho_1 |\nabla^{s+1} \dr|^2_{L^2} + \eta \rho_1 |\lambda_1| |\dr|^2_{\dot{H}^s}$$ and $$ \mathcal{D} (t) \equiv  |\nabla \v|^2_{H^s} +   |\dot{\dr}|^2_{H^s} +  |\nabla \dr|^2_{\dot{H}^s}\, ,$$
  then the following inequality holds :
  \begin{equation}\label{Global-Bounds}
    \frac{\d}{\d t} \mathcal{E}(t) + \theta \mathcal{D}(t) \leq C_3 \sum_{q=1}^4 [\mathcal{E} (t) ]^\frac{q}{2} \mathcal{D} (t)\, ,
  \end{equation}
  where the constant $ C_3 = 4 C' \big{(} 1 + \frac{1}{\sqrt{\rho_1}} \big{)}  \big{(}1+\mu_1 + |\lambda_1| - \lambda_2 + \mu_6 - \rho_1 \lambda_2 + \rho_1 + \frac{1}{\sqrt{\rho_1}} \big{)} > 0 $ and $\theta = \min \big{\{} \alpha, \eta , \frac{1}{2} |\lambda_1| \big{\}} > 0$.

  If $E^{in} \leq \frac{1}{|\lambda_1| + 2} \min \{ 1, \frac{\theta^2}{(8 C_3)^2} \}$, then we have
  \begin{align}
    \no \mathcal{E}(0) =& |\v^{in}|^2_{H^s} + \rho_1 (1-\eta) |{\tilde\dr}^{in}|^2_{H^s} + (1-\eta \rho_1 ) |\nabla \dr^{in}|^2_{H^s} \\
    \no &+ \eta \rho_1 |{\tilde\dr}^{in} + \dr^{in}|^2_{\dot{H}^s} + \eta \rho_1 |\nabla^{s+1} \dr^{in}|^2_{L^2} + \eta \rho_1 |\lambda_1| |\dr^{in}|^2_{\dot{H}^s}\\
    \no \leq& |\v^{in}|^2_{H^s} + \rho_1 (1-\eta) |{\tilde\dr}^{in}|^2_{H^s} + (1-\eta \rho_1 ) |\nabla \dr^{in}|^2_{H^s} \\
    \no &+ 2 \eta \rho_1 |{\tilde\dr}^{in}|^2_{H^s} + 2 \eta \rho_1 |\nabla \dr^{in}|^2_{H^s} + \eta \rho_1 |\nabla \dr^{in}|^2_{H^s} + \eta \rho_1 |\lambda_1| |\nabla \dr^{in}|^2_{H^s}\\
    \no \leq& |\v^{in}|^2_{H^s} + \tfrac{3}{2} \rho_1 |{\tilde\dr}^{in}|^2_{H^s} + ( \tfrac{1}{2} |\lambda_1| + 2 ) |\nabla \dr^{in}|^2_{H^s} \\
    \no \leq& ( |\lambda_1|  + 2 ) ( |\v^{in}|^2_{H^s} +  \rho_1 |\dot{\dr}|^2_{H^s} +  |\nabla \dr^{in}|^2_{H^s} ) \\
    \no =& ( |\lambda_1| + 2 ) E^{in} \leq 1\, ,
  \end{align}
  which implies that
  \begin{equation}\no
    C_3 \sum_{q =1}^4 [ \mathcal{E} (0) ]^\frac{q}{2} \leq 4 C_3 \sqrt{\mathcal{E}(0)} \leq 4 C_3 \sqrt{ (|\lambda_1| + 2) E^{in} } \leq \frac{1}{2} \theta \, .
  \end{equation}

  We define $T^* = \sup \big{\{} T > 0; C_3 \sum\limits_{q=1}^4 [\mathcal{E} (t) ]^\frac{q}{2} \leq \theta \ \textrm{ holds\ for\ all} \ t \in [0,T] \big{\}} $. By the continuity of $\mathcal{E}(t)$ we know that $T^* > 0$. Then for all $t \in [0,T^*]$
  $$ \frac{\d}{\d t} \mathcal{E}(t) + \left( \theta - C_3 \sum\limits_{q=1}^4 [\mathcal{E} (t) ]^\frac{q}{2} \right) |\nabla \v|^2_{H^s} (t) \leq \frac{\d}{\d t} \mathcal{E}(t) + \left( \theta - C_3 \sum\limits_{q=1}^4 [\mathcal{E} (t) ]^\frac{q}{2} \right) \mathcal{D} (t) \leq 0 \, , $$
  which implies that $ \mathcal{E}(t) \leq \mathcal{E}(0) \leq 1  $ for all $t \in [0,T^*]$ and consequently
  $$ C_3 \sum\limits_{q=1}^4 [\mathcal{E} (t) ]^\frac{q}{2} \leq C_3 \sum\limits_{q=1}^4 [\mathcal{E} (0) ]^\frac{q}{2} \leq \frac{1}{2} \theta < \theta \, . $$
  By the definition of $T^*$ we know that $T^* = + \infty$ and
  \begin{equation}\no
    \sup_{ t \geq 0} \mathcal{E}(t) + \frac{1}{2} \int_0^{\infty} |\nabla \v|^2_{H^s}(t) \d t \leq \mathcal{E}(0) \leq (|\lambda_1|+2) E^{in}\, .
  \end{equation}

  Noticing that $1- \eta \geq \frac{1}{2}$, $1 - \eta \rho_1 \geq \frac{1}{2} \rho_1$ and
  \begin{align}
    \no \mathcal{E}(t) =& |\v|^2_{H^s} + \rho_1(1-\eta) |\dot{\dr}|^2_{H^s} + (1-\eta \rho_1 ) |\nabla \dr|^2_{H^s} \\
    \no& + \eta \rho_1 |\dot{\dr} + \dr|^2_{\dot{H}^s} + \eta \rho_1 |\nabla^{s+1} \dr|^2_{L^2} + \eta \rho_1 |\lambda_1| |\dr|^2_{\dot{H}^s} \\
    \no \geq& |\v|^2_{H^s} + \rho_1 (1-\eta) |\dot{\dr}|^2_{H^s} + (1-\eta \rho_1 ) |\nabla \dr|^2_{H^s} \\
    \no \geq& |\v|^2_{H^s} + \frac{1}{2} \rho_1 |\dot{\dr}|^2_{H^s} + \frac{1}{2} |\nabla \dr|^2_{H^s} \\
    \no \geq& \frac{1}{2} ( |\v|^2_{H^s} + \rho_1 |\dot{\dr}|^2_{H^s} +  |\nabla \dr|^2_{H^s} ) \, ,
  \end{align}
  we obtain the following inequality :
\begin{align}
  \no \sup_{ t \geq 0} \left( |\v|^2_{H^s} + \rho_1 |\dot{\dr}|^2_{H^s} +  |\nabla \dr|^2_{H^s} \right) +  \int_0^{\infty} |\nabla \v|^2_{H^s} \d t \leq 2 (|\lambda_1|+2) E^{in}  \, .
 \end{align}
 If we want to extend the solution $(\v,\dr)$ constructed in Part (I) of Theorem \ref{Main-Thm} to $[0, \infty)$, the condition $2 (|\lambda_1|+2) E^{in} \leq \eps_0$ is required. So there exists an $\eps_1 = \frac{1}{|\lambda_1| + 2} \min \big{\{} \frac{1}{2} \eps_0, \frac{\theta^2}{(8 C_3)^2} \big{\}} > 0$ such that if $E^{in} \leq \eps_1$, then there is a unique solution $(\v, \dr)$ to the system \eqref{Parabolic-Hyperbolic-Liquid-Crystal-Model}. Then the proof of Part (III) of Theorem \ref{Main-Thm} is finished.


\section*{Acknowledgement}
We first appreciate Prof. Fanghua Lin, who suggested the Ericksen-Leslie's parabolic-hyperbolic liquid crystal model to us when we visited  the NYU-ECNU Institute of Mathematical Sciences at NYU Shanghai during the spring semester of 2015. We also thanks the hospitality of host institute. We also thanks for the conversations with Zhifei Zhang and Wei Wang, in particular on their work \cite{Wang-Zhang-Zhang-ARMA2013}, which shares light on the current work. During the preparation of this paper, Xu Zhang made several valuable comments and suggestions, we take this opportunity to thank him here.

\bigskip



\end{document}